\documentclass[11pt]{aims}
\usepackage{amsmath}
\usepackage{mathabx}
\usepackage{amsmath, amssymb, mathrsfs}
\usepackage{bbm}
\usepackage{paralist} 
\usepackage{graphics} %% add this and next lines if pictures should be in esp format
\usepackage{epsfig} %For pictures: screened artwork should be set up with an 85 or 100 line screen
\usepackage{graphicx} 
\usepackage{epstopdf}%This is to transfer .eps figure to .pdf figure; please — compile your paper using PDFLeTex or PDFTeXify.
\usepackage[colorlinks=true]{hyperref}
\hypersetup{urlcolor=blue, citecolor=blue}
\usepackage[toc,page]{appendix}
\usepackage{chngcntr}
\usepackage{hyperref}

\usepackage{titlesec}
\titleformat{\section}{\Large\bfseries}{\thesection.}{4pt}{}
\titleformat{\subsection}{\large\bfseries}{\thesection.\arabic{subsection}.}{4pt}{}
\titleformat{\subsubsection}{\bfseries}{\thesection.\arabic{subsection}.\arabic{subsubsection}.}{4pt}{}
\titleformat*{\paragraph}{\bfseries}
\titleformat*{\subparagraph}{\bfseries}
\usepackage[margin=1in]{geometry}

\def\currentyear{}
\def\currentmonth{}
\def\ppages{}
\def\DOI{}

\newtheorem{theorem}{Theorem}

\usepackage[pagewise]{lineno}
\newtheorem{prop}{Proposition}[section]
\newtheorem{lemma}[prop]{Lemma}
\newtheorem{definition}[prop]{Definition}
\newtheorem{cl}[prop]{Claim}

\newtheorem{rem}[prop]{Remark}

\def\qed{\hbox {\hskip 1pt \vrule width 4pt height 6pt depth 1.5pt
		\hskip 1pt}\\} % cqfd

\def\T0{T_{0,1}}

\def\sf{s_{0,3}}

\def\sc {s_{0,1}}

\def\tt {{\varrho}}

\newcommand {\R}{ \mathbb{R}}
\newcommand {\C}{ \mathbb{C}}
\newcommand {\N}{ \mathbb{N}}

\newcommand {\pa}{\partial}

\newcommand {\beqna} {\begin{eqnarray}}
	\newcommand {\eeqna} {\end{eqnarray}}
\newcommand {\beqtn} {\begin{equation}}
	\newcommand {\eeqtn} {\end{equation}}

\newcommand {\dsp}{\displaystyle}

\def\and {{\rm \; and \;}}

\def\exp {{\rm exp}}

\numberwithin{equation}{section}

\title{Flat blow-up solutions for the complex Ginzburg Landau equation}

\author[G. K. Duong,  N. Nouaili and H. Zaag ]{}


\thanks{\today}
\begin{document}
	\maketitle
	
	% Enter the first author's name and address:
	
	\centerline{Giao Ky Duong$^{(1)}$,   Nejla Nouaili$^{(2)}$   and Hatem Zaag$^{(3)}$} 
	\medskip
	{\footnotesize
		\centerline{ $^{(1)}$ Institute of Applied Mathematics, University of Economics Ho Chi Minh City,  Vietnam}
		\centerline{ $^{(2)}$ CEREMADE, Universit\'e Paris Dauphine, Paris Sciences et Lettres, France }
		\centerline{ $^{(3)}$Universit\'e Sorbonne Paris Nord,
			LAGA, CNRS (UMR 7539), F-93430, Villetaneuse, France}
	}

	\bigskip
	\begin{center}\thanks{\today}\end{center}

	\begin{abstract}	
  In this paper, we consider the complex Ginzburg-Landau equation
$$ \partial_t u = (1 + i \beta) \Delta u + (1 + i \delta) |u|^{p-1}u - \alpha u, \quad \text{where } \beta, \delta, \alpha \in \mathbb{R}. $$
The study focuses on investigating the finite-time blow-up phenomenon, which remains an open question for a broad range of parameters, particularly for \(\beta\) and \(\delta\). Specifically, for a fixed \(\beta \in \mathbb{R}\), the existence of finite-time blow-up solutions for arbitrarily large values of \( |\delta| \) is still unknown. According to a conjecture made by Popp et al. \cite{POPphd98}, when \(\beta = 0\) and \(\delta\) is large, blow-up does not occur for \textit{generic initial data}. 

In this paper, we show that their conjecture is not valid for all types of  initial data, by presenting the existence of blow-up solutions for \(\beta = 0\) and any \(\delta \in \mathbb{R}\) with different types of blowup.
	\end{abstract}

	\maketitle

	\section{Introduction}
		In this paper, we  study  the complex Ginzburg-Landau   equation
	\beqtn
	\left\{ \begin{array}{rcl}
		u_t &=&(1+i\beta)\Delta u+(1+i\delta) |u|^{p-1}u -\alpha u,\\[0.2cm]
		u(.,0) & = & u_0\in L^\infty (\R^N,\C),
	\end{array} \right.
	\mbox{     CGL}
	\label{GL}
	\eeqtn
	where $\beta, \delta $ and $\alpha $   are real numbers, $p>1$ and $u(\cdot, t): \R^N \to u(x,t) \in \C$, and $u_0$ is a given initial condition. 

	\medskip

	The CGL  is named after  V. Ginzburg and L. Landau,   models  various physical situations. Significantly, the cubic case, i.e., $p=3$ is one  of the most-studied  nonlinear equations in  physics. Specifically, it   describes    a wide range  of phenomena, including  superconductivity, superfluidity, nonlinear waves, second-order phase transitions, Bose-Einstein condensation, liquid crystals and strings in field theory. For more details on the physical background, we refer to   \cite{AKRMP2002}, \cite{HSSBJFM72} \cite{KBSPRL88} \cite{KSALPNP95}  
	\cite{PSKKPD98}, \cite{SSJFM71}, \cite{TPRE93}, and the references therein. 
	
	\medskip 
	
	We also note that the CGL equation shares  fundamental features with the Navier-Stokes equation, making it a valuable  model for studying  the singular behavior of the  Navier Stokes.	In fact, the CGL equation can be derived from the Navier-Stokes equations using multiple-scale methods in several problems, particularly in convection (see, for example  \cite{NWJFD69}).

 \medskip 
	The local Cauchy problem can be solved in various functional spaces, using   semigroup theory,  as is done  with the nonlinear heat equation (see \cite{CNYUCIM03,GV96,GVCMP97}).  In particular, for  initial data $u_0\in L^\infty(\R^N)$,  there exists a time $T>0$ and a unique solution $u\in C([0,T]; L^\infty(\R^N,\C))$ to \eqref{GL}. 

 \noindent 
	Either the solution exists globally on $[0,\infty)$, or the maximal existence time is finite, i.e., $T=T_{max} <+\infty$ in the following sense: 
	\begin{equation}\label{finite-time-blowup}
		\lim_{t\to T}\|u(\cdot,t)\|_{L^\infty(\R^N)}=+\infty.
	\end{equation}
	In the latter case, the solution  blows up in finite time. 
	Furthermore,  a point $x_0\in\R^N$ is called   a blow-up point  if there exists  a sequence $\{(x_j,t_j)\}$ such that $x_j\to x_0$, $t_j\to T$, and $|u(x_j,t_j)|\to \infty$ as $j\to\infty$.

	\medskip
	An extensive body of literature is devoted to the study of singularity formation  for equation \eqref{GL}. For instance,  \cite{SSJFM71} provides a description of an unstable plane Poiseuille flow. Moreover,  \cite{HSSBJFM72} (see also in \cite{KBSPRL88,KSALPNP95}), where the authors describe a series of experiments on traveling wave convection in an ethanol/water mixture, including  observation of collapse solutions that appear experimentally. Additionally,  \cite{TPRE93} establishes  sharp  criteria for wave collapse in \eqref{GL}  in the case of subcritical bifurcation.\\

	\medskip 
	In the present paper, we are interested in studying the finite time blowup phenomenon  to the CGL equation. More precisely, we will focus on the problem of construction of  finite time blowup solutions  to the CGL equation in the sense that 
	
	\[\lim_{t\to T }\|u(t)\|_{L^\infty(\R^N)}=+\infty.\]

	Let us first focus on the case where \( \beta = \delta = 0 \), which simplifies equation \eqref{GL} to the classical nonlinear heat equation:
\begin{equation}\label{equa-classical-heat}
		\partial_t u =  \Delta u + |u|^{p-1}u, \text{ and } p > 1.
	\end{equation}
In fact, a vast body of research has been devoted to studying blowup phenomena for equation \eqref{GL}.
  We refer  to \cite{FKcpam92},  \cite{HVcpde92}, \cite{HVDIE92},  \cite{MZgfa98},  and \cite{HVaihn93} where the authors investigated a classification on the blowup behavior of the positive solutions to \eqref{equa-classical-heat} in the dimension one. According to those papers,  let $u$ be a positive blowup solution  to \eqref{equa-classical-heat} blowing up at a point $a \in \R$, and  at time $T>0$, then  
	\medskip
	\begin{itemize}
		\item either
		\begin{equation}
			u(t) = \kappa (T-t)^{-\frac{1}{p-1}}, \forall t \in [0,T), \text{ with } \kappa = (p-1)^{-\frac{1}{p-1}},
		\end{equation}
		\item or
		\beqtn
		\sup_{|x-a|\leq K\sqrt{(T-t)\ln(T-t)}}\left|(T-t)^{\frac{1}{p-1}}u(x,t)-f_{0}\left(\frac{x-a}{\sqrt{(T-t)|\ln(T-t)|}}\right)\right|\to 0,
		\label{u2}
		\eeqtn
		where $ 	f_{0}(z)=\left(p-1+ \frac{(p-1)^2}{4p}z^2\right)^{-\frac{1}{p-1}}
		,$
		\item or 
		\beqtn
		\sup_{|x-a|<K(T-t)^{1/2k}}\left|(T-t)^{\frac{1}{p-1}}u(x,t)-f_b\left(\frac{ (x-a)}{(T-t)^{1/2k}}\right)\right|\to 0,
		\label{um}
		\eeqtn
		as $t\to T$, for some $k\in \N$, $k\ge 2$, and $b>0$,  where $f_b(z) =\left(p-1+b|z|^{2k}\right)^{-\frac{1}{p-1}}$.
	\end{itemize}
	
	\medskip
	The asymptotic \eqref{u2}  is generic; this  was proved in \cite{HVcpde92} and \cite{HVasnsp92}  in one space dimension. In higher dimensions, Herrero and Vel\'{a}zquez gave a proof, but they never published it.  We also cite \cite{VTAMS93} and \cite{FLaihn93} for the classification of the behavior of blowup solutions to \eqref{equa-classical-heat} in higher dimensions. In particular, 
	we mention  \cite{BKnon94} (see also \cite{MZdm97})  where
	the authors constructed blowup solutions that  obey the asymptotic  \eqref{u2}.  Additionally, the asymptotics  \eqref{um} are unstable. We also cite the results of \cite{HVDIE92} (for $k=4$), \cite{BKnon94} and \cite{ADIE95} for the construction of blowup solutions corresponding to the behavior \eqref{um}. In particular, we  recommend  the book \cite{QPbook19} for a more comprehensive understanding on the blowup phenomenon and related problems.    
	
	\medskip
	In the general case, the  blowup phenomena in the CGL equation is much more complicated to investigate. The major difficulty is that it is non-variational. Therefore,   classical methods based on energy-type estimates generally  break down, except the results in   \cite{CDWSIAM13} and \cite{CDFJEE14} which show the existence of blowup solution to the  CGL equation in the case $\beta=\delta$. Due to the lack of  energy-type estimates, there are no classification results  of the type \eqref{u2}-\eqref{um} in the general case. Even though the construction problem remains open in general, in  \cite{BRWSIAM05} and \cite{PSCPAM01}, the authors provided  evidence for the existence of a radial solution which blows up in a self-similar way.  Let us now focus on  the works of Zaag \cite{ZAAihn98}  and Masmoudi and Zaag \cite{MZjfa08}, they showed that if 
	\begin{equation} \delta^2 + \beta\delta (p+1)<p,\text{ the so-called subcritical case},\label{subcritical}
	\end{equation}
	then there exists a solution of equation \eqref{GL}, which blows up in finite time $T>0$, only at the origin,  with the   following blowup profiles
	\begin{eqnarray}
		\dsp\left\|  \bar \psi_{\rm sub}(t) u(\cdot ,t)-\left(p-1+\frac{b_{sub}|\cdot |^2}{(T-t)|\ln(T-t)|}\right)^{-\frac{1+i\delta}{p-1}}
		\right\|_{L^\infty}   \leq \frac{C}{1+\sqrt{|\ln(T-t)|}}, \label{profilesubc} 
	\end{eqnarray}
	where
	\begin{equation*}
		\bar \psi_{sub}(t) = (T-t)^{\frac{1+i\delta}{p-1}}\left |\ln(T-t)\right |^{-i\mu},
	\end{equation*}
	and 
	\beqtn
	\dsp b_{sub}=\frac{(p-1)^2}{4(p-\delta^2-\beta\delta(1+ p))}>0\mbox{ and }\mu=-\frac{2b_{sub}\beta}{(p-1)^2}(1+\delta^2).
	\label{bs}
	\eeqtn
	It is worth noting that this result was  formally  obtained  in \cite{HSPRSLS72} ($p=3$) and  later in  \cite{POPphd98}.

	\medskip 

In the critical case,  there are also results concerning   the construction of  blowup solutions obeying the generic behavior \eqref{u2}, which were investigated   by Nouaili and Zaag \cite{NZ2017} and Duong, Nouaili and
Zaag \cite{DNZMAMS20}. More specifically, when
\[\delta^2 + \beta\delta(p+1) -p=0,\]	the authors constructed the blowup solutions which satisfy
	\begin{equation*}
		\dsp\left\|  \bar \psi_{cri}(t) u(\cdot,t)-\left(p-1+\frac{b_{cri}|\cdot|^2}{(T-t)|\ln(T-t)|^{\frac 12}}\right)^{-\frac{1+i\delta}{p-1}}
		\right\|_{L^\infty}      \leq  \frac{C}{1+|\ln(T-t)|^{\frac 14}},%\label{profilesubc} 
	\end{equation*}
	where
	\begin{equation*}
		\bar \psi_{cri}(t) = (T-t)^{\frac{1+i\delta}{p-1}}\left |\ln(T-t)\right |^{-i\mu}e^{-i\nu \sqrt {T-t}}, 
	\end{equation*}
	with the constants $\nu= \nu(\beta, p), \mu = \mu(\beta, p) $ as determined in  \cite{DNZMAMS20}, and 
	\[b_{cri}^2  =  \frac{(p-1)^4 (p+1)^2 \delta^2}{  16 (1 + \delta^2) (p(2p-1) - (p-2)\delta^2)( (p+3)\delta^2 + p(3p+1)) }.\]

	\medskip
	 
The question of the existence of blowup solutions for equation \eqref{GL} in the supercritical cases (i.e., $\delta^2 +\beta\delta (p+1) >p$) remain open. Specifically, in \cite{POPphd98}, the authors  made  a formal conjecture stating that, for a fixed $\beta$, the existence of finite time blowup solutions with generic behavior (i.e., \eqref{u2}),  for arbitrarily large $|\delta|$, is not possible (we refer to Remark \ref{remark-Popp-conjec} for further details). In the following, we construct  new blowup solutions for the CGL equation in dimension $N=1$, for  the supercritical range of parameters. More precisely, our model, when $\beta =0$ and for any $\delta\in \R$, is expressed as
\beqtn\label{equa-CGL}
	\pa_t u=\Delta u+(1+i\delta )|u|^{p-1}u,
	\eeqtn
and our result reads.
\begin{theorem} \label{Theorem-principal} Consider the equation \eqref{equa-CGL} in one dimension.  
 Let $\delta \in \R$, $p>1, k \in \mathbb N, k \ge 2$ and $b_0 >0$. Then,  there exist $\gamma_0 =\gamma_0(\delta, p,k,b_0)>0$ and  $T_0=(\delta, p,k,b_0) >0$ such that  for all  $T \in (0,T_0)$, there exist  initial data $u_{0,T} \in L^\infty(\R)$ such that the corresponding solution blows up in finite time $T$ and only at the origin. Furthermore, there exists a constant  $b^* =b^*(u_{0,T}) >0$  such that \begin{equation}\label{theorem-intermediate}
		\left \| (T-t)^{\frac{1+i\delta}{p-1}} u(\cdot, t) - f_{b^*}\left(\frac{|\cdot|}{(T-t)^\frac{1}{2k}} \right)\right \|_{L^\infty(\mathbb{R})}\lesssim  (T-t)^{\frac{\gamma_0}{2}(1-\frac 1k)}, \forall t \in [0,T),
	\end{equation}
	where $f_{b^*}$ is defined by 
	\begin{equation}
		f_{ b^*}(y) = \left(p-1+ b^* |y|^{2k}\right)^{-\frac{1+i\delta}{p-1}},\;\;
	\end{equation}
 and \\
	\begin{equation}\label{estimate-b-t-b-*}
		\left|  b^* - b_0  \right| \lesssim   T^{\frac{\gamma_0}{2}(1-\frac 1k)}.
	\end{equation}

\end{theorem}

%%%%%%%OUT
%\textcolor{red}{Ky: It is better to be removed the corollary\\
	%	\begin{coro}\label{corollar}
%	Under the same hypothesis of Theorem \ref{Theorem-principal}, it holds that
%	\begin{equation}\label{convergence-corollar}
%		\left \| (T-t)^{\frac{1+i\delta}{p-1}} u(\cdot, t) - f_{b^*}\left(\frac{|\cdot|}{(T-t)^\frac{1}{2k}} \right)\right \|_{L^\infty(\mathbb{R})}\lesssim  (T-t)^{\frac{\gamma}{2}(1-\frac{1}{k})} \text{ as } t \to T,
%	\end{equation}
%	where $b^*$ is  defined in ii) of Theorem \ref{Theorem-principal}.
%\end{coro}}
%%%%%%%%%%OUT
\bigskip
\begin{rem} We emphasize that the results proved in Theorem \ref{Theorem-principal} remain valid for all \(\alpha \in \mathbb{R}\). In fact, by using the similarity variables in \eqref{change-variable2}, the term involving \(\alpha\) becomes exponentially decaying in the variables \((y, s)\), allowing us to handle it using perturbative methods. Additionally, we have only proven Theorem \ref{Theorem-principal} in one dimension. We believe that the result can likely be extended to higher dimensions (\(N \geq 2\)), although the proof would involve  more complex technical challenges.
\end{rem}

%\bigskip

\begin{rem}[Comments on the result]\label{remark-Popp-conjec}
Kuznetsov and co-authors, in several publications \cite{KKSEL94}, \cite{KKSPR94}, \cite{KKRT95}, and \cite[section 2, page 87]{POPphd98}, has conjectured that the blowup of solutions of the CGL equation \eqref{equa-CGL} can be suppressed for certain parameters, if \(\delta\) is not too small and \(\beta\) is within a specific range around \(\beta = 0\). More precisely, based on numerical computations, they observed that, for generic initial conditions, the solution appears to remain bounded for all time.

In fact, in our paper, we show that their conjecture is not valid for all initial data classes. More precisely, we identify a countable family of blowup solutions for \(\beta = 0\) and any \(\delta \in \mathbb{R}\).

Accordingly, we suspect that all of our solutions are unstable with respect to initial data. More precisely, by construction, our solutions should exhibit co-dimension \(2k\)-instability. As a result, it is unlikely that these solutions would be observed in numerical simulations or physical experiments modeled by the CGL equation.

\end{rem}

\begin{rem}
In our paper, we focus on the construction of solutions in the case $\beta=0$, but we believe that the  construction 
of blow-up solution as in \eqref{theorem-intermediate} is also  possible for the case $\beta\not =0$, but there is additional difficulty coming from the fact that the linearized operator in that case is neither self-adjoint nor diagonalizable. %\textcolor{red}{However, we think that we could have some limitations on the parameters $\beta$ and $\delta$ for the construction of such a profile.} 
However, we believe that certain constraints on the parameters \( \beta \) and \( \delta \) may arise during the construction of such a profile.

\end{rem}

\begin{rem}(Difficulty and  strategy of the proof)
\begin{itemize}
\item {To prove  Theorem \ref{Theorem-principal}, we consider a perturbative problem arising from the linearization  around the constructed profile} and then we prove that the perturbation goes to $0$ as we approach the blow-up time.
In this work we use a tricky linearization introduced by Bricmont and Kupiainen in \cite{BKLcpam94}. Indeed, we will introduce the following perturbation, modulo a phase, $u(T-t)^{-\frac{1}{p-1}}|f_b|^{-(p-1)}f_b^{-1}-(p-1+by^{2k})$. The study of such linearisation will simplify the computations as you will see in section \ref{section-priori-estimates}.
\item Our construction in this work is inspired by the work of Bricmont and Kupiainen in \cite{BKLcpam94} and our recent result in \cite{DNZCPAA24}. But this is far from being a simple adaptation of the construction made in the case of the nonlinear heat equation because of the complex structure of the CGL equation \ref{equa-CGL}. Indeed we have a potential term $V$ (see \eqref{defi-V}) that appears in the linearized equation. We note that the computation of the projection of the potential is much more difficult to handle (see Lemma \ref{lemma-estimation-Pn-V} and \ref{lemma-estimation-P--V}). We would like to add that in comparison to the proof of Bricmont and Kupiainen, we provide  more technical details, we cite for example the proof of the spectral gap estimates on the semigroup given in Section \ref{section-estimation-on-the-semigroup}.

\item  The proof of Theorem \ref{Theorem-principal} relies on the understanding of the dynamics of the self-similar
version of \eqref{equa-CGL} (see \eqref{equation-w} below) around the profile \eqref{defi-f-b}. Moreover, we proceed in two steps:
\begin{itemize}
	\item First, we reduce the question to a finite-dimensional problem: we show that it is
	enough to control a (2k)-dimensional variable in order to control the solution (which is infinite dimensional) near the profile.
	\item Second, we use a contradiction argument to solve the finite-dimensional problem and
	conclude using a topological argument.
\end{itemize}
\end{itemize}
\end{rem}
\medskip 

\noindent
\textbf{Structure of the paper:} We proceed in several sections to prove Theorem \ref{Theorem-principal}. First, in Section \ref{section-formalation-problem}, we give a formal approach to our problem and set up the main linearized problem around the suitable approximation. Then, in Section \ref{section-spectral-properties} we show spectral properties of the linear operators. In Section \ref{section-Proof-assuming-estimates}, we will explain the strategy of the proof of Theorem \ref{Theorem-principal} without technical details. Finally, in Sections \ref{section-priori-estimates} and \ref{section-estimation-on-the-semigroup}, we 
provide crucial \textit{a priori} estimates for  our solution and give the technical details of the proof of Theorem \ref{Theorem-principal}.

\medskip

%\textcolor{red}{\textbf{OUT}\\
%To make the structure of the paper more convenient for readers, we  mention here the structure of the paper. In Section \ref{section-formalation-problem}, we give a formal approach to our problem and set up the main linearized problem around the suitable approximation. Next,  in Section \ref{section-spectral-properties}, we show spectral properties of the linear operators.   
 % Section \ref{section-Proof-assuming-estimates} plays a central role in the paper by reducing our problem to a finite-dimensional one and providing the conclusion for this reduced problem. Finally, the conclusion yields the proof of Theorem \ref{Theorem-principal}  (see Section \ref{proof-Theorem-1}). In Sections \ref{section-priori-estimates} and \ref{section-estimation-on-the-semigroup}, we 
%provide  \textit{a priori} estimates for  our solution which are crucial to our analysis. Finally, in the last section, we provide the necessary estimates on the action of semigroups on the negative part of the solution. We hope this clarifies the structure of the paper. 
%}
\medskip

\textbf{Acknowledgements}: 
Hatem Zaag wishes to thank Pierre Raphael and the "SWAT" ERC project for their support.
% The work of  Hatem Zaag  is  supported by ERC Advanced Grant LFAG/266  ``Singularities for waves and fluids''.  
The work of  Duong Giao Ky is   supported by a grant from the Vietnam Academy of Science and Technology under the grant number CTTH00.03/23-24. In particular, the authors would like to thank the referees for the thorough review, attentive reading, and valuable comments.

	\section{Formulation of the problem}\label{section-formalation-problem}
	
	In this section,  we   aim to formulate the  main problem addressed in our paper. First,  we provide a formal explanation of how the profile given by \eqref{theorem-intermediate} in Theorem \ref{Theorem-principal} is selected,   and then we  introduce   the linearized  problem around this  profile.
 
	\medskip  
	Let $\beta=\alpha =0$, then  the equation  \eqref{GL}  reduces to \eqref{equa-CGL},  which we rewrite below
	\beqtn%\label{equa-CGL}
	\pa_t u=\Delta u+(1+i\delta )|u|^{p-1}u.
	\eeqtn
	Now, we assume that  $u$ is a solution to \eqref{equa-CGL} on $[0, T)$ for some $T >0$ and   $k  \ge 2$   is an integer number. We introduce the similarity variables  as follows 
	\begin{equation}
		w(y,s)=(T-t)^{\frac{1+i\delta}{p-1}} u(x,t),\;y=\frac{x}{(T-t)^{\frac{1}{2k}}},\;\; s= -\ln(T-t).
		\label{change-variable2}
	\end{equation}
	Thanks to \eqref{equa-CGL},  $w$  solves the following equation
	\beqtn\label{equation-w}
	\pa_s w=I^{-2}(s)\Delta w-\frac{1}{2k} y \cdot \nabla w-\frac{1+i\delta}{p-1}w+(1+i\delta) |w|^{p-1}w,
	\eeqtn
	where $I(s)$ is defined   by 
	\begin{equation}\label{defi-I-s}
		I(s) = e^{\frac{s}{2} \left(1-\frac{1}{k} \right)}.
	\end{equation}
	In our paper, we  are interested in  a formal solution to \eqref{equation-w}  of the form
	\begin{equation*}
		w(y,s) = \sum_{j=0}^\infty  \frac{w_j(y)}{I^{2j}(s)}= \sum_{j=0}^\infty  \frac{w_j(y)}{e^{sj \left(1-\frac{1}{k} \right)}}, 
	\end{equation*}
	where  each   $w_j$ 	is assumed to be smooth and globally bounded. Plugging this ansatz into \eqref{equation-w}  and considering the leading order, we find that 
	$$ - \frac{1}{2k} y  \cdot \nabla w_0 - (1 +  i \delta )\frac{w_0}{p-1} + (1+ i \delta) |w_0|^{p-1} w_0=0 .$$
	Since we aim to search global solutions  $w$, $w_0$ must be the same.   Up to a phase modulation, there exists \( b > 0 \) such that the solution can be chosen as follows: 
	\begin{equation}\label{solution-w-0}
		w_0 (y) = (p-1 + b y^{2k})^{-\frac{1+i \delta}{p-1}}  .
	\end{equation}
	Thus,   \eqref{solution-w-0}   formally explains how the profile in Theorem \ref{Theorem-principal} arises. 
	
	\medskip 
{	Inspired by   \cite{BKLcpam94}, we consider  the linearized problem as follows}
	\beqtn\label{definition-q}
	w(y,s)=e^{i\theta(s)}f_{b(s)}(y,s) \left (1+e_{b(s)}(y,s)q(y,s)\right),
	\eeqtn
	where the functions $f_{b(s)}$ and $e_{b(s)}$ are given by 
	\begin{equation}\label{defi-f-b}
		f_{b(s)}(y,s)=\left (p-1+b(s)y^{2k}\right)^{-\frac{1+i\delta}{p-1}}, 
	\end{equation}
	and 
	\begin{equation}\label{defi-e-b}
		e_{b(s)}(y,s)=\left (p-1+b(s)y^{2k}\right)^{-1}.
	\end{equation}

It is worth noting that the combined influence of the two parameters $b(s)$ and $\theta(s)$ is crucial to our analysis, as they govern the dynamics of neutral modes in our construction problem. 	Our linearized problem focuses on the  linearized solution $q$ and also extends to determine the flows $(\theta, b)$ (so-called modulation parameters). Using \eqref{equation-w} and \eqref{definition-q}, we arrive at
	\begin{equation}\label{equa-q}
		\partial_s q= \mathcal{L}_{\delta,s} q + B(q)+ T(q)+N(q)+ D_s(\nabla q)+R_s(q)+V(q),
	\end{equation}
	where
	\begin{align}
		\mathcal{L}_{\delta,s} q &=I^{-2}(s)\Delta q-\frac{1}{2k} y \cdot \nabla q+(1+i\delta)\Re ( q), \label{defi-mathcal}  \\
		V(q) & =  \left ((p-1)e_b-1\right )\left [(1+i\delta){\Re (q)} -q\right], \label{defi-V}\\
		N(q) &=(1+i\delta)\left ( |1+e_bq|^{p-1}(1+e_bq)-1-2e_b{\Re (q)}-\frac{p-1}{2} e_b q-\frac{p-3}{2}e_b \bar q\right ), \label{defi-N}
	\end{align}
	
	and 
	\begin{equation}\label{defi-rest-term-in-equation-of-q}
		\left\{ 
		\begin{array}{rcl}
			B(q)&=&\frac{b'(s)}{p-1}y^{2k}\left (1+i\delta+(p+i\delta)e_b q\right ),\\[0.3cm]
			T(q)&=&-i\theta'(s)\left (e_b^{-1}-q\right )=-i\theta'(s)\left (p-1+by^{2k}-q \right ),\\[0.3cm]
			D_s(q) &=&-I^{-2}(s)\frac{p+i\delta}{p-1}4kb y^{2k-1}  e_b\nabla q,\\[0.3cm]
			R_s(q) &=& I^{-2}(s)y^{2k-2}\left( \alpha_1+\alpha_2y^{2k}e_b+ \left (\alpha_3+\alpha_4 y^{2k}e_b\right )q\right),
		\end{array}
		\right. 
	\end{equation}
	with the explicitly determined constants as follows
	\beqtn
	\begin{array}{ll}
		\alpha_1=-(1+i\delta)2k(2k-1)\frac{b}{p-1}   &   \alpha_2=4(1+i\delta)(p+i\delta)k^2\frac{b^2}{(p-1)^2},\\[0.4cm]
		\alpha_3=-(p+i\delta) 2k(2k-1)\frac{b}{p-1}    &
		\alpha_4=4(p+i\delta)(2p-1+i\delta)k^2\frac{b^2}{(p-1)^2}.
	\end{array}
	\eeqtn

	%$z = \Re z + i \Im z$ as follows
	%$$z = \Re z + i \Im z = (1+i\delta) z_1 + i z_2,  $$
	%where  $z_1,z_2 \in \R$ 
	%$$  z_1 = \Re z \text{ and }  z_2  = \Im z - \delta \Re z .$$
	\medskip

	\medskip 
	
	%\textbf{Ky's Idea:} Dear %Nejla, I suggest to consider %the system of $q_1, q_2$  the %unique 

	\section{Spectral properties of linear operators }\label{section-spectral-properties}
	In this  section,  we  aim to present some  spectral properties of the linear operators  that appear in our paper. 
	First, let us introduce 
	\begin{equation}\label{defi-rho-s}
		\displaystyle \rho_s(y) =\frac{I(s)}{\sqrt{4\pi}} e^{-\frac{I^{2}(s)y^2}{4}},
	\end{equation}
 we introduce  $L^2_{\rho_s}(\R,\C)$,  the  weighted Hilbert space defined by 
	\begin{equation}\label{defi-L-2-rho-s}
		L^2_{\rho_s}(\R,\C) = \left\{  f \in L^2_{loc}(\R,\C) \text{ such that }  \int_\R  |f|^2 \rho_s dy  < +\infty \right\}.
	\end{equation}
	For each $ m \in \N$, we note $h_m(z)$  the Hermite polynomial of degree $m$,   defined by 
	\begin{equation}\label{defi-harmite-H-m}
		h_m(z) = \sum_{\ell=0}^{[\frac{m}{2}]}\frac{m!}{\ell!(m-2\ell)!} (-1)^\ell z^{m-2\ell},
	\end{equation}
	{where $[n]$ stands for the greatest  integer number less than or equal $n$}. In particular, it is well known that  
	\begin{equation}\label{integral-h-m-n}
		\int h_nh_m\rho_s dy=2^nn!\delta_{nm}, \text{ where } \delta_{m,m} = \left\{  \begin{array}{rcl}
		 1 \text{ if } m =n\\
		 0 \text{ if } m \ne n
		\end{array}  \right..
	\end{equation}
	Next,  we define  $H_m$  by 
	\beqtn\label{eigenfunction-Ls}
	H_m(y,s)=I^{-m}(s)h_m(I(s) y)=\sum_{\ell=0}^{[\frac{m}{2}]}\frac{m!}{\ell!(m-2\ell)!}(-I^{-2}(s))^\ell y^{m-2\ell}.
	\eeqtn
Thanks to \eqref{integral-h-m-n}, it holds that 
	\beqtn\label{scalar-product-hm}
	\dsp
	({H}_n(.,s),{H}_m(.,s))_s=\int {H}_n(y){H}_m(y)\rho_s(y)dy=I^{-2n}(s)2^n n!\delta_{nm}.
	\eeqtn
	We introduce
	\begin{equation}\label{defi-mathcal-L-0}
		\mathcal{L}_{0,s} q  :=  I^{-2}(s)\Delta q- \frac{1}{2k} y \cdot \nabla q,
	\end{equation}
	and 
	\begin{equation}\label{defi-mathcal-L-s}
		\mathcal{L}_s q:=I^{-2}(s)\Delta q-\frac{1}{2k} y \cdot \nabla q + q = \mathcal{L}_{0,s} q + q. 
	\end{equation}%here
By using \cite{DNZArxiv2022}, we can easily check that $H_m$ defined by \eqref{eigenfunction-Ls}, satisfies
	\begin{equation}\label{Jordan-bloc-L-s}
		\mathcal{L}_s H_m(y,s)=
		\left \{
		\begin{array}{lll}
			\left(1-\frac{m}{2k} \right)H_m(y,s) +m(m-1)(1-\frac{1}{k})I^{-2}(s) {H_{m-2}}  & \mbox{ if } & m\geq 2,  \\[0.3cm]
			\left(1-\frac{m}{2k} \right)H_m(y,s)& \mbox{ if } & m\in \{0,1\}.
		\end{array}
		\right .
	\end{equation}
	
	\medskip 
	
	For some $s \ge \sigma $, we denote $\mathcal{K}_{0,s,\sigma}$ and $\mathcal{K}_{s,\sigma}$ as the semigroups corresponding to the linear operators $\mathcal{L}_{0,s}$ and $\mathcal{L}_s$, respectively. In fact,   these semigroups represent the fundamental solutions  to the following equations 
	\begin{equation}\label{defi-semigroup-mathcal-K-0-s-sigma}
		\left\{
		\begin{array}{rcl}
			\partial_s \mathcal{K}_{0,s,\sigma} & = &  \mathcal{L}_{0,s} \mathcal{K}_{0, s,\sigma} \text{ for all } s > \sigma,    \\[0.2cm]
			\mathcal{K}_{0, \sigma, \sigma} & =  &  Id,
		\end{array}
		\right. 
	\end{equation}
	and 
	\beqtn\label{defi-mathcal-K}
	\left\{
	\begin{array}{rcl}
		\partial_s \mathcal{K}_{s,\sigma} & = &  \mathcal{L}_s \mathcal{K}_{s,\sigma} \text{ for all } s > \sigma,    \\[0.2cm]
		\mathcal{K}_{ \sigma, \sigma} & =  &  Id.
	\end{array}
	\right. 
	\eeqtn
	Thanks to 	Mehler's formula,  the kernels of these semigroups are explicit (initially proved in \cite{BKnon94}) and are given by 
	\beqtn\label{Kernel-Formula}
	\dsp \mathcal{K}_{0, s, \sigma}(y,z) = \mathcal{F} \left ( e^{-\frac{s-\sigma}{2k}}y-z \right )	\text{ and } \dsp \mathcal{K}_{s\sigma}(y,z)=e^{s-\sigma}\mathcal{F} \left ( e^{-\frac{s-\sigma}{2k}}y-z \right ),
	\eeqtn
	where 
	\beqtn
	\dsp \mathcal{F}(\xi)=\frac{L(s,\sigma)}{\sqrt{4\pi}}e^{-\frac{L^2\xi^2}{4}}\mbox{ with } L^2(s, \sigma) =\frac{{I^{2}(\sigma)}}{(1-e^{-(s-\sigma)})}\mbox{ and }I(s)=\dsp e^{\frac s2(1-\frac 1k)}.
	\eeqtn
	In particular,	we have
	\beqtn
	\mathcal{K}_{0,s,\sigma}H_n(.,\sigma)=e^{-(s-\sigma)(\frac{n}{2k})}H_n(.,s) \text{ and } \mathcal{K}_{s,\sigma}H_n(.,\sigma)=e^{(s-\sigma)(1-\frac{n}{2k})}H_n(.,s).
	\label{kernel-Hn}
	\eeqtn

	\medskip
	\noindent
 Now, let 
	\begin{equation}\label{defi-hat-H-and-check-H-n}
		\left\{ \hat H_m=(1+i\delta) H_m(y,s), \check H_m= iH_m(y,s)|m\in\N  \right\},
	\end{equation}
	where $H_m(y,s)$ is defined   in  \eqref{eigenfunction-Ls}.
		{Then, by \eqref{Jordan-bloc-L-s}, we  have the following  \textit{Jordan block} decomposition of  $\mathcal{L}_{\delta,s}$,}
	\beqtn
	\left\{
	\begin{array}{lll}
		\mathcal{L}_{\delta,s} (\hat {H}_m)&=&\dsp \left(1-\frac{m}{2k} \right)\hat {H}_m+ m(m-1)\left(1-\frac 1k\right)I^{-2}(s)\hat{H}_{m-2},\\[0.5cm]
		\mathcal{L}_{\delta,s} (\check H_m)&=&-\frac{m}{2k} \check{H}_m+ m(m-1)\left(1-\frac 1k\right)I^{-2}(s) \check{H}_{m-2},  \text{ for all } m \ge 2,
	\end{array}
	\right.
	\label{spectLtilde}
	\eeqtn
	and for $m\in \{0,1\}$,   we have 
	\beqtn
	\left\{
	\begin{array}{lll}
		\mathcal{L}_{\delta,s} (\hat {{H}}_m)&=&\dsp \left(1-\frac{m}{2k} \right)\hat {{H}}_m,\\[0.4cm]
		\mathcal{L}_{\delta,s} (\check {H}_m)&=&-\frac{m}{2k}\check {H}_m.
	\end{array}
	\right .
	\label{spectLtilden=0-1}
	\eeqtn

	%we note also that each $r \in L^{2}_{\rho} $ can be uniquely used as
	%\[F(y)=(1+i\delta) F_1 (y)+i F_2(y)=\dsp (1+i\delta) \left(\sum_{m=0}^{+\infty} F_{1m}\right)+i \left(\sum_{m=0}^{+\infty}F_{2m}\right),\]
	%where
	%\beqtn
	%\begin{array}{l}
	%\dsp F_1(y)=\Re (F(y))\mbox{, } F_2(y)=\Im(F(y))-\delta \Re (F(y))\\
	%\dsp \mbox{and  for }i={1,2},\;\; F_{im}(y)=\dsp\int _{\R} F_i(y)\frac{H_m(y,s)}{\|H_m(y,s)\|_{L^{2}_{\rho_s}}}\rho_s(y) dy.
	%\end{array}
	%\eeqtn

	\subsection{Decomposition of $q$}\label{sectidecompq}
	In this part,  we  aim to decompose   the linearized solution $q(s)$  in terms of   time-dependent polynomials $\{H_n(s), n \ge 0\}$ (also  $\{  \hat H_n(s), \check H_n(s), n \ge 0  \}$).  Let $M >1$ (will be fixed later by \eqref{defi-M}), we decompose  $q$ as follows, 
	\beqtn
	q(y,s)=\sum_{0 \le n\leq [M]} Q_n(q,s) H_n(y,s)+q_{-}(y,s),
	\label{decomp1}
	\eeqtn
where  $[M]$  stands for the greatest integer less than or equal to $M$, 
	 $H_n$ is defined in \eqref{eigenfunction-Ls}, and {$Q_n(q,s)\in \C$ is given by
	\begin{equation}\label{defi-Q-n}
		Q_n(q,s)=\dsp \frac{\dsp\int q H_n\rho_s}{\dsp\int H_{n}^{2}\rho_s}.
	\end{equation}}
	Additionally, we note that  $q_-$ is orthogonal to $H_n$ for all $n \le [M]$, that is 
	\[\int q_{-}(y,s)H_n(y,s)\rho_s(y)dy=0\mbox{ for all }n\leq [M].\]
	{We also introduce the projection operators $P_{+,[M]}$ and $P_{-,[M]}$, 
	\begin{equation}\label{defi-P_+M}
		 P_{+,[M]}(q) =  \sum_{0 \le n\leq [M]} Q_n(q,s) H_n(y,s),
	\end{equation}
	and 
	\begin{equation}\label{defi-P-}
		P_{-,[M]} (q) :=  q  - P_{+,[M]}(q).      
	\end{equation}
For simplicity, we note
	\begin{equation}\label{defi-q-+P-+M-}
		q_+ = P_{+,[M]}(q) \text{ and } q_- = P_{-,[M]}(q),  
	\end{equation}
that we express
	\begin{equation}\label{decom-q=q+plus-q-}
		q =  q_+ + q_-.
	\end{equation}}
	
	\medskip
	\begin{definition}[$\delta$-decomposition]\label{Definition-complex-decomposition} Let $\delta \in \R$, if we introduce the projectors $\mathscr{Q}_{\Re,  \delta} $  and $ \mathscr{Q}_{\Im,  \delta}$  on complex numbers   by  
		\begin{equation}\label{defi-mathscr-Q-Re-Im}
			\mathscr{Q}_{\Re, \delta}(z) = \Re (z) = \hat z   \text{ and }   \mathscr{Q}_{\Im, \delta}(z)  = \Im (z) - \delta \Re(z) = \check z,
		\end{equation}
		 then, for each $z \in \C$, we have a unique decomposition  
		$  z = \hat z  (1 + i\delta) + i \check z .   $
	\end{definition}
	
	\medskip 
	\noindent 
	Consequently, 
	\begin{lemma}\label{lemma-complex-decomposition} The  projectors $\mathscr{Q}_{\Re, \delta}$ and $\mathscr{Q}_{\Im, \delta}$  introduced  in Definition \ref{Definition-complex-decomposition}, satisfy 
		\begin{itemize}
			\item[(i)] for all $\lambda \in \R, z_1, z_2   \in \C $,
			\begin{equation}\label{complex-projection-linear}
				\mathscr{Q}_{\Re, \delta}(z_1 + z_2 ) = \mathscr{Q}_{\Re, \delta}(z_1) + \mathscr{Q}_{\Re, \delta}(z_2) \text{ and } \mathscr{Q}_{\Im, \delta}(z_1+ z_2) = \mathscr{Q}_{\Im, \delta}(z_1) + \mathscr{Q}_{\Im, \delta}(z_2),
			\end{equation}
			and 
			\begin{equation}
				\mathscr{Q}_{\Re, \delta}(\lambda z_1) = \lambda \mathscr{Q}_{\Re, \delta}(z_1) \text{ and } \mathscr{Q}_{\Im, \delta}(\lambda z_1) = \lambda \mathscr{Q}_{\Im, \delta}(z_1),
			\end{equation}
			\item[(ii)] In particular, for the potential term $V(q)$ defined by  \eqref{defi-V}, we have 
   %for $V(q)$ defined  as in  \eqref{defi-V}, then  
			\begin{equation}
				\widehat{V(q)} = 0 \text{ and }\widecheck{V(q)}= \left (1 - (p-1)e_b\right )\check q.
			\end{equation}
		\end{itemize}
		
	\end{lemma}
	\begin{proof}
		{
Let $z_1, z_2 \in \C$, then we write  $z_1 = a_1 + i b_1 $ and $z_2 = a_2+ i b_2$ where $ a_j,b_j \in \R, $ for $j=1,2$.  Now, we have 
\begin{align*}
	\lambda z_1 &=  \lambda a_1 (1 + i \delta ) + i \lambda (b_1 - \delta a_1), \text{ for } \lambda \in \R,\\
z_1 + z_2 &=  (a_1 + a_2)  + i (b_1 + b_2) =  (a_1 + a_2) (1 + i \delta) + i [  (b_1 + b_2) - \delta(a_1+a_2)], 
\end{align*}
which concludes the identities in  item \textit{(i)}. Next, we  consider $V(q)$ defined by \eqref{defi-V}, we apply these identities that we find that
\begin{align*}
\widehat{V(q)} &=  ((p-1)e_b - 1)  \left( \widehat{(1+i \delta) \Re(q) }  - \widehat{q} \right)  = ((p-1)e_b - 1)  \left( \Re(q) - \Re(q) \right) =0,\\
\widecheck{V(q)} & = ((p-1)e_b - 1) \left( \widecheck{(1+i \delta) \Re(q) }  - \widecheck{q} \right) =((p-1)e_b - 1) \left(- \check{q} \right) =(1 - (p-1)e_b ) \check q. 
\end{align*}
} Thus, item (ii) follows, and we finish the proof of the Lemma.
	\end{proof}

	According to Definition \ref{Definition-complex-decomposition},  we can  decompose  $q_+$ and $q_-$ as follows
	\beqtn
	\dsp q_+=\sum_{n\leq [M]}Q_n H_n(s)=\sum_{n\leq [M]}\hat q_n\hat{{H}}_n(s)+  \check q_n \check{H}_n(s),
	\label{decomp2}
	\eeqtn
	and 
	\begin{equation}\label{decompo-q--=hat q--+check-q--}
		q_- = (1 + i \delta  ) \hat q_-  + i \check{q}_-,
	\end{equation}
	where $\hat q_n,\check q_n \in \R$ and   $\hat q_-, \check q_-$ are real-valued functions.  In particular, these components can be explicitly expressed by

	%Finally, we notice that for all $s$, we have
	%	\[\dsp\int q_-(y,s)q_+(y,s)\rho(y)dy=0.\] 
	%	Our purpose is to project \eqref{equation-q-1} in order to write an equation for $\hat q_n$ and $ \check q_n$. Note that
	\beqtn
	\hat q_n(s)  = \hat P_{n} (q) := \mathscr{Q}_{\Re, \delta}(Q_n(q)),\; \check q_n(s) = \check P_{n,M}(q) :=  \mathscr{Q}_{\Im, \delta} (Q_n(q)). 
	\label{decomp3}
	\eeqtn
	and
	\begin{equation}\label{defi-hat-q--check-q--}
		\hat q_- = \mathscr{Q}_{\Re, \delta } (q_-) \text{ and  }  \check{q}_- =  \mathscr{Q}_{\Im, \delta}(q_-) .
	\end{equation}
	Finally, we  can express $q$   as follows
	\beqtn
	q(y,s)=\left(\sum_{n\leq [M]}\hat q_n(s) \hat{{H}}_n(y,s)+ \check q_{n}(s) \check {H}_n(y,s)\right)+ (1+i\delta)\hat q_-(y,s)+i\check q_-(y,s).
	\label{decomposition-q}
	\eeqtn

	\subsection{Equivalent norms }
	\noindent
	{	In this section, we establish the   equivalent norms  arising from  our problem.  Let us introduce  $L^\infty_M$  defined by }
	\begin{equation}\label{defi-L-M}
		L^{\infty}_M(\R)=\{g \text{ such that } (1+|y|^M)^{-1} g \in L^\infty (\R)\}.
	\end{equation}
Then, the space 	$L^\infty_M$ is complete with the following norm
	\begin{equation}\label{defi-norm-L-M}
		\|g\|_{L^\infty_M} = \|(1+|y|^M)^{-1} g \|_{L^\infty},
	\end{equation}
	as $L^\infty$ is complete with respect to the norm $\|.\|_{L^\infty}$. We now define the norm $\|.\|_s$ as follows
	\beqtn\label{norm-q-2}
	\|q\|_s=\sum_{m=0}^{[M]}{|Q_m(q)|}+|q_-|_s,
	\eeqtn
{	where $Q_m(q)$  is defined  by \eqref{defi-Q-n} and }
	\begin{equation}\label{defi-norm-q---s}
		|q_-|_s=\displaystyle \sup_{y}\frac{|q_-(y,s)|}{I^{-M}(s)+|y|^M}.
	\end{equation}
{	Therefore, we have 
	\beqtn\label{equivalence-norm}
	C_1(s)\|q\|_{L^\infty_M} \leq \|q\|_s\leq C_2(s)\|q\|_{L^\infty_M},
	\eeqtn 
for some positive constants $C_1(s), C_2(s)$  depending   on $s$. Thus,  it follows that the space $L^\infty_M$ is also complete with respect to the norm $\|.\|_s$.} 
	
	\section{Proof of Theorem \ref{Theorem-principal} in 
		assuming some technical results}\label{section-Proof-assuming-estimates}
	In this section, we give the complete proof of Theorem \ref{Theorem-principal}. The main idea is to reduce the problem to a finite-dimensional problem (a $2k$-dimensional one). Then, the reduced problem can be solved by a topological argument. We aim to explain our strategy in a reader-friendly manner. The main steps of the strategy are outlined below:

	\begin{itemize}
		\item In the first step:  we construct a shrinking set $V_{A,\gamma, b_0, \theta_0}(s)$ including necessary bounds such that the belonging in this set completely implies the result in Theorem \ref{Theorem-principal}.
	\item In the second step: we construct initial data at initial time  $s_0=-\ln (T)$ for  \eqref{equa-q},  which is parameterized by $2k $  parameters following the $2k$ projections  $\hat q_0,..,\hat q_{2k-1}$ of $q$ on $\hat H_n, n \le 2k$.
	\item In the third step, we impose  orthogonal conditions
	\begin{equation}\label{orthogonal-conditions-H-2k-H-0}
		\mathscr{Q}_{\Re, \delta} \left( \int q(s)  H_{2k}(y,s) \rho_s dy \right)  = \mathscr{Q}_{\Im, \delta}\left (\int q(s)  H_0(y,s) \rho_s dy  \right) =0,
	\end{equation}
	this condition nullify $\hat q_{2k}$ and $ \check q_0$, which are the projections of $q$ onto $\hat H_{2k}$ and $\check H_0$. 
	According to \eqref{spectLtilde}, the projections involve the zero modes arising as significant challenges in the construction. Therefore, the use of these modulations is critical to our construction. Additionally, we show the locally unique existence of the solution $(q,b,\theta)$ to the coupled problem  \eqref{equa-q}   \& \eqref{orthogonal-conditions-H-2k-H-0}. 
	\item In the fourth step:  By using the spectral approach  of the linear operator $\mathcal{L}_{\delta,s}$, we reduce the control of $(q,b,\theta)(s)$ (an infinite dimensional problem) to 
	a $2k$-dimensional one involving  $(\hat q_0,...,\hat q_{2k-1}) $.
	\item  In the last subsection, we solve the finite-dimensional one by using a topological argument and give the complete conclusion to   Theorem \ref{Theorem-principal}.
	\end{itemize}
	\iffalse 
	In this section, assuming the technical details given in section ??, we prove the existence of a solution $q(y,s), \theta(s),b(s)$ of problem ?? such that
	\beqtn
	\lim_{s\to \infty}\left \|\frac{q}{1+|y|^M}\right \|_{L^\infty}=0,\;\; |b'(s)|\leq ??\mbox{ and }|\theta '(s)|\leq ??\mbox{ for all }s\in [-\ln T,+\infty).
	\label{limitev}
	\eeqtn
	%where $\theta_0\in\R$.\\
	\begin{rem} Hereafter, we denote by $C$ a generic positive constant, depending only on $p$ and $k$ introduced in \eqref{defchi14}, itself depending on $p$. In particular, $C$ does not depend on $\delta$ and $s_0$, the constants that will appear shortly and throughout the paper and need to be adjusted for the proof.
	\end{rem}
	\begin{rem}  Note that   $s_0= -\ln (T)$ is the \textit{master constant}. In almost every argument in this paper it is almost to be sufficiently depending on the choice of all other constants ($\delta,\;\nu_0$ and $b_0$). In addition, we denote $C$ as the universal constant that is independent to $b_0$ and $s_0$. 
	\end{rem}
	\fi 
	\subsection{Definition of a shrinking set}
	% and preparation of initial data.}
In this section, we construct a  ``shrinking set'' which  control the behavior  of the solution by some suitable error bounds. 
% We also give the general form of initial data, so that our solution will be in the shrinking set at the starting time. Let us first give this definition.
\begin{definition}[Shrinking set]\label{definition-shrinking-set}
	Let  $k \in \N$, $k\ge2, \gamma >0, b_0 >0, \theta_0 \in \R, s \ge 1$, $ A \ge 1$, and $M$ defined as in \eqref{defi-M},   we define $V_{ A, \gamma, b_0, \theta_0}(s)$ {as the set of all $(q,b,\theta)  \in \left(L^\infty_M (\R,\C),\R,\R\right )$ satisfying  the following conditions:}
	\begin{itemize}  
		\item[(i)] The first condition: for all $n$  satisfying $0\leq n\leq [M]$, we have 
		\begin{equation}\label{definition-shrinking-set-q_n}
			|\hat q_{n}|   \leq I^{-\gamma}(s) \text{ and } |\check q_{n}|   \leq I^{-\gamma}(s).
		\end{equation}
		\item[(ii)] The second condition: 
		\begin{equation}\label{definition-shrinking-set-q_-}
			\displaystyle \left \| \dsp\frac{\hat q_-(y,s)}{I^{-M}(s)+|y|^{M}}\right \|_{L^\infty}\leq 
			I^{-\gamma}(s),\;\; \displaystyle \left \| \dsp\frac{\check q_-(y,s)}{I^{-M}(s)+|y|^{M}}\right \|_{L^\infty}\leq 
			AI^{-\gamma}(s).
		\end{equation}
		where $I(s)$ defined as in \eqref{defi-I-s}; $\hat q_n$ and $\check q_n$ given as in \eqref{decomp3}; and $\hat q_- = \mathscr{Q}_{\Re, \delta}(q_-)$, $ \check q_- =\mathscr{Q}_{\Im, \delta}(q_-)$ and the negative part $q_-$ defined as in \eqref{defi-q-+P-+M-}. 
		\item[(iii)]  The third condition: 
		\begin{align}
			\frac{b_0}{2}\leq b \leq 2 b_0\label{bound-b},\\
			{\theta_0 - \frac{1}{4} \leq  \theta  \leq  \theta_0 + \frac{1}{4}}.\label{bound-theta}
			\end{align}
	\end{itemize}
\end{definition}

Below,  we  show some rough bounds for functions belonging to $V_{ A,\gamma, b_0, \theta_0}(s)$ for some $s \ge 1$.
\begin{lemma}\label{estimation-l-infin-q} Let $(q,b) \in V_{ A,\gamma, b_0, \theta_0}(s)$ arbitrarily given, then the following estimates hold
	\begin{align}
		\left| q_+(y) \right| &\le C I^{-\gamma }(s) \left(  \sum_{n \le [M]}  (I^{-n}(s) + |y|^n)  \right), \forall y \in \R,\label{esti-rough-q-+}\\ 
		\left| q_-(y)  \right| & \le  CA I^{-\gamma}(s) ( I^{-M}(s) + |y|^M), \forall  y \in \R, \label{esti-rought-pointwise-q--}\\
		\left| \mathbbm{1}_{\{|y| \ge \frac{1}{2}\}} q_+(y) \right| &\le C I^{-\gamma }(s) \left(  I^{-M}(s) + |y|^M \right),   \forall y \in \R,  \\
		\left| \mathbbm{1}_{\{|y| \ge \frac{1}{2}\}} q(y) \right| &\le C AI^{-\gamma }(s) \left(  I^{-M}(s) + |y|^M \right),   \forall y \in \R, \label{bound-q-ygep1-2}
	\end{align}
	where $C>0$ is a universal constant depending only on the nonlinear power $p$ and $k$.
\end{lemma}
\begin{proof}
	The result  immediately  follows from the definition of $V_{ A,\gamma, b_0, \theta_0}(s) $. We kindly refer the reader to check these  estimates. 
\end{proof}

	\begin{rem}
	{Our aim is to construct $(q,b,\theta)$,   defined  on $ [-\ln T, +\infty)$,  where    $$\theta(\cdot) \text{ and } b(\cdot) \in C^1([-\ln T,\infty),\R)$$
 	 are suitably selected   such that equation \eqref{equa-q} has a unique  solution  $q(\cdot,s)$  on $[-\ln T,\infty)$ satisfying
		\[\dsp\|q(s)\|_{L^\infty_M}=\left\|\frac{q(s)}{1+|y|^M}\right \|_{L^\infty} \to 0 \mbox{ as } s\to \infty,\]
		where $M$ is given by
		\beqtn
		M = \frac{2kp}{p-1}. \label{defi-M}
		\eeqtn
  The choice of $M >2k$ ensures that the spectral gap estimates in Lemma  \ref{lemma-estimation-K-q-} are valid, on the one hand. On the other hand,  $M$ is not too large so that  we can control  the growth of the  nonlinear term $N(q)$ as in Lemma \ref{lemma-control-N-q-outer-region}. }
\end{rem}

\subsection{ Preparation of initial data}
Let us now prepare the initial data for  our problem. Consider $\hat{\textbf{d}} = (\hat d_0, ..,\hat d_{2k-1}) \in \R^{2k} $,  we introduce 
\beqtn\label{initial-data}
\psi(\hat{\textbf{d}},y,s_0)=I^{-\gamma}(s_0)  \sum_{j=0}^{2k-1} \hat d_j \hat H_j(y,s_0). 
\eeqtn
As we constructed in Section \ref{section-spectral-properties},  it holds true that  
\begin{equation}\label{intial-orthogonality}
	\hat P_{2k}(\psi) =0, \text{ and } \check P_0(\psi) =0,
\end{equation}
which satisfies \eqref{orthogonal-conditions-H-2k-H-0} at $s_0$.
Below, we show the important properties of the initial data, given by \eqref{initial-data}.
\begin{lemma}[Preparation of the initial data]\label{lemma-initial-data}
	Let     $\hat d=(\hat d_i)_{0\le i \le 2k-1} \in \R^{2k}$ satisfying $ \max_{0 \le i \le 2k-1 } \left| \hat d_i \right| \le 2 $ and $b_0 >0$ arbitrarily given.  Then, there exists {$ \gamma_1(b_0) >0$} such that for all {$\gamma \in (0, \gamma_1)$}, there exists $s_1(\gamma, b_0) \ge 1$ such that for all $s_0 \ge s_1$,    the following properties are valid with $\psi(\hat d,y,s_0)$ defined  in \eqref{initial-data}:
	\begin{itemize}
		\item[(i)]  There exits a quadrilateral $ \mathbb{D}_{s_0} \subset \left[-2,2\right]^{2k} $ such that the mapping   
		\begin{equation}\label{defi-mapping-Gamma-initial-data}
			\begin{gathered}
				\Gamma: \mathbb{D}_{s_0} \to \mathbb{R}^{2k} \hfill \\
				\hspace{-1.5cm} (\hat d_0,...,\hat d_{2k-1}) \mapsto (\hat \psi_0,...,\hat \psi_{2k-1}) \hfill \\ 
			\end{gathered},
		\end{equation}
		is linear one-to-one from $ \mathbb{D}_{s_0}$ to $\hat{\mathcal{V}}(s_0)$, where  
		\begin{equation}\label{define-hat-V-A-s}
			\hat{\mathcal{V}}(s) = \left[ -I^{-\gamma}(s), I^{-\gamma}(s) \right]^{2k},  
		\end{equation}
		and $(\hat \psi_0,...,\hat \psi_{2k-1})$ are the coefficients of initial data $\psi(\hat d_0,...,\hat d_{2k-1})$ corresponding to the decomposition  \eqref{decomposition-q}.
		In addition to that, we have 
		\begin{equation}\label{des-Gamma-boundary-ne-0}
			\left. \Gamma \right|_{\partial \mathbb{D}_{s_0}} \subset \partial \hat{\mathcal{V}}(s_0) \text{ and } \text{deg}\left(\left. \Gamma \right|_{\partial \mathbb{D}_{s_0}} \right) \ne 0.
		\end{equation}
		\item[(ii)] For all $(\hat d_0,...,\hat d_{2k-1}) \in \mathbb{D}_{s_0}$, the following estimates are valid
		\begin{equation}\label{intial-data-check-q}
				\left|\hat  \psi_0  \right| \le I^{-\gamma}(s_0),.... ,  \left|\hat  \psi_{2k-1}  \right| \le  I^{-\gamma}(s_0), \quad 
			\hat \psi_{2k}=... =\hat \psi_{M} =0 \text{ and }\check\psi \equiv 0,  \psi_-\equiv 0.
		\end{equation}
	\end{itemize}
\end{lemma}
\begin{proof} \text{ } \\
	- \textit{The proof of item (i):} 
	From \eqref{initial-data}, definition's $\hat H_n $  as in \eqref{eigenfunction-Ls} and   \eqref{decomp3},  and by a direct computation we can prove that there exists   a square matrix $\mathscr{A}_{s_0}$ satisfying 
	$$  \mathscr{A}   =   Id  +    \tilde{\mathscr{A}_{s_0}}       \text{ with }      \tilde{\mathscr{A}_{s_0}}   = (\tilde a_{ij}) , | \tilde a_{ij}| \le   C I^{-2}(s_0) ,  $$   
	and we have
	\begin{equation}
		\left(	\begin{matrix}
			\hat{\psi}_0\\
			\hat{\psi}_1\\
			\vdots\\
			\hat{\psi}_{2k-1}
		\end{matrix} \right)  =  I^{-\gamma}(s_0)\mathscr{A}_{s_0} 	\left(	\begin{matrix}
			\hat d_0\\
			\hat d_1\\
			\vdots\\
			\hat d _{2k-1}
		\end{matrix} \right),
	\end{equation}
	which immediately  concludes item \textit{(i)}. \\
	- \textit{ The proof of item (ii):}  From \eqref{initial-data}, $\psi(\hat d_0,...,\hat d_{2k-1},s_0)$ is a polynomial of order  $2k-1$, so it follows that 
	$$ \hat \psi_{n} =0, \forall n \in \{2k,..,M\} \text{, } \check\psi_{n} =0, \forall n \in \{0,..,M\}\text{ and }  \psi_{-} \equiv 0.$$
	In addition, since  $(\hat d_0,...,\hat d_{2k-1}) \in \mathbb{D}_{s_0}$, we  {apply}  item (i) to deduce that $(\hat \psi_0,...,\hat \psi_{2k-1}) \in \hat{\mathcal{V}}(s_0)$ and 
	$$ \left| \hat \psi_{n} \right| \le I^{-\gamma}(s_0), \forall n \in \{ 0,...,2k-1\}  $$
	which concludes item (ii) {and} the proof of the lemma.\end{proof}
\begin{rem} In our proof we use many parameters to control the problem. To avoid confusion for readers, we will explain some conveniences here. Firstly, the constants $p, k$, $M$, $\theta_0$, and $b_0$ are given parameters we usually omit in some estimates of our proof. Secondly, the parameters $\gamma, A$ and $s_0$ will be fixed at the end of the proof. In particular,  the parameter $s_0$ (where $T=e^{-s_0}$) will  be  finally fixed based on $A$ and $\gamma$. We also denote $C$ as the universal constant independent of $A$ and $s_0$. 
\end{rem}

\subsection{Local in time solution of problem \eqref{equa-q} \& \eqref{orthogonal-conditions-H-2k-H-0}} 
\iffalse 
As one see from (?), at the initial time $s=s_0$, from the choice of $\psi(y,s_0)$, we see that $f_b$ is an approximate solution of equation \eqref{equation-q-1}. From the continuity of the flow associated with equation \eqref{equation-q-1} $w(y, s)$ will stay close to $f_b(y,s)$, at least for a short time
after $s_0$. We can do better, using the modulation theory, which require that the fluctuation $q(y,s)$ to be almost orthogonal to $\tilde H_{2k}$ and $H_0$, killing the zero directions of  the linearized operator of equation \eqref{equation-w}:

\beqtn\label{condition-modulation}
\hat P_{2k}(q)=0 \mbox{ and }\check P_0(q)=0.
\eeqtn
\fi 
In this part, we prove the local existence  of   solution to  the problem \eqref{equa-q}   \& \eqref{orthogonal-conditions-H-2k-H-0}. The result reads. 
\begin{prop}[Local existence of the  problem \eqref{equa-q}  \& \eqref{orthogonal-conditions-H-2k-H-0}]\label{propo-existence-local}
	Let $\widehat{\textbf{d}} \in \R^{2k}$ satisfying $\left|\hat{\textbf{d}} \right|  \leq 2$, $\gamma >0$, $b_0>0, \theta_0 \in \R$ and $A \ge 1$.  Then, there exits $s_2(\gamma, b_0, \theta_0) \geq  1$, such that for all $s_0\geq s_2$, the following holds:  if we take  initial data $(\psi, b_0, \theta_0)$ with $\psi$ defined as    as in \eqref{initial-data}, then there exists $s_{loc}> s_0$ such that the coupled problem \eqref{equa-q}  \& \eqref{orthogonal-conditions-H-2k-H-0} has a unique solution  $(q,b,\theta) (s)\in  V_{ A,\gamma,b_0, \theta_0}(s)$ on $[s_0,s_{loc}]$. 
\end{prop}
\begin{proof}
	First, it is classical that equation \eqref{equation-w} is locally well-posed in $L^\infty(\R)$. So, with initial data $(\psi, b_0, \theta_0)$ and the transformation in \eqref{definition-q}, there exists $\bar s_1 > s_0$ such that the solution $w$ to  equation \eqref{equation-w} uniquely exists  on $[s_0, \bar s_1]$.  Now,  let us introduce $\vec \mu=(b, \theta)$ and define
	\beqtn
	\begin{array}{rcl}
		\mathcal{F}(s,\vec \mu)&=&
		\left (\
		\begin{array}{l}
			%\mathscr{Q}_{\Re, \delta} \left(\left\langle w(s) (f_be_b)^{-1}e^{-i\theta} - \left( p-1 +b y^{2k} \right), H_{2k}  \right\rangle_{L^2_{\rho_s}}\right)\\
			\displaystyle \mathscr{Q}_{\Re, \delta} \left(\int \left [w(y,s) (f_be_b)^{-1}e^{-i\theta} - \left( p-1 +b y^{2k} \right)\right ] H_{2k}  \rho_s dy\right)\\
			[0.4cm]
			\displaystyle	\mathscr{Q}_{\Im, \delta}  \left(\int \left [ w(s) (f_be_b)^{-1}e^{-i\theta} - \left( p-1 +b y^{2k} \right)\right ] H_{0}  \rho_s dy\right)
		\end{array} \right) \end{array} = \left(\begin{array}{l}
		\mathcal{F}_1(s,\vec \mu)\\
		\mathcal{F}_2(s, \vec \mu)
	\end{array}  \right). 
	\eeqtn
	\iffalse 
	Here and in what follows, all inner products are $L^2_{\rho_s}$ inner products.
	\fi
	By \eqref{defi-Q-n} and 	\eqref{intial-orthogonality}, we have 
	\begin{equation*}
		\mathcal{F}(s_0,\vec \mu_0) =0  \text{ where  }  \vec \mu_0 = (b_0, \theta_0). 
	\end{equation*}	
	So, the  result	 will be a  direct consequence of the implicit function theorem. Let us recall the Jacobian matrix of $\mathcal{F}$ in accordance with  the variable $\vec \mu$
	\iffalse 
	
	Using the implicit function theorem we will prove that for any $\mu_0 := (\theta_0, b_0) \in  [ 0 , x] × (0, 2b_0) $ there exists a unique $C^1$ function $g : X \to \R^+ × \R^+$  defined in a neighborhood $U_{\mu_0}$ such that $G(g(v), v) = 0 $ for all $v \in U_{\mu_0}$.
	
	Note first that the mapping $G$ is $C^1$ and $G(0,\mu_0) = 0$ for all $\mu_0$. We claim that the linear map $\pa_{\mu}G(s_0,\mu_0)$ is invertible. We compute the Jacobian matrix 
	\fi 
	\begin{equation}\label{Jacob-mathcal-F}
		\textbf{J}_{\vec \mu}[\mathcal{F}](s,\vec \mu)=
		\left (\begin{array}{ll}
			\pa_{b} \mathcal{F}_1(s,\mu) & \pa_{\theta} \mathcal{F}_1 (s,\vec \mu)  \\
			\pa_{b} \mathcal{F}_2(s,\vec \mu) & \pa_{\theta} \mathcal{F}_2(s,\vec \mu) 
		\end{array}
		\right ),
	\end{equation}
	where 
	\begin{align*}
		\displaystyle \pa_{\theta }\mathcal{F}_1(s,\vec \mu) &= \mathscr{Q}_{\Re,\delta}\left (-i\int w(s) (f_be_b)^{-1}e^{-i\theta} H_{2k}  \rho_sdy\right),\\
		%\pa_{b }\mathcal{F}_1(s, \vec \mu) &=\mathscr{Q}_{\Re,\delta} \left(\left\langle \frac{p+i\delta}{p-1} w(s) y^{2k}(f_b)^{-1}e^{-i\theta}-y^{2k}, H_{2k}   \right\rangle_{L^2_{\rho_s}}\right),\\
		\displaystyle	 \pa_{b }\mathcal{F}_1(s, \vec \mu) &=\mathscr{Q}_{\Re,\delta} \left(\int \left [ \frac{p+i\delta}{p-1} w(s) y^{2k}(f_b)^{-1}e^{-i\theta}-y^{2k}\right ] H_{2k} \rho_s dy\right),\\
		\displaystyle \pa_{\theta }\mathcal{F}_2(s, \vec \mu) & =
		\mathscr{Q}_{\Im,\delta} \left( -i\int w(s) (f_be_b)^{-1}e^{-i\theta} H_{0} \rho_s dy\right), \\
		% \pa_{\theta }\mathcal{F}_2(s, \vec \mu) & =
		%\mathscr{Q}_{\Im,\delta} \left( -i\left\langle w(s) (f_be_b)^{-1}e^{-i\theta}, H_{0}  \right\rangle_{L^2_{\rho_s}}\right), \\
		%\pa_{b }\mathcal{F}_2(s, \vec\mu) & =
		%\mathscr{Q}_{\Im, \delta} \left(\left\langle\frac{p+i\delta}{p-1} w(s) y^{2k}(f_b)^{-1}e^{-i\theta}-y^{2k}, H_{0}  \right\rangle_{L^2_{\rho_s}}\right).
		\displaystyle   \pa_{b }\mathcal{F}_2(s, \vec\mu) & =
		\mathscr{Q}_{\Im, \delta} \left(\int \left [\frac{p+i\delta}{p-1} w(s) y^{2k}(f_b)^{-1}e^{-i\theta}-y^{2k}\right ] H_{0}  \rho_s\right).
	\end{align*}		
	The main goal is to prove 
	\begin{equation}\label{det-J-F-ne-0}
		\textbf{Det}\left( \textbf{J}_{\vec \mu}[\mathcal{F}](s,\vec \mu) \right) \left. \right|_{(s,\vec \mu) = ( s_0, \vec \mu_0)} \ne 0. 
	\end{equation}		
	
	\medskip 
	\noindent
	\textit{- Expansion  for $\partial_b \mathcal{F}_1(s_0, \vec \mu) \left. \right|_{\vec \mu = \vec \mu_0}$:} 	Thanks to \eqref{definition-q} and \eqref{initial-data}, it follows that  		
	\begin{equation}\label{identity-w-f-e-1-e-b-Psi}
		w(s_0) f_{b_0}^{-1} e^{-i\theta} = 1 + e_{b_0} \Psi,
	\end{equation}
	\iffalse 
	
	\begin{equation}
		w(s_0) f_{b_0}^{-1}e_{b_0}^{-1}=e_{b_0}^{-1} + \Psi
	\end{equation}
	\fi 	
	which yields 
	\begin{align}
		\frac{p+i\delta}{p-1} w(s_0) y^{2k}(f_{b_0})^{-1}e^{-i\theta_0}-y^{2k} = \frac{1+i\delta}{p-1} y^{2k} + \frac{1+i\delta}{p-1} y^{2k} e_{b_0} \Psi.\label{expansion-partial-b-F-1}
	\end{align}	
	So, we derive from \eqref{initial-data} that
	\begin{align*}
		% \left\langle \frac{p+i\delta}{p-1} w(s) y^{2k}(f_b)^{-1}e^{-i\theta}-y^{2k}, H_{2k}  \right\rangle_{L^2_{\rho_s}} =\frac{1+ i\delta}{p-1} I^{-4k}(s_0) 2^{4k} (2k)! \left( 1 + O(I^{-\gamma}(s_0)) \right).
		\displaystyle \int \left [ \frac{p+i\delta}{p-1}w(s) y^{2k}(f_b)^{-1}e^{-i\theta}-y^{2k} \right ]H_{2k}  \rho_s dy =\frac{1+ i\delta}{p-1} I^{-4k}(s_0) 2^{4k} (2k)! \left( 1 + O(I^{-\gamma}(s_0)) \right).
	\end{align*}			
	Therefore, we get
	\begin{align*}
		\partial_b \mathcal{F}_1(s_0, \vec \mu) \left. \right|_{\vec \mu = \vec \mu_0} = \frac{1 }{p-1} I^{-4k}(s_0) 2^{4k} (2k)! \left( 1 + O(I^{-\gamma}(s_0)) \right).
	\end{align*}
	\textit{- Expansion  for $\partial_\theta \mathcal{F}_1(s_0, \vec \mu) \left. \right|_{\vec \mu = \vec \mu_0}$:} Using \eqref{identity-w-f-e-1-e-b-Psi} again, we get
	\begin{align*}
		%-i\left\langle w(s_0) (f_{b_0}e_{b_0})^{-1}e^{-i\theta_0}, H_{2k}  \right\rangle_{L^2_{\rho_s}} = - i  b_0 I^{-4k}(s_0)2^{4k} (2k)!.
		-i\displaystyle\int  w(s_0) (f_{b_0}e_{b_0})^{-1}e^{-i\theta_0} H_{2k}  \rho_s  dy = - i  b_0 I^{-4k}(s_0)2^{4k} (2k)!.
	\end{align*}
	Thus, it follows that 
	\begin{equation}
		\partial_\theta \mathcal{F}_1(s_0, \vec \mu) \left. \right|_{\vec \mu = \vec \mu_0} = 0.
	\end{equation}
	\iffalse 
	\[\pa_{\theta }\mathcal{F}_1(s,\mu)=\Re \left [-i\left\langle p-1+by^{2k}+q(y,s), f_{2k}  \right\rangle\right ]\]
	
	at time $s_0$ and by definition of \eqref{initial-data}, we get
	\beqtn 
	\pa_{\theta }\mathcal{F}_1(s_0,\mu_0)=\Re \left [-i\left\langle b_0y^{2k}, f_{2k}  \right\rangle\right ]=0
	\label{partial_theta_F1}
	\eeqtn

	\[\pa_{\theta }\mathcal{F}_2(s,\mu)=
	\Im  \left [-i\left\langle e_b^{-1}+q , f_{0}  \right\rangle\right ]
	-\delta \Re \left [-i\left\langle e_b^{-1}+q, f_{0}  \right\rangle\right ]
	\]
	Writing $\left\langle e_{b_0}^{-1}+q_0 , f_{0}  \right\rangle= p-1+I^{-\gamma}\left ((1+i\delta)\hat d_0+i \check d_0+C(b_0) I^{-2k}\right )$, we get 
	
	\fi 
	\textit{- Expansion  for $\partial_b \mathcal{F}_2(s_0, \vec \mu) \left. \right|_{\vec \mu = \vec \mu_0}$:} By using \eqref{expansion-partial-b-F-1}, we obtain 
	\begin{equation*}
		\left| \partial_b \mathcal{F}_2(s_0, \vec \mu) \left. \right|_{\vec \mu = \vec \mu_0} \right|  \le C I^{-2k}(s_0).
	\end{equation*}
	\textit{- Expansion  for $\partial_\theta \mathcal{F}_2(s_0, \vec \mu) \left. \right|_{\vec \mu = \vec \mu_0}$:}
	By the same way, we get 
	\beqtn \label{partial_theta_F2}
	\begin{array}{lll}
		\pa_{\theta }\mathcal{F}_2(s_0,\mu_0) & = & (-(p-1)-I^{-\gamma}(s_0)\hat d_0)-\delta I^{-\gamma}(s_0)\left (\delta \hat d_0+\check d_0 \right )+C(b) I^{-2k}(s_0) \\
		&=&  (1-p) + O\left( I^{-\gamma}(s_0) \right).
	\end{array}
	\eeqtn
	\iffalse 
	%From the proof of Proposition 4.4 in (NLH submitted cite ??), we have
	%\[\pa_b \mathcal{F}_2(s_0,\mu_0)=I^{-4k}(s_0)2^{4k}(2k)!\left (1+O(I^{-\gamma}(s_0))\right ).\]
	Now let us estimate $\pa_b \mathcal{F}_2(s_0,\mu_0)$. using the fact that 
	\[w f_b^{-1}=1+e_bq,\]
	then we have 
	\[
	\begin{array}{lll}
		
		\left\langle \frac{p+i\delta}{p-1}w y^{2k}(f_{b_0})^{-1}e^{-i\theta_0}-y^{2k}, f_{2k}  \right\rangle&=&\left\langle y^{2k}(1+e_{b_0}q(y,s_0))-y^{2k}, f_{2k}  \right\rangle\\
		&=&\displaystyle I^{-\gamma}(s_0)
		\frac{p+i\delta}{p-1}
		\sum_{j=0}^{2k-1} C_j(\hat d_j, \check d_j)
		\left\langle  e_{b_0} y^{j+2k}, f_{2k}  \right\rangle

		+\frac{1+i\delta}{p-1}\left\langle y^{2k}, f_{2k}  \right\rangle,
	\end{array}
	\]
	where $C_j(\ hat d_j,\check d_j)\in \C$ and $\hat d_j$, $\check d_j$ are defined in the initial data \eqref{initial-data}.
	
	Arguing as in the proof of  Proposition 4.4 in (NLH submitted cite ??), we obtain for all $0\leq j\leq 2k-1$
	\[
	\left |\left\langle  e_{b_0} y^{j+2k}, f_{2k}  \right\rangle \right | \lesssim I^{-4k}(s_0)
	\]
	and we conclude that 
	
	\beqtn \label{partial_b_F2}
	\pa_b \mathcal{F}_2(s_0,\mu_0)=I^{-4k}(s_0)2^{4k}(2k)!\left (1+O(I^{-\gamma}(s_0))\right ).
	\eeqtn
	
	From \eqref{partial_theta_F1},  \eqref{partial_theta_F2} \eqref{partial_b_F1}, we obtain 
	\fi 
	By combining the previous expansions, we obtain 
	\beqtn\label{partial_mu_G}
	\textbf{Det}\left( \textbf{J}_{\vec \mu}[\mathcal{F}](s,\vec \mu) \right) \left. \right|_{(s,\vec \mu) = ( s_0, \vec \mu_0)}  =-2^{4k}(2k)!  I^{-4k}(s_0)\left (1+O(I^{-\gamma}(s_0))\right )\not =0,
	\eeqtn
	provided that $s_0 \ge s_{2,1}$. Thus, it is a direct consequence  of the implicit function theorem that  there exists 	
	$\vec \mu(s)=(\theta(s), b(s))\in C\left ([s_0,s_{\text{loc}}],\R^2\right ) \cap C^1\left ((s_0,s_{\text{loc}}),\R^2\right )$ with $s_{\text{loc}} \in (0,\bar s_1)$ such that 
	$$ \mathcal{F}(s,\vec \mu(s))  \equiv 0, \forall s \in [s_0, s_{\text{loc}}],$$
	and 
	$$  \frac{b_0}{2} < b(s) \le  2 b_0 \text{ and }     \theta_0 - \frac{1}{4} < \theta(s) < \theta_0 + \frac{1}{4}, $$
	by the continuity. 	In particular, $(q,b,\theta)(s) \in V_{ A,\gamma, b_0, \theta_0}(s) \quad  \forall s \in [s_0, s_{\text{loc}}] $. Finally, we get the conclusion of the proposition. 		
	\iffalse
	
	such that 
	$G(s,\mu(s))=0$ for all $s\in [s_0,s^*)$. Defining $q$ by \eqref{equation-q-1} gives a unique solution of coupled problem \eqref{equation-q-1}$\&$\eqref{condition-modulation} for all $s\in [s_0,s^*)$.....\fi 
\end{proof}
\begin{rem}[Propagation of the existence] By a similar argument to the proof of Proposition \ref{propo-existence-local},  one can  extend the existence of $(q, b, \theta)$   in the sense   that if the solution $(q, b, \theta)$ exists on $[s_0, \bar s]$ for some $\bar s > s_0$ and $ (q, b, \theta)(\bar s) \in V_{A,\gamma, b_0, \theta_0}(\bar s)$. Then, there exists $\epsilon >0$ such that the solution  $(q, b, \theta)$  to problem \eqref{equa-q}  \& \eqref{orthogonal-conditions-H-2k-H-0},  uniquely exists on $[s_0, \bar s + \epsilon]$. Indeed, on the one hand, using  \eqref{definition-q} and the fact  $M = \frac{2kp}{p-1}$, the solution $w $ of equation \eqref{equation-w} belongs to $L^\infty$, and $\nabla w$ does by the parabolic regularity. 
%Indeed, on the one hand,  thanks to the transformation \eqref{definition-q} and the fact  $M = \frac{2kp}{p-1}$, it immediately follows that the solution $w $ of equation \eqref{equation-w} belongs to $L^\infty$, and $\nabla w$ does by the parabolic regularity. 
Therefore, $q$ and $\nabla q$ belong to $L^\infty_{M}$. In addition to that, since \eqref{equation-w} is well-posed in $L^\infty$,   there exists $\epsilon >0$ small enough such that $w$ exists on $[\bar s, \bar s +\epsilon]$. On the other hand, arguing as in the proof of  Proposition \ref{propo-existence-local}, we obtain the existence  of $(q, b, \theta)$ on $[\bar s, \bar s + \epsilon]$. Moreover, we have   $ (q, b, \theta) \in V_{\tilde A, \gamma, \tilde b_0, \tilde \theta_0}(s)$ for all $s \in [s_0, \bar s + \epsilon]$ for some new parameters $\bar A\ge A, \tilde b_0  \ge b_0$ and $ \tilde \eta_0 \ge \eta_0$. 
\end{rem}

%\[\label{partial_b_F1}\pa_{b }\mathcal{F}_1(s_0,\mu_0)=
%\Im \left [\left\langle\frac{p+i\delta}{p-1} w(y,s_0) y^{2k}(f_{b_0})^{-1}-y^{2k}, f_{0}  \right\rangle\right ]
%-\delta\Re \left [\left\langle \frac{p+i\delta}{p-1} w(y,s_0)y^{2k}(f_{b_0})^{-1}-y^{2k}, f_{0}  \right\rangle\right ]
%\]

\subsection{Reduction to a finite dimensional problem}
By the definition of the shrinking set $V_{ A,\gamma, b_0, \theta_0}(s)$ given in Definition \ref{definition-shrinking-set}, it is sufficient to prove that there exists a unique global  solution $(q,b,\theta)$ on $[s_0, +\infty)$ for some sufficiently large $s_0 $ that satisfies 
$$ (q,b,\theta)(s) \in V_{ A,\gamma, b_0, \theta_0}(s), \forall s \ge s_0.$$
In particular, we show in this part that the control of  infinite problem is reduced to a finite dimensional one. {As an important step  to get the conclusion of  our result, we  first show the following  \textit{a priori} estimates.
	%involving to $q_j, q_-$ and $b$.}
\begin{prop}[A \textit{priori} estimates]\label{proposition-priori-estimates}
	\label{proposition-ode} Let $b_0 > 0$, $k \in \mathbb{N}, k \ge 2, \theta_0  \in \R$ and $A \ge 1$, then there exists $\gamma_{3}(k, b_0) > 0$ such that  for all $\gamma \in (0,\gamma_3)$, there exists $s_3(\gamma, b_0, \theta_0, A)$ such that for all $s_0 \ge s_3$,  the following property holds:  Assume  $(q,b,\theta)$ is a solution to problem \eqref{equa-q}  \& \eqref{orthogonal-conditions-H-2k-H-0}  that  $(q,b,\theta)(s) \in V_{A,\gamma, b_0,\theta_0}(s) $ for all $s\in[s_0, \bar s]$ for some $\bar s \geq s_0$, and  $\hat q_{2k}(s)=0$ for all $s\in [s_0,\bar s]$, then for all $s\in [\tau, \bar s], $ with $s_0 \le \tau \le \bar s  $. Then, 
	\begin{itemize}
		\item[$(i)$] (Smallness of the modulation parameter $\theta(s)$). It holds that for all $s \in [s_0, \bar s]$
		\begin{equation*}\label{estimat-derivative-b}
			\left| \theta'(s) \right| \leq C A^2I^{-2\gamma}(s) \text{ and }  \theta_0 -\frac{1}{8}   \le \theta(s) \le \theta_0 + \frac{1}{8}.
		\end{equation*}
		\item[$(ii)$] (Oscillation of  the modulation flow $b(s)$). It holds that for all $s \in [s_0, \bar s]$
		\begin{equation*}\label{estimat-derivative-b}
			\left| b'(s) \right| \leq C A^2I^{-2\gamma}(s)\mbox{ and }\frac 34b_0\leq b(s)\leq \frac 54b_0.
		\end{equation*}
		\item[(iii)] (ODEs of the finite modes). For all $j \in \{ 0,...,[M] \}$ and for all $s \in [s_0, \bar s]$,
		$$\left |\hat q_j'(s)-\left( 1-\frac{j}{2k}\right)\hat q_j(s) \right |\leq CA^2I^{-2\gamma}(s), $$
		$$\left |\check q_j'(s)+\frac{j}{2k}\check q_j(s) \right |\leq CI^{-2\gamma}(s),\text{ for all } j \in \{0,...,2k-1\}, $$
		and for all $j \in \{ 2k,...,[M]\}$ 
		$$\left |\check q_j'(s)+\frac{j}{2k}\check q_j(s) - \displaystyle\sum_{  \substack{ m, l \in \N, m\ge 1 \\
				2km +l =j}  } c_{m,l}(p)b^m(s) \check q_l(s) \right | \le CI^{-2\gamma}(s).$$
		\item[$(iv)$]  (Control of the infinite-dimensional part $q_-$): For $q_-=(1+i\delta)\hat q_-+i\check{q}_-$, we have 
		$$ 	|\hat q_-(s)|_s   \le  e^{-\frac{s - \tau}{p-1}} |\hat q_-(\tau)|_\tau +   C \left(  I^{-\frac{\bar p  +1 }{2}\gamma}(s) + e^{-\frac{s-\tau}{p-1}} I^{-\frac{\bar p +1}{2}\gamma}(\tau) \right), $$
		and 
		$$ |\check q_-(s)|_s   \le  e^{-\frac{s - \tau}{p-1}} |\check q_-(\tau)|_\tau +  C \left(  I^{-\gamma}(s) + e^{-\frac{s-\tau}{p-1}} I^{-\gamma}(\tau) \right), $$
		where $\bar p = \min(p,2)$.
	\end{itemize}
\end{prop}
\begin{proof}[Proof of Proposition \ref{proposition-ode}]
	This result plays an important role in the proof of Theorem \ref{Theorem-principal}. For the reader's convenience,   we  put the complete proof of Proposition \ref{proposition-ode} in Section   \ref{section-priori-estimates}.\end{proof}
%. In addition to that, the proof {is} long computation which is technical. 
Consequently, we have the following result.
\begin{prop}[Reduction to a finite dimensional problem]\label{propositionn-transversality}
	Let  $ k \in \mathbb{N}, k \ge 2, b_0 >0, \theta_0 \in \R $ and $A \ge 1$, then there exists $\gamma_4(b_0, \theta_0)$ such that for all $ \gamma \in (0,\gamma_4)$, there exists $s_4(A, b_0, \gamma)$ such that for all $s_0 \ge s_4$,  the following  property holds:  Assume that $(q,b, \theta)$ is a solution to 
	problem \eqref{equa-q}  \& \eqref{orthogonal-conditions-H-2k-H-0} in accordance with  {initial data  $(q,b, \theta)(s_0) = (\psi(d_0,...,d_{2k-1},s_0),b_0, \theta_0)$ where  $\psi(d_0,...,d_{2k-1}),s_0)$ defined as in \eqref{initial-data} with} $ \max_{0 \le i \le 2k-1} |d_i| \le 2 $;   and $(q,b)(s)\in V_{ A,\gamma, b_0, \theta_0}(s)$ for all $s \in [s_0, \bar s]$ for some $\bar s > s_0$ that  $(q,b)( \bar s) \in \partial V_{A, \gamma, b_0, \theta_0}( \bar s)$, then the following properties are valid:
	\begin{itemize}
		\item[(i)] \textbf{(Reduction to some finite number of modes)}: Consider $\hat q_0,..., \hat q_{2k-1}$ be projections of $q$ corresponding to    \eqref{decomp3} then, we have 
		$$\left (\hat q_0,..,\hat q_{2k-1}\right )(\bar s) \in \partial \mathcal{V}(\bar s),$$
		where $\mathcal{V}(\bar s)=[-I^{-\gamma}(\bar s),I^{-\gamma}(\bar s) ]^{2k}$ and $I(s)$ is given by \eqref{defi-I-s} .\\
		%\textcolor{blue}{ \item[(iii)] %(\textbf{Transversality}): There %exists $\nu_0 >$ such that 
			%   $$(q,b)(\bar s +\nu) \notin \mathcal{V}_{\delta, b_0}(\bar s +\nu), \forall \nu \in (0, \nu_0).$$
			\item[(ii)]  \textbf{(Transverse crossing)} There exists $m\in\{0,..,2k-1\}$ and $\omega \in \{-1,1\}$ such that
			{	\[\omega \hat q_m(\bar s)=I^{-\gamma}(\bar s)\mbox{ and }\omega \frac{d \hat q_m}{ds} \left. \right|_{s=\bar s}>0.\]}
		\end{itemize}
	\end{prop}  
	
	\begin{rem}
		In (ii) of Proposition \ref{propositionn-transversality}, we show that the solution $q(s)$ crosses the boundary $\partial V_{ A,\gamma, b_0, \theta_0}(s)$ {at $\bar s$} with positive speed. In other words,  all points on {$\pa V_{ A,\gamma, b_0, \theta_0}( \bar s)$} are strict exit points in the sense of \cite[Chapter 2]{Conbook78}. 
	\end{rem}
	\begin{proof} Let us start the proof of  Proposition \ref{propositionn-transversality} assuming Proposition \ref{proposition-ode}. First, we consider $\gamma \le \gamma_3$ and $s_0 \ge s_3$ such that Proposition \ref{proposition-ode} holds.\\

		\noindent 
		- \textit{Proof of item (i):}  Let $L \gg 1$,  which will be fixed at the end of the proof.  According to item  \textit{(i)}  in Proposition \ref{proposition-ode}, it is sufficient to  show that   there exists $s_4(A,b_0, \theta_0, M, L)$ such that for all $s_0 \ge s_4$,  and for all $s \in [s_0, \bar s]$ the following are valid
		\begin{equation}\label{improve-hat-q-j-ge-2k+1}
			\left|   \hat q_j(s)    \right|	 \le \frac{1}{2} I^{-\gamma}(s), \forall j \in \{ 2k+1,...,[M] \}  \quad (\text{note that } q_{2k} \equiv 0),
		\end{equation}
		\begin{equation}\label{improve-check-q-j-ge-2k+1}
			\left|   \check{q}_j(s)    \right|	 \le \frac{1}{L^\frac{1}{j +4}} I^{-\gamma}(s), \forall j \in \{ 1,...,[M] \}  {\quad (\text{note that } \check q_{0} \equiv 0),}
		\end{equation}
		and 
		\begin{equation}\label{improve-q_-}
			\left|\hat q_-(s)  \right|_s \le  \frac{1}{2} I^{-\gamma}(s), \quad \left|\check q_-(s)  \right|_s \le  \frac{A}{2} I^{-\gamma}(s).
		\end{equation}\\
		
		+ Proof for \eqref{improve-hat-q-j-ge-2k+1}: Let $j \in \{ 2k+1,...,[M] \}$. Using the first estimate of  item \textit{(ii)} in Proposition  \ref{proposition-ode}, we find that 
		{	\begin{equation}\label{q-j-prime-minus-q_j-1-2k}
				\frac{d}{ds}\left[ \hat q_j(s)  \pm \frac{1}{2} I^{-\gamma}(s)  \right] =  \left( 1 - \frac{j}{2k} \right) \hat q_j(s)  \pm \frac{\gamma}{2} \left( \frac{1}{2k} - \frac{1}{2} \right) I^{-\gamma}(s) + O(I^{-2 \gamma}(s)).
		\end{equation}}%proof-check-q-j
	Using \eqref{intial-data-check-q} again, we have  $\hat q_j(s_0) =0$ which ensures  $\hat q_j(s_0) \in \left( -\frac{1}{2}I^{-\gamma}(s_0), \frac{1}{2} I^{-\gamma}(s_0) \right) $. Now, we  rely on  \eqref{q-j-prime-minus-q_j-1-2k} to conclude that 
			$$  \hat q_j(s) \in \left[  -\frac{1}{2}I^{-\gamma}(s), \frac{1}{2} I^{-\gamma}(s) \right], \forall s \in [s_0, \bar s], $$
		which concludes \eqref{improve-hat-q-j-ge-2k+1}. Assume for contradiction that there exist $s^* \in (s_0, \bar s)$ and $\omega \in \{ -1,1\}$ such that 
		\begin{equation*}
			\hat q_j(s^*)  =  
			\omega \frac{1}{2} I^{-\gamma}(s^*), \text{ and }     \hat q_j(\tau) \in \left(-\frac{1}{2} I^{-\gamma}(\tau), \frac{1}{2} I^{-\gamma}(\tau) \right), \forall \tau \in [s_0, s^*).
		\end{equation*}
		Without loss of generality, we may assume that $\omega = 1$. Then, we deduce from \eqref{q-j-prime-minus-q_j-1-2k} and the fact $ 1 - \frac{j}{2k} <0$  that
		 \begin{equation*}
		 \left.	\frac{d}{ds}\left[ \hat q_j(s)   -  \frac{1}{2} I^{-\gamma}(s)  \right] \right|_{s =s^*} =  \left( 1 - \frac{j}{2k} \right) \frac{1}{2} I^{-\gamma}(s^*)    + \gamma \left(  \frac{1}{2}-\frac{1}{2k} \right) I^{-\gamma}(s^*) + O(I^{-2 \gamma}(s*)) <0,
		 \end{equation*}
		provided that $\gamma \le \tilde \gamma_{4,1}$ and $s^* \ge \tilde s_{4,1}(\gamma)$. Thus,  it contradicts the assumption. Finally, we get the conclusion for  \eqref{improve-hat-q-j-ge-2k+1}.

		+ Proof for \eqref{improve-check-q-j-ge-2k+1}:  We also rely on the argument from the previous proof. However, it is much more complicated that we need to combine an induction argument due to the recurrent relation \eqref{recurrent-relation-check-q-j}.  First, we consider $j \in \{ 1,...,2k-1\}$,  and $L \ge 1$. Note that $\check q_j(s_0) =0$, thanks to \eqref{intial-data-check-q}. We also argue by contradiction that is  similar to  \eqref{improve-hat-q-j-ge-2k+1}. Indeed, we assume that there also exists $s^* \in (s_0, \bar s)$  and $\omega \in \{ -1, 1\}$ such that 
			\begin{equation}\label{contradict-assume-s-check-q-j}
			\check q_j(s^*)  =  
			\omega \frac{1}{L^\frac{1}{j+4}} I^{-\gamma}(s^*), \text{ and }     \check q_j(\tau) \in \left(-\frac{1}{L^\frac{1}{j+4}} I^{-\gamma}(\tau), \frac{1}{L^\frac{1}{j+4}} I^{-\gamma}(\tau) \right), \forall \tau \in [s_0, s^*).
		\end{equation}
	Without loss of generality, we may assume that $\omega =1$.   By the second estimate  of item \textit{(ii)} in Proposition  \ref{proposition-ode}, we find that
		\begin{align*}
			\left.	\frac{d}{ds}\left[ \check q_j(s)   -  \frac{1}{L^\frac{1}{j+4}} I^{-\gamma}(s)  \right] \right|_{s =s^*} =   - \frac{j}{2k}\frac{1}{L^\frac{1}{j+4}} I^{-\gamma}(s^*)    + \gamma \left(  \frac{1}{2}-\frac{1}{2k} \right) \frac{1}{L^\frac{1}{j+4}}I^{-\gamma}(s^*) + O(I^{-2 \gamma}(s*)) < 0,
		\end{align*}
		provided that $\gamma \le \tilde \gamma_{4,2}$ and $s^* \ge \tilde s_{4,2}(\gamma)$. Thus, it contradicts to the assumption, and the conclusion of 
	\eqref{improve-check-q-j-ge-2k+1} for all $j \in \{ 1,...,2k-1\}$. Next,  we consider $ j \in \{2k+1,...,[M] \}$.  We emphasize that we start the proof by induction from $j=2k+1$. By the same  contradiction  argument, we also assume that \eqref{contradict-assume-s-check-q-j} holds. Then, it is sufficient to check 
	$$ 	\left.	\frac{d}{ds}\left[ \check q_j(s)   -  \frac{1}{L^\frac{1}{j+4}} I^{-\gamma}(s)  \right] \right|_{s =s^*}  <0.$$
 	By the  third estimate of item \textit{(ii)} of Proposition \ref{proposition-ode} that we find
	\begin{align*}
	\left.	\frac{d}{ds}\left[ \check q_j(s)   -  \frac{1}{L^{j+4}} I^{-\gamma}(s)  \right] \right|_{s =s^*} &=   - \frac{j}{2k}\frac{1}{L^\frac{1}{j+4}} I^{-\gamma}(s^*)    + \gamma \left(  \frac{1}{2}-\frac{1}{2k} \right) \frac{1}{L^\frac{1}{j+4}}I^{-\gamma}(s^*) \\
	&+  \sum_{\substack{ 2km +l =j\\
	1 \le m \in \N, l \in \N  }}   c_{m,l}(p) b^{m} \check q_l(s^*) +  O(I^{-2 \gamma}(s^*))\\
& \le  - \frac{j}{2k}\frac{1}{L^\frac{1}{j+4}} I^{-\gamma}(s^*)    + \gamma \left(  \frac{1}{2}-\frac{1}{2k} \right) \frac{1}{L^\frac{1}{j+4}}I^{-\gamma}(s^*) \\
&+  C\sum_{\substack{ 2km +l =j\\
		1 \le m \in \N, l \in \N  }}    \frac{1}{L^\frac{1}{l+4}}I^{-\gamma}(s^*) +  CI^{-2 \gamma}(s*).
	\end{align*}
	We remark that for $m \ge 1, l \ge 0$ such that $2km + l =j $ which implies 
	$ j -  l = 2km \ge 2k \ge 4 $. Henceforth,
	$$ \frac{1}{L^{\frac{1}{l+4}}} \ll  \frac{1}{L^{\frac{1}{j+4}}} \text{ as } L \text{ large enough}.$$
	So,  there exist  $ \gamma_{4,1} $ and $L_{4,1} \ge 1$ such that for all $\gamma \le \gamma_{4,1} $ and $L \ge L_{4,1}$, there exists $ s_{4,1}(L,\gamma)$ such that 	
			$$ 	\left.	\frac{d}{ds}\left[ \check q_j(s)   -  \frac{1}{L^\frac{1}{j+4}} I^{-\gamma}(s)  \right] \right|_{s =s^*}  < 0,$$
	Thus, we get the conclusion of \eqref{improve-check-q-j-ge-2k+1} for $j =2k+1$. Next, we continue with $j=2k+2$ using the same argument. \\
 Arguing in a similar fashion, we can  continue to the final case $j=[M]$. There exists $L_{4,[M]} \ge 1 $ and $\gamma_{4,[M]}$ such that for all $L \ge L_{4,[M]}$ and $\gamma \le \gamma_{4,[M]}$, there exists $s_{4,[M]}(L, \gamma)$ such that for all $s^* \ge s_{4,[M]}$, it holds that
		\begin{align*}
			\left.	\frac{d}{ds}\left[ \check q_{[M]}(s)   -  \frac{1}{L^{[M]+4}} I^{-\gamma}(s)  \right] \right|_{s =s^*} &=   - \frac{[M]}{2k}\frac{1}{L^\frac{1}{[M]+4}} I^{-\gamma}(s^*)    + \gamma \left(  \frac{1}{2}-\frac{1}{2k} \right) \frac{1}{L^\frac{1}{[M]+4}}I^{-\gamma}(s^*) \\
			&+  \sum_{\substack{ 2km +l =[M]\\
					1 \le m \in \N, l \in \N  }}   c_{m,l}(p) b^{m} \check q_l(s^*) +  O(I^{-2 \gamma}(s*)) <0,
		\end{align*}
		where $s^*$ determined as  in the previous steps. Finally, fixing $L= \max\{1,L_{4,j}, j \in \{2k+1,...,[M]\}  \}$, $\gamma \le \min\{\tilde \gamma_{4,1}, \tilde \gamma_{4,2}, \gamma_{4,j}, j \in \{2k+1,...,[M]\}\}$, and $s_0 \ge  \max\{ \tilde s_{4,1}, \tilde s_{4,1}, s_{4,j},j \in \{2k+1,...,[M] \} $ we get the complete conclusion for \eqref{improve-check-q-j-ge-2k+1}. 
		
		\medskip 
		
		+ For $\hat q_-$ and $\check q_-$.  We  remark that  the proof for $\hat q_-$ is the same as for $\check q_-$. For that reason,  we only give the proof for the estimate involving to $\check q_-$ as in \eqref{improve-q_-}. For technical reasons, we divide the proof into two cases where $ s - s_0 \le s_0 $ and $s - s_0 \ge s_0 $.      According to the first case, we apply item (iii) in Proposition \ref{proposition-ode} with $\tau = s_0$ and note that $| \check q_- (s_0)|_{s_0} = 0$   that 
		\begin{eqnarray*}
			\left| \check q_-(s)  \right|_s  &\le& C \left( I^{-\gamma}(s) + e^{-\frac{s-s_0}{p-1}} I^{-\gamma}(s_0) \right) \\
			&\le & C \left( I^{-\gamma}(s) + e^{ \left(\frac{\gamma}{2}\left(1-\frac{1}{k}\right)-\frac{1}{p-1}\right)(s-s_0)} I^{-\gamma}(s) \right)    \\
			&\le &  \frac{A}{4} I^{-\gamma}(s) +  \frac{A}{4} I^{-\gamma}(s) \le        \frac{A}{2} I^{-\gamma}(s),
		\end{eqnarray*}
		provided that $A \ge A_{4,1}, \gamma \le \tilde \gamma_{4,3}$. For  the second case, we use item (iii) again  with $\tau = s - s_0 \in [s_0,\bar s]$, and we obtain
		\begin{eqnarray*}
			\left| \check q_-(s)  \right|_s & \le &  e^{-\frac{s_0}{p-1}} I^{-\gamma}(\tau)  +  {C\left( I^{-\gamma}(s) + e^{-\frac{s_0}{p-1}} I^{-\gamma}(\tau)   \right)}\\
			&  \le  &  C  \left( 1 + 2e^{\left( \frac{\gamma}{2}(1-\frac{1}{k}) -\frac{1}{p-1} \right)s_0}  \right)  I^{-\gamma}(s) \le \frac{A}{2} I^{-\gamma}(s),
		\end{eqnarray*}
	provided that $A \ge A_{4,2}$ and $\gamma \le \tilde \gamma_{4,4}$. 	Thus, \eqref{improve-q_-}  completely follows. Finally, using  the definition of $V_{A,\gamma, b_0, \theta_0}(s)$; the fact $(q,b, \theta)(\bar s) \in \partial V_{ A,\gamma, b_0, \theta_0}(\bar s)$; estimates \eqref{improve-hat-q-j-ge-2k+1}, \eqref{improve-check-q-j-ge-2k+1} and  \eqref{improve-q_-}; and   item (ii) of Proposition \ref{propositionn-transversality}, we get the conclusion of  item (i).
		\iffalse
		We define $\sigma=s-\tau$ and take $s_0=-\ln T\ge 3 \sigma,\;(\mbox{ that is } T\le e^{-3\sigma})$ so that for all $\tau \geq s_0$ and $s\in [\tau, \tau +\sigma]$, we have

		\[\tau\leq s\leq \tau +\sigma\leq \tau+\frac{1}{3} s_0\leq \frac{4}{3} \tau.\]
		
		From (i) of proposition Proposition \ref{proposition-ode}, we can write for all $2k\leq j\leq [M] $
		
		\[\left |\left ( e^{-(1-\frac{j}{2k})t}q_j(t)\right )' \right |\leq  C e^{-(1-\frac{j}{2k})t}I^{-2\delta}(t),\]
		We consider that $\tau \leq s\leq \frac{4}{3}\tau$.   Integrating the last inequality between $\tau$ and $s$, we obtain
		%\[|\int_{\tau}^{s}
		\[
		\begin{array}{lll}
			|q_j(s)|&\leq &e^{-(1-\frac{j}{2k})(\tau-s)}  q_j(\tau)+ C(s-\tau)e^{(1-\frac{j}{2k})s} I^{-2\delta}(\tau)\\
			&\leq &e^{(1-\frac{j}{2k})(s-\tau)}  q_j(\tau)+ C(s-\tau)e^{(1-\frac{j}{2k})s} I^{-\frac{4}{3}\delta}(s),
		\end{array}
		\]
		
		There exists $\tilde s_1=\max \{s_j,\;5\le j\le 9\}$, such that if $s\geq \tilde s_1$, then we can easily derive
		\[|q(s)|\leq Ce^{(1-\frac{j}{2k})(s-\tau)} I^{-\frac{4}{3}\delta}(s)+I^{-\frac{7}{6}\delta}(s)< \frac{1}{2}I^{-\delta}(s).\]
		In a similar fashion, exists $\tilde s_2=\max \{s_j,\;9\le j\le 13\}$, for all $s\geq \tilde s_2$, we obtain 
		\[|q_-(s)|_s< \frac{1}{2}I^{-\delta}(s).\]
		%For the $q_2k$, we have by (ii) of Proposition \ref{proposition-ode}
		Thus we finish the prove of (ii) of Proposition \ref{propositionn-transversality}.\\
		\fi
		
		\noindent  
		- \textit{Proof of item (ii)}: As a consequence of item (i), there exist $m \in \{0,..2k-1\}$ and $\omega \in \{- 1,1\}$ such that {$\hat q_m( \bar s)=\omega I^{-\gamma}(\bar s)$.} By item (ii) in  Proposition \ref{proposition-ode}, we find that for $\gamma>0$
		\[\omega \hat q_m'(\bar s)\geq \left(1-\frac{m}{2k} \right)\omega \hat q_m(\bar s)-CI^{-2\gamma}(\bar s)\geq \left (\left(1-\frac{m}{2k}\right)I^{-\gamma}(\bar s)- CI^{-2\gamma}(\bar s)\right )>0,\]
		provided that $s \ge \tilde s_{4,4}$. Finally, we conclude the proof of item (ii), and the conclusion of the proposition follows.
	\end{proof} 

	\subsection{Conclusion  of Theorem \ref{Theorem-principal}}\label{proof-Theorem-1}
	
	In this section,  we aim to give the complete proof to Theorem  \ref{Theorem-principal}  by using a topological \textit{shooting argument} and  the results in  Proposition \ref{propositionn-transversality}:
	\begin{proof}[The proof of Theorem \ref{Theorem-principal}]
		First, 	we aim to prove that there exist
$(\hat d_0,..,\hat d_{2k-1}) \in \mathbb{D}_{s_0}$ such that problem \eqref{equa-q} $\&$
\eqref{orthogonal-conditions-H-2k-H-0} with initial data  
$ (\psi(\hat d_0,...,\hat d_{2k-1},s_0), b_0, \theta_0)$
and  $\psi(\hat d_0,...,\hat d_{2k-1},s_0)$ defined as in  \eqref{initial-data},  has a  solution $(q_{\hat d_0,..,\hat d_{2k-1}},b, \theta)(\cdot)$ is defined for all $s \in  [s_0,\infty)$, satisfying
$$  (q_{\hat d_0,..,\hat d_{2k-1}},b, \theta)(s) \in  V_{ A,\gamma_0, b_0, \theta_0}(s) \text{ for all } s \ge s_0.  $$ 
\iffalse 
\beqtn
\|q(s)\|_{L^\infty_M} \leq C I^{-\delta}(s), |b(s)-b^*|\leq C I^{-2\gamma_0}(s), \mbox{ and }  
\label{goal-of-the proof}
\eeqtn
for some $b^*>0$,  $\gamma_0>0$ and  $s_0$ large enough. 
\fi

\medskip 
Let us now begin the proof of existence. We consider $b_0 >0, \theta_0 \in \R, \gamma_0 = \min(\gamma_1, \gamma_{3},\gamma_4)$ and $ s_0 \ge \max(s_1, s_3,s_4)$  such that the results in  Lemma \ref{lemma-initial-data} and  Propositions  \ref{proposition-ode}-\ref{propositionn-transversality} hold,  and we also denote  $T= e^{-s_0} > 0$ which is small,  since $s_0$ is large enough.
We proceed by contradiction, 
%from (ii) of Lemma \ref{initial-data-new}, 
we assume that for all $(\hat d_0,...,\hat d_{2k-1}) \in \mathbb{D}_{s_0}$ (the set is defined in Lemma \ref{lemma-initial-data}), there exists $s_*=s_*(\hat d_0,..,\hat d_{2k-1}) < +\infty$ such that 
\begin{equation*}
	\begin{array}{ll}
		q_{\hat d_0,..,\hat d_{2k-1}}(s)\in V_{ A,\gamma_0, b_0, \theta_0}(s), & \forall s\in [s_0, s_*],  \\
		q_{\hat d_0,..,\hat d_{2k-1}}(s_*)\in \pa V_{ A,\gamma_0, b_0, \theta_0}(s_*).&
	\end{array}
\end{equation*}
By using item  (i) in  Proposition \ref{propositionn-transversality}, it follows that $(\hat q_0,..,\hat q_{2k-1})(s_*)  \in \pa V(s_*)$. Consequently, we can  define the mapping $\Phi$ given by,
\[%\Phi:
\begin{array}{lll}
\Phi:&	\mathbb{D}_{s_0}\to \pa [-1,1]^{2k}&\\
	&(\hat d_0,..\hat d_{2k-1})\to I^{\gamma_0}(s_*)(\hat q_0,..,\hat q_{2k-1})(s_*).
\end{array}
\]
Moreover, one can show that such a $\Psi$ exists and satisfies the following properties:
\begin{itemize}
	\item[$(i)$] $\Phi$ is continuous from  $\mathbb{D}_{s_0}$ to  $\pa [-1,1]^{2k}$. This property  is based on   the continuity  in time of $q$,  on the one hand. On the other hand,  the continuity of $s_*$ in $(\hat d_0,...,\hat d_{2k-1})$ which  is a direct consequence  of the transversality   from item (ii) in  Proposition \ref{propositionn-transversality}.
	\item[(ii)] It holds that $\Phi \left. \right|_{\partial \mathbb{D}_{s_0}}$ has nonzero degree. Indeed, for all $(\hat d_0,...,\hat d_{2k-1})  \in \partial \mathbb{D}_{s_0}$, we derive from item (i) of Lemma \ref{lemma-initial-data}  that $s_*(\hat d_0,...,\hat d_{2k-1})  =s_0$ and 
	$$ \text{ deg}\left( \Phi \left. \right|_{\partial \mathbb{D}_{s_0}} \right) \ne 0.  $$
\end{itemize}
A contradiction follows by the Index Theory, which excludes the existence of such $\Phi$. Henceforth,  there exists $(\hat d_0,...,\hat d_{2k-1}) \in \mathbb{D}_{s_0}$ such that  $(q_{\hat d_0,...,\hat d_{2k-1}},b, \theta)(s) \in  V_{ A,\gamma_0, b_0, \theta_0}(s), \forall s \ge s_0$. 
\iffalse 

and by (iii) of Proposition \ref{propositionn-transversality}, $\Phi$ is continuous.\\
In the following we will prove that $\Phi$ has nonzero degree, which mean by the degree theory (Wazewski's principle) that for all $s\in [s_0, \infty )$ $q(s)$ remains in $V_{ A,\gamma_0, b_0, \theta_0}(s)$, which is a contradiction with the Exit Proposition.\\
Indeed Using Lemma \ref{lemma-initial-data}, and the fact that $q(-\ln T)=\psi_{\hat d_0,..,\hat d_{2k-1}}$, we see that when $(\hat d_0,..,\hat d_{2k-1})$ is on the boundary of the quadrilateral $\mathbb{D}_T$, $\hat q_0,..,\hat q_{2k-1}(-\ln T)\in \pa [-I^{-2\gamma_0}(s),I^{-2\gamma_0}(s)]^{2k}$ and $q(-\ln T)\in V_{ A,\gamma_0, b_0, \theta_0}(-\ln T)$ with strict inequalities for the other components.\\
By Proposition \ref{propositionn-transversality}, $q(s)$ leaves $V_{A, \gamma_0, b_0, \theta_0}$ at $s_0=-\ln T$, hence $s_*=-\ln T$.\\  
Using (ii) of Proposition \ref{propositionn-transversality}, we get that the restriction of $\Phi$ on the boundary of $\mathbb{D}_{s_0}$ is of degree $1$, which means by the shooting method that for all $s\in [s_0, \infty )$ $q(s)$ remains in $V_{ A,\gamma_0, b_0, \theta_0}(s)$, which is a contradiction.\\
\fi
We conclude that there exist   $(\hat d_0,..,\hat d_{2k-1})\in \mathbb{D}_{s_0}$ and $(b,\theta)(\cdot) \in (C^1(-\ln T,+\infty))^2$ such that for all $s\geq -\ln T = s_0$, $(q_{\hat d_0,..,\hat d_{2k-1}}, b, \theta)(s) \in V_{ A,\gamma_0, b_0, \theta_0}(s)$ for all $s \ge s_0$. In particular, we   obtain 
\beqtn\label{estimation-linftyM-q}
\left \|\frac{q(\cdot, s)}{1+|y|^M}\right\|_{L^\infty}\leq C I^{-\gamma_0}(s), \forall s \ge s_0, \text{ where } C=C(A).
\eeqtn

		\bigskip
		Now, we give the proof of  Theorem \ref{Theorem-principal}.\\

		\noindent 
		\textit{- The proof of  Theorem  \ref{Theorem-principal}:}
Since $(q,b,\theta)(s) \in V_{A, \gamma_0, ,b_0, \theta_0}(s)$ for all $s \ge s_0$, we derive from item (i) in Proposition \ref{proposition-ode} 
\begin{equation}\label{esti-b-prime-tau}
	|b'(\tau)| \le C e^{-{\gamma_0 \tau}\left(1 -\frac{1}{k} \right)}, \forall \tau \ge s_0,
\end{equation}
which ensures that 
$$ \int_{s_0}^\infty |b'(\tau)| d\tau < \infty. $$
Therefore, we define $b^* = b(s_0) + \displaystyle\int_{s_0}^\infty b'(\tau) d\tau$, then it follows that
$$  b(s)  \to  b^*  \text{  as  } s \to +\infty.  $$
Using  	\eqref{esti-b-prime-tau}  again that we find
$$  |b(s) - b_*| \le  C e^{ -{s \gamma_0} \left(1 -\frac{1}{k} \right) }, \forall s \ge s_0.$$
Now, for simplicity,  we  still  write $b(t) = b(s) $    with   $s = - \ln(T-t)$.  Then, 
using \eqref{change-variable2} and \eqref{estimat-derivative-b}, we get  
$$ \left| b(t) - b^* \right| \le C(T-t)^{{\gamma_0}\left( 1 -\frac{1}{k}\right)}, \forall t \in [0,T).  $$
In particular, by taking $t =0$, we obtain \eqref{estimate-b-t-b-*}. By  \eqref{definition-q}, the fact that $M=\frac{2kp}{p-1}$, and the following estimate  
{	\[|f_be_b|=|f_b|^p\leq C(1+|y|^{2k})^{-\frac{p}{p-1}} \le  C(1+|y|^{M})^{-1}, \]}
we find that  
\[\|w(s)- f_{b(s)}\|_{L^\infty}=\|f_{b(s)}e_bq(s)\|_{L^\infty} \leq C I^{-\gamma_0}(s), \text{ with } I(s) = e^{\frac{s}{2}\left(1 -\frac{1}{k}\right)},\]
%where  we use the notation $b(t) = b(s)$ 
and we get
$$ \left\|(T-t)^{\frac{1+i\delta}{p-1}} u(\cdot, t) - f_{b(t)} \left( \frac{|\cdot|}{(T-t)^{\frac{1}{2k}}} \right)\right\|_{L^\infty} \le C (T-t)^{\frac{\gamma_0}{2}(1-\frac{1}{k})}, \forall t \in (0,T), T = e^{-s_0},   $$ 
Let us now introduce the function $F(a)=(p-1+\left (ab(t)+(1-a)b^*\right )y^{2k})^{-\frac{1+i\delta}{p-1}}$, where $a\in [0,1]$. We can easily derive
\[|F'(a)|\leq C\left| b(t) - b^* \right| \le C(T-t)^{{\gamma_0}\left( 1 -\frac{1}{k}\right)},\]
then, we obtain
\[ \|f_{b(t)}-f_{b^*}\|_{L^\infty}\leq C(T-t)^{{\gamma_0}\left( 1 -\frac{1}{k}\right)}, \]
which concludes \eqref{theorem-intermediate}  and 
the proof of Theorem \ref{Theorem-principal}.  
	\end{proof}

	%%%%%%%%%%%OUT
	%\medskip
%	\textcolor{red}{Ky: It is better to be removed\\
	%Additionally, we end this part by completing  the proof of Corollary \ref{corollar}. 
	%\begin{proof}[Proof of Corollary \ref{corollar}]Let us introduce the function $F(a)=(p-1+\left (ab(t)+(1-a)b^*\right )y^{2k})^{-\frac{1+i\delta}{p-1}}$, where $a\in [0,1]$. We can easily derive
	%	\[|F'(a)|\leq C(b_0)\left| b(t) - b^* \right| \le C(T-t)^{{\gamma}\left( 1 -\frac{1}{k}\right)},\]
		%then, we obtain
		%\[|f_{b(t)}-f_{b^*}|=|\int_0^1 F'(a)da|\leq C(T-t)^{{\gamma}\left( 1 -\frac{1}{k}\right)}, \]
	%	which concludes the proof of Corollary \ref{corollar}.
	%\end{proof}}
	%%%%%%%%%%%%%%%OUT
	\section{A priori estimates}\label{section-priori-estimates}
	In this section, we aim to give the complete proof of  Proposition  \ref{proposition-priori-estimates}. We divide the section into three parts:    
	\begin{itemize}
		\item Subsection \ref{subsection-finite-projection}: we project the terms of the equation \eqref{equa-q} onto $\{ \hat H_n, \check H_n, \text{ for } n =0,...,[M]\}$ for all $n \in \{0,.., [M]\}$ so that we can derive  \textit{a priori} estimates for   $\hat q_n$ and $ \check q_n$ for all $n =0,...,[M]$.
		%The result is showed in Section \ref{subsection-finite-projection} below.
		\item Subsection \ref{subsection-esti-P--term-equation-s}: we  provide  the estimation of  the infinite parts, i.e., $P_-$ (defined in  \eqref{defi-P-})  of  the terms in equation \eqref{equa-q}. 
		%The result is showed in Section \ref{subsection-esti-P--term-equation-s}.
		\item Subsection \ref{subsection-conclusion-propo-priori-estimate}: we use the  estimates established in the previous parts to conclude the proof of Proposition \ref{proposition-ode}. 
	\end{itemize}
	\subsection{The finite dimensional part $q_+$ }\label{subsection-finite-projection}
	In this part,  we project  equation \eqref{equa-q} onto the ``eigenfunctions'' of the operator $\mathcal{L}_{\delta, s}$. 
	
	\medskip 
	
	Let $A \ge 1, b_0 >0, \theta_0 \in \R, \gamma >0$ and $s_0 \ge 1$,  and we  assume  that
	$ (q, b, \theta)(s)   $ lies in $  V_{ A, \gamma, b_0, \theta_0}(s)$ for all $s \in [s_0, \bar s]$ for some $\bar s > s_0$.

	\bigskip
	\textbf{ + First term: $\pa_s q$.}
	\begin{lemma}\label{lemma-pn-pas-q}For all $0\leq n\leq [M]$, we  obtain
		\beqtn
		\begin{array}{lll}
			\hat P_{n,M}\left(\pa_s q \right)&=&\partial_s\hat q_n+ (1-\frac 1k)(n+1)(n+2)I^{-2}(s)\hat q_{n+2},\\[0.4cm]
			\check P_{n,M}\left(\pa_s q\right)&=&\pa_s \check q_n+ (1-\frac 1k)(n+1)(n+2)I^{-2}(s) \check q_{n+2}.
		\end{array}
		\eeqtn
	\end{lemma}
	\begin{proof}
		The result follows from  Definition \ref{Definition-complex-decomposition} and the  identities from        \cite[Lemma 5.1]{DNZCPAA24}. 
	\end{proof}
	\textbf{+ Second term: $\mathcal{L}_{\delta,s} q$.}
	\begin{lemma}\label{lemma-P-nmathcal-L-delta-s} For all $0\leq n\leq [M]$, we have 
		\beqtn
		\begin{array}{ll}
			\hat P_{n}\left(\mathcal{L}_{\delta,s} q\right)&=\dsp (1-\frac{n}{2k})\hat q_n+ (1-\frac 1k)(n+1)(n+2)I^{-2}(s)\hat q_{n+2},\\[0.2cm]
			\check P_{n}\left(\mathcal{L}_{\delta,s} q\right)&=\dsp -\frac{n}{2k} \check q_n+ (1-\frac 1k)(n+1)(n+2)I^{-2}(s)\check q_{n+2}.
		\end{array}
		\eeqtn
	\end{lemma}
	\begin{proof}
		It follows by   Definition \ref{Definition-complex-decomposition} and    \cite[Lemma 5.2]{DNZCPAA24}
	\end{proof}
	\textbf{ + Third term  $B(q)=\dsp \frac{b'(s)}{p-1}y^{2k}\left (1+i\delta+(p+i\delta)e_b q\right )$.} \\
	
	\begin{lemma}\label{lemma-P-n-B} There exists $s_5(A)\geq 1$, such that for all $s\geq s_5$,  it holds that  
		\begin{itemize}
			\item [a)] For $0\leq n\leq [M]$, $n\not =2k$,  
			\[\left |\hat P_n(B(q))\right |\leq C |b'(s)|  I^{-\gamma}(s),\]
			\item[b)]For $0\leq n\leq [M]$, 
			\[\left | \check P_n(B(q))\right |\leq C |b'(s)|I^{-\gamma}(s),\]
			\item[c)] For $n =2k$, 
			\[\left | \hat P_{2k}(B(q))-\frac{1}{p-1}b'(s)\right |\leq C |b'(s)|I^{-\gamma}(s).\]
		\end{itemize}
	\end{lemma}
	\begin{proof}
		First, we use the orthogonality of $\{ H_n, n \ge 1\}$ to   derive 
		\begin{equation}
			\int y^{2k} H_n(y,s) \rho_s(y) dy    = 	\left\{    
			\begin{array}{rcl}
				0   &\text{ if } &     n > 2k,\\
				\|H_{2k}(s)\|^{2}_{L^2_{\rho_s}} + O(I^{-2k -2 })   &\text{ if }&   n =2k,\\
				O(I^{-2k-2})       &\text{ if }&  n  < 2k.
			\end{array}
			\right.
		\end{equation}
		Thanks to  \eqref{scalar-product-hm},  \eqref{defi-Q-n} and Definition \ref{Definition-complex-decomposition}, we observe that   the result of  this lemma will   follow the  following   			\begin{equation}\label{key}
			\left| \int y^{2k} H_n(s) e_{b(s)} q(s) \rho_s dy  \right| \le  C I^{-\gamma -2n}(s) , \text{ for all } n \le [M]. 
		\end{equation}
		For the  proof of  \eqref{key}, we rely   on \eqref{decomp1} that  we  get 
		\begin{align*}
			\int y^{2k}e_b qH_n\rho_s dy &= \dsp \sum_{j \le [M]}  Q_j(q) \int y^{2k}e_bH_n(s) H_j(s) \rho_s dy  
			+ \int  y^{2k}e_b q_-H_n\rho_s dy,\\
			&= E_1 + E_2, \text{ respectively}.
		\end{align*}
	
	\medskip 
	+ Estimate for $E_1$:   Let $L \in \N$ which will be fixed at the of the proof. First, we have the following identity
	\begin{equation}\label{identity-e-b}
	e_b(y) = (p-1)^{-1} \left[ \sum_{l=0}^L \left( -\frac{b}{p-1} y^{2k}\right)^l +  \left( -\frac{b}{p-1} y^{2k}\right)^{L+1} e_b(y)  \right]. 
	\end{equation} 	
So, we get the following expression 
\begin{align*}
y^{2k} e_b(y) =(p-1)^{-1} \left[ \sum_{l=0}^L \left( -\frac{b}{p-1} \right)^ly^{2k(l+1)} +  \left( -\frac{b}{p-1} \right)^{L+1}y^{2k(L+2)} e_b(y)  \right]    
\end{align*}	
Note that for all $l \le L,$ and $ j \le [M]$, 
$$ \left| \int y^{2k(l+1)} H_j H_n \rho_{s} dy  \right| \le C \|H_n\|^2_{L^2_{\rho_s}} \le CI^{-2n}(s).     $$
	On the other hand, we have 
	\begin{align*}
	\left|\int y^{2k(L+2)} H_jH_n \rho_{s} dy \right|  \le CI^{-2k(L+2) -n - j}(s) \le C I^{-2n }(s),
	\end{align*}
	provided that $ L \ge L_{5,1}(M)$ and $s \ge s_{5,1}$.  Note that $|Q_n(q)| \le CI^{-\gamma}(s)$. Fixing $L =L_{5,1}$, and   taking the sum in $j =0,...,[M]$, we   find that
	$$ |E_1| =  \left|  \sum_{j \le [M]}  Q_j(q) \int y^{2k}H_n(s) H_j(s) \rho_s dy \right|  \le C I^{-\gamma -2n}(s).    $$

	\medskip 
	+ Estimate for $E_2$:   we apply \eqref{esti-rought-pointwise-q--}  that we obtain 
		\begin{equation*}
			\left|   \int    y^{2k}e_b(y)  H_n (s)q_-(y,s)  \rho_s dy \right| \le CAI^{-\gamma} (s) \int |y|^{2k} (I^{-M}(s) + |y|^M) (I^{-n}(s)  + |y|^{n}) \rho_s(y)dy. 
		\end{equation*}
		Now, we aim to prove that 	 for all $n \le  [M]$
		\begin{align}
			\left| 	\int |y|^{2k} (I^{-M}(s) + |y|^M) (I^{-n} (s) + |y|^{n}) \rho_s(y)dy \right| \le C I^{-2k - M - n}(s)\label{integral-polynomial-M-m-2k}.
		\end{align}
		By changing variable  $  z = I(s) y  $  and $\rho_s$'s definition in  \eqref{defi-rho-s},  we  get
		\begin{align*}
			\left| 	\int |y|^{2k} (I^{-M}(s) + |y|^M) (I^{-n} (s) + |y|^{n}) \rho_s(y)dy \right|   \le CI^{-2k - M - n}(s)  \int |z|^{2k} (1 + |z|)^{M + n }  e^{- \frac{|z|^2}{4}} dz, 
		\end{align*}
		which concludes \eqref{integral-polynomial-M-m-2k}. In particular, it also follows
		\begin{align*}
		|E_2|  =	\left|   \int    y^{2k}e_b(y)  H_n (s)q_-  \rho_s dy \right| \le   CA I^{-\gamma -2k -2n}(s) \le CI^{-\gamma - 2n }(s),
		\end{align*}
		provided that $s   \ge s_{5,2}(A)$. Thus, by fixing $s_5(A)= \max(s_{5,1},s_{5,2})$ that  \eqref{key}  holds true for all $s \ge s_5$. Finally, we finish the proof of the lemma.
		\iffalse 
		By the definition of $e_b$ we have $|y^{2k}e_b|\leq Cb^{-1} $. Using the bound on $b$ given by the shrinking set definition given by Definition \ref{definition-shrinking-set},  the first term can be bounded by
		\beqtn
		\int |H_n||q_-|\rho_s \leq C I(s)^{2n}\left \|\frac{q_-}{1+|y|^{2k+1}} \right \|_{L^\infty}. 
		\eeqtn
		
		Now, we deal with the second term. We only focus on the terms involving $\hat H_j$ since the estimates are the same for the terms involving $\check H_j$. First, we write
		\[\int \hat H_j H_ny^{2k}e_b\rho_s=(1+i\delta)\left (\displaystyle \int_{|y|\leq 1}H_jH_ny^{2k}e_b\rho_s+ \int_{|y|\leq 1}H_jH_ny^{2k}e_b\rho_s
		\right ),\]
		
		arguing as in the proof of Lemma 5.4 from \cite{DNZCPAA24}, we claim that
		\[|\int \hat H_j H_ny^{2k}e_b\rho_s|\leq CI(s)^{-2\delta}, \]
		\fi 
	\end{proof}
	
	%\[
	%\begin{array}{lll}
	%\dsp \int BH_n(y) \rho (y) dy& =&\frac{1+i\delta}{p-1}b'(s)I(s)^{-2k}\dsp\int y^{2k}H_n\rho (y) dy\\
	%    &&+ \frac{p+i\delta}{p-1}b'(s)I(s)^{-2k}\dsp\int y^{2k}e_b qH_n\rho (y) dy
	%\end{array}
	
	%\]

	\medskip
	
	\textbf{+ Fourth term: $T(q)=-i\theta'(s)(e_b^{-1}+q)=-i\theta'(s)(p-1+by^{2k}+q)$.}\\

	\begin{lemma}[Projection of $T(q)$ on $\hat H_n$ and $\check H_n$]\label{lemma-estimation-Pn-T} For $0\leq n\leq [M] $, the projection on $\tilde H_n$ is given by
		\[
		\begin{array}{lll}
			\hat P_n (T(q))&=&\left \{ \begin{array}{ll}
				%-\theta'\left ((p-1) +(1+\delta^2)\hat q_0\right ) &\mbox{ if %} n=0,\\
				\theta'\check q_{2k} &\mbox{ if } n=2k,\\
				-\theta'\left ((1+\delta^2)\hat q_n-\delta \check q_n \right )&\mbox{else.}
			\end{array}
			\right .\\
		\end{array}
		\]
		and the projection on $\check H_n$ is given by
		\[
		\begin{array}{lll}
			\check P_n (T(q))
			% &=&-\theta'\left (\check P_n(i(p-1))+b\check P_n(iy^{2k})+\check P_n(iq)\right ),\\
			&=&\left \{ \begin{array}{ll}
				-\theta'\left ((p-1) +(1+\delta^2)\hat q_0\right ) &\mbox{ if } n=0,\\
				-\theta'\left ( b-\delta\check q_{2k}\right ) &\mbox{ if } n=2k,\\
				-\theta'\left ((1+\delta^2)\hat q_n-\delta \check q_n \right )&\mbox{else,}
			\end{array}
			\right .\\
		\end{array}
		\]
	where $\hat{P}_n$ and $\check P_n$ are defined as in \eqref{decomp3}.	
		
		%	there exists $\gamma_?$ $s_?\geq 1$, such that for all $0<\gamma< \gamma_0$, $s\geq s_?$, if $q\in V_{A,\gamma,b_0}(s)$, then
		%\begin{itemize}
		%	\item [a)] For $0\leq n\leq [M]$,
		%	\[\left |\hat P_n(T)\right |\leq CI^{-2\delta}(s),\]
		%	\item[b)]For $0\leq n\leq [M]$,  $n\not =0$,
		%	\[\left |\check P_n(T)\right |\leq CI^{-2\delta}(s),\]
		
		%				\item[c)]
		%				\[\left |\check P_{0}(T)+ (p-1)\theta'(s)\right |\leq CI^{-2\delta}(s),\]
		%				\item[d)]
		%				\[\left |\check P_{2k}(T)+\theta'(s)\right |\leq CI^{-2\delta}(s).\]
		
	\end{lemma}
	\begin{proof}
		First,  by using  \eqref{complex-projection-linear}  and  \eqref{decomp3}   we can   express as follows
		\begin{align*}
			\hat P_n(T(q))  & = - \theta'(s)\left (\hat P_n(i(p-1))+b\hat P_n(iy^{2k})+\hat P_n(iq)\right ), \\[0.4cm]
			\check{P}_n (T(q)) & = - \theta'(s)\left (\check P_n(i(p-1))+b\check P_n(iy^{2k})+\check P_n(iq)\right ). 
		\end{align*}
		Note that the first two terms are easy to derive, since they do not include $q$. Now, we aim to  compute the  projections $\hat P_n(iq)$ and  $\check P_n(iq)$.   We observe from \eqref{defi-hat-H-and-check-H-n} that
		\[\begin{array}{lll}
			i\hat H_n&=&\delta \hat H_n-\check H_n,\\
			i\check H_n&=&(1+\delta^2 )\hat H_n-\delta \check H_n,
		\end{array}
		\]
	which yields 
		$$
		\begin{array}{lll}
			\hat P_n(iq) &=& \delta \hat q_n-\check q_n,  \\[0.2cm]
			\check P_n(iq)&=&(1+\delta^2)\hat q_n-\delta \check q_n. 
		\end{array}
		$$
		Using the  condition \eqref{orthogonal-conditions-H-2k-H-0},  we get
		\[
		\begin{array}{lll}
			\hat P_n (T)&=&
			-\theta'\left (\hat P_n(i(p-1))+b\hat P_n(iy^{2k})+\hat P_n(iq)\right )\\[0.2cm]
			&=&\left \{ \begin{array}{ll}
				%-\theta'\left ((p-1) +(1+\delta^2)\hat q_0\right ) &\mbox{ if %} n=0,\\
				\theta'\check q_{2k} &\mbox{ if } n=2k,\\[0.2cm]
				-\theta'\left ((1+\delta^2)\hat q_n-\delta \check q_n \right )&\mbox{else.}
			\end{array}
			\right .\\
		\end{array}
		\]

		\[
		\begin{array}{lll}
			\check P_n (T)&=&-\theta'\left (\check P_n(i(p-1))+b\check P_n(iy^{2k})+
			\check P_n(iq)\right )\\[0.2cm]
			&=&\left \{ \begin{array}{ll}
				-\theta'\left ((p-1) +(1+\delta^2)\hat q_0\right ) &\mbox{ if } n=0,\\[0.2cm]
				-\theta'\left ( b-\delta\check q_{2k}\right ) &\mbox{ if } n=2k,\\[0.2cm]
				-\theta'\left ((1+\delta^2)\hat q_n-\delta \check q_n \right )&\mbox{else.}
			\end{array}
			\right .
		\end{array}
		\]
		Finally, we finish the proof of the lemma. 
	\end{proof}

	\medskip
	
	\textbf{+ Fifth term: $N(q)=(1+i\delta)\left ( |1+e_bq|^{p-1}(1+e_bq)-1-2e_b\Re q-\frac{p-1}{2} e_b q-\frac{p-3}{2}e_b \bar q\right )$.}\\
	
	\begin{lemma}\label{lemma-estimation-N}
		Let  $K \ge 1,  K \in \N, b_0 >0,  \frac{b_0}{2} \le  b \le 2b_0 $ and $\left(\sup_{|y| \le 1} e_b(y)\right) |q| \le \frac{1}{2}$. Then,  we have 
		\[\displaystyle 
		\left |N(q)  -\sum_{\begin{array}{rcl}
				0\leq j, m \leq K\\
				2\leq j+m\leq K\end{array}}
		B_{j,m}(y) q^j\bar q^m+\tilde B_{j,m}(y) q^j\bar q^m
		\right |\leq C\left (|q|^{K+1}\right ),
		\]
		where $B_{j,m}(y)$ is an even polynomial of degree less or equal to $2kK$ and the rest $\tilde B_{j,m}$ satisfies
		\[ |\tilde B_{j,m}(y)|\leq C(1+|y|^{2k(K+1)}), \forall |y| \le 1.  \]
		Moreover, 
		\[  |B_{j,m}(y)|+ |\tilde B_{j,m}(y)| \leq C, \forall |y|\le 1, \]
		and  
		\[|N(q)|\leq   C(1 +|e_bq(y)|^p), \forall y \in \R,\]
		where $C$ depends only on $b_0,p,$ and $K$.
	\end{lemma}
	\begin{proof}
		We observe that  in the region $|y| \leq  1$,   $|e_b(y) q| \le \frac{1}{2}$. Then, the results are mainly based on  the  Taylor expansion in terms of $e_bq$ and $e_b\bar q$. Indeed,  we
		can  estimate $N(q)$ as follows
		\[\dsp\left |N(q)-
		\sum_{
			\begin{array}{rcl}
				0\leq j, m \leq K\\
				2\leq j+m\leq K\end{array}}
		c_{j,m} e_b^{j+m} q^j \bar q^m \right |\leq \tilde C|e_b q|^{K+1} \le C|q|^{K+1}, \forall |y| \le 1.\]
		Using \eqref{defi-e-b} and Taylor expansion, we find that %expand-N-q
		\beqtn
	e_b^{j+m}(y)= \sum_{l=0}^{K}A_{j,m,l} y^{2kl} + \tilde A_{j,m,K}(y) \text{ and }  \left| \tilde A_{j,m,K}(y) \right|   \leq C|y|^{2k(K+1)}, \forall |y| \le 1.
		\eeqtn

	Let $B_{j,m}(y)=\sum_{l=0}^{K} A_{j,m,l} y^{2kl} $ and $\tilde B_{j,m} = \tilde A_{j,m, K}(y)$, and we obtain the following estimate
		\[\displaystyle 
		\left |N(q)-\sum_{\begin{array}{rcl}
				0\leq j, m \leq K\\
				2\leq j+m\leq K\end{array}}
		\left( B_{j,m}(y) +\tilde B_{j,m}(y)\right) q^j\bar q^m
		\right |\leq C|q|^{K+1},	
		\]
		where $B_{j,m}$ is an even polynomial of degree less or equal to $2kK$ and the rest $\tilde B_{j,m}(y)$ satisfies
		\[   |\tilde B_{j,m}(y)|\leq C|y|^{2k(K+1)}.\]
		Moreover, 
		\[ |B_{j,m}(y)|+| \tilde B_{j,m}(y)|\leq C.  \]
		On the other hand, for $|y| \ge 1$, we also obtain 
		$$ |N(q)| \le C(1 + |e_b q|^{p}), $$
		since $p >1$. Finally,  we get the conclusion of the lemma.
	\end{proof}
	As a consequence of  Lemma \ref{lemma-estimation-N}, we have the following result.
	\begin{lemma}\label{lemma-estimation-Pn-N}
		There exits $s_6(A, b_0, \theta_0) \ge 1$	 such that  for all $ s_0 \ge s_6$ and    $s \in [s_0, \bar s]$, for some $\bar s>s_0$, and  $0 \leq n\leq [M]$, it holds that 
		\beqtn\label{estimation-Pn-N}
		\left |\hat P_n(N(q)) \right|  +  \left| \check P_n(N(q))\right |\leq C A^2 I^{-2\gamma}(s).
		\eeqtn

	\end{lemma}%proof-P-N-q
	
	\begin{proof}
		From \eqref{defi-Q-n} and \eqref{decomp3}, it is sufficient  to prove 
\begin{equation}\label{projec-N-q-H-n}
\left| \int N(q) H_n(y,s) \rho_s(y) dy    \right|  \le  CA^2I^{-2\gamma -2n}(s). 
\end{equation}
Now, we decompose the integral as follows 
\begin{align*}
\int  N(q)H_n(y,s) \rho_s(y) dy &=  \int_{|y|\le  1} H_n(y,s) N(q)\rho_s(y) dy+\int_{|y| \ge  1} H_n(y,s) N(q)\rho_s(y) dy\\
& = N_1 + N_2,  \text{ respectively}.
\end{align*}

\medskip 
+ Estimate for $N_1:$ let $K \in K$ (fixed later), and by using Lemma \ref{lemma-estimation-N}, we deduce that
		\beqtn \label{estimation-Pn_interior-region}
		\begin{array}{cc}
			\displaystyle \left |N_1 -\int_{|y|\le 1} H_n \rho_s(y)\sum_{\begin{array}{rcl}
					0\leq j, m \leq K\\
					2\leq j+m\leq K\end{array}}
			B_{j,m}(y) q^j\bar q^m+\tilde B_{j,m}(y) q^j\bar q^m
			\right |\\
			\displaystyle\leq C \int_{|y|\le  1} |H_n||e_b q|^{K+1}\rho_s.
		\end{array}
		\eeqtn
Note that for all $|y|
\le 1$, we have  $|H_n(y,s)|\leq C$, and 
by the definition of the shrinking set $V_{A,\gamma, b_0, \theta_0}(s)$ that we have 
\[|e_b q(y,s)|^{K+1}\leq CI^{-\gamma(K+1)}(s), \forall |y| \le 1.\]
Therefore, 
\begin{align*}
\int_{|y|\le 1} |H_n(y,s)||e_b(y) q(y,s)|^{K+1}\rho_s(y) dy &\leq C I^{-\gamma(K+1)}\int_{|y|\le 1} \rho_s(y) dy \\
&\leq CI^{-\gamma(K+1)}(s) \le C I^{-2\gamma -2n}(s),
\end{align*}%proof-N-q-2
provided that $K \ge K_{6,1}$.  Besides that, sine $j+m \ge 2$, we find that 
\begin{align*}
\left| \int_{|y| \le 1} \tilde B_{j,m}(y) q^{j} \bar q^m H_n \rho_s dr  \right|  \le \int_{|y| \le 1}  |y|^{2k(K+1)} I^{-2\gamma} (I^{-n}(s) + |y|^{n}) \rho_s(y) dy 
\end{align*}
By changing variable $z = I(s) y$, we obtain
\begin{align*}
\left| \int_{|y| \le 1} \tilde B_{j,m}(y) q^{j} \bar q^m H_n \rho_s dy  \right|  \le   \tilde CI^{-2k(K+1) -2\gamma -n}(s)\int_\R |z|^{2k(K+1)} (1 + |z|^{n}) e^{-\frac{|z|^2}{4}} dz \le C I^{-2\gamma -2n}(s),   
\end{align*}
provided that $K \ge K_{6,2}$. Now, let 
\begin{equation*}
N_{j,m} = \int_{|y| \le 1} B_{j,m}(y) H_{n} q^j \bar q^m \rho_s(y) dy,  
\end{equation*}
we aim to prove that for $j+m \ge 2$
\begin{equation}
|N_{j,m}| \le CI^{-2\gamma -2n}(s).
\end{equation}
Indeed, we rely on \eqref{decom-q=q+plus-q-} and the bounds \eqref{bound-q-ygep1-2} to derive that
$$  \left| q^j \bar q^m - q_+^j \bar q_+^m  \right| \le C A^2I^{-2\gamma}(I^{-M}(s) + |y|^M), \forall |y| \le 1, $$
provided that $ s \ge s_{6,2}(A)$. In addition, we have 
\begin{align*}
 \int_{|y| \le 1} I^{-2\gamma}(s)|B_{j,m}(y)|(I^{-M}(s) + |y|^M)|H_n(y,s)| \rho_s(y) dy \le  C I^{-2\gamma -M -n}(s) \le CI^{-2\gamma -2n}(s). 
\end{align*}
Therefore, 
$$  |N_{j,m} - \int_{|y| \le 1} B_{j,m}(y)q_+^j\bar q^m H_n \rho dy      |  \le CI^{-2\gamma -2n}(s),  $$
where $q_+$ denoted by 
$$ q_+ = \sum_{l \le [M]} Q_l(q) H_l \text{ with }  |Q_\ell| \le CI^{-\gamma}(s).$$
	On the one hand, by changing variable $z = I(s) y$, we deduce that
	\begin{align*}
 \left| \int_{|y| \ge 1} B_{j,m}(y)q_+^j\bar q^m_+ H_n \rho dy \right|  &\le C I^{-(j+m)\gamma -n}(s)  \int_{|z| \ge I(s)} (1 + |z|^{4kK +n})  e^{-\frac{|z|^2}{4}} dz \\
&\le  C   I^{-(j+m)\gamma -n} (s)I^{-\frac{I^2(s)}{8}}(s) \le CI^{-2\gamma -2n}(s),
	\end{align*}
provided that $s \ge s_{6,3}$.  On the other hand, we can  express 
as follows 
\begin{align*}
 B_{j,m}(y) q_+^j \bar q_+^m = \sum_{l=0}^{2kK +(j+m)[M]} b_l(s) H_l, 
\end{align*}
with $|b_l(s)| \le CI^{-2\gamma}(s)$, and  the fact that
\begin{align*}
\left| \int H_l(y,s) H_n(y,s) \rho_s(y) dy  \right| \le C I^{-2n}(s),
\end{align*} 	
	which implies
	\begin{equation*}
	|N_{j,m}| \le CI^{-2\gamma -2n }(s).
	\end{equation*}
	Combining the related bounds, we can conclude
	$$ |N_1| \le CI^{-2\gamma -2n}.$$
	
	\medskip 
	+ Estimate for $N_2$: From the last bound in Lemma \ref{lemma-estimation-N}, we can estimate $N_2$ as follows
	\begin{align*}
	\left |E_2\right| \le C \int_{|y| \ge 1}  (I^{-n}(s) + |y|^n) (1 + I^{-p\gamma }(1 + |y|^{Mp}) ) \rho_s dy,    
	\end{align*}
and by changing variable $z = I(s)y $ that we implies
\begin{align*}
|E_2| \le C \int_{|z| \ge I(s)} (1 + |z|^{n+Mp}) e^{-\frac{|z|^2}{4}} dz \le C es^{-\frac{I(s)}{8}} \le CI^{-2\gamma - 2n}(s)   
\end{align*}	
	Finally, we conclude \eqref{projec-N-q-H-n} and we finish the proof  of the lemma.	
	\end{proof}

	\textbf{+ Sixth term: $D_s(q)=-\displaystyle \left( \frac{p+i\delta}{p-1} \right)4kby^{2k-1} I^{-2}(s) e_b\nabla q$.}\\
	\begin{lemma}\label{lemma-estimation-Pn-Ds} For $\gamma \in (0, \frac{1}{2})$, there exists $s_7(A, b_0, \theta_0, \gamma) \ge 1$ such that   for all $s \ge s_7$, it holds that
		\[|\hat P_n(D_s(q))|\leq CI^{-2\gamma}(s)\mbox{ and }|\check P_n(D_s(q))|\leq CI^{-2\gamma}(s).\]
	\end{lemma}
	\begin{proof}
		By \eqref{defi-Q-n} and Definition \ref{Definition-complex-decomposition}, it is sufficient to prove 
		\begin{equation*}
			\left|   Q_n(D_s(q))(s) \right| \le C I^{-2\gamma}(s)  \text{ for all } s \in [s_0, \bar s],
		\end{equation*}		  
		provided that $s \ge s_7(A, b_0, \theta_0, \gamma)$, for some $s_7$  large enough.  In particular,  from   \eqref{norm-q-2},  we only need  to prove
		\begin{align*}
			\left| \displaystyle \int D_s(q) H_n \rho_s dy    \right|  \le C I^{-2\gamma -2n}(s).
		\end{align*}
		By  $D_s(q)$'s   definition, we can express  		
		\[\displaystyle \int D_s(q) H_n \rho_s dy=-\frac{p+i\delta}{p-1}4 bk I^{-2}(s)\int y^{2k-1} e_b \nabla q H_n \rho_s dy. \]
		So, it follows by  integration by parts which 	is quite the same as  \cite[Lemma 5.5]{DNZCPAA24}.  We kindly refer the reader to check the details. Finally, we finish the proof of the lemma. 
	\end{proof}
	\medskip
	
	\textbf{+ Seventh term: $R_s(q)= I^{-2}(s)y^{2k-2}\left( \alpha_1+\alpha_2y^{2k}e_b+ \left (\alpha_3+\alpha_4 y^{2k}e_b\right )q\right)$.}\\
	
	\begin{lemma}\label{lemma-estimation-Pn-Rs} For $\gamma \in (0,\frac{1}{2})$, there exists $s_8(A,b_0, \theta_0, \gamma) \ge 1$    such that for all $s_0 \ge s_5, $ and $ s \in [s_0, \bar s]$, and $n \le [M]$ 
		\[|\hat P_n(R_s(q))|\leq CI^{-2\gamma}(s)\mbox{ and }|\check P_n(R_s(q))|
		\leq CI^{-2\gamma}(s).\]
	\end{lemma}
	\begin{proof}
The proof of the lemma is  the same to the  proof of Lemma \ref{lemma-P-n-B}.
	\end{proof}
	
	\medskip
	
	\textbf{Eighth term: $V(q)=\left ((p-1)e_b-1\right )\left [(1+i\delta)\Re q -q\right ]$.}\\
	
	\begin{lemma}\label{lemma-estimation-Pn-V} Let $\gamma >0$, then there exists $s_9(A, b_0) \ge 1$  such that for all    $s_0 \ge s_9, $ and $s \in [s_0, \bar s]$, we have     
		\begin{equation*}
			\hat P_n(V(q(s)))=0, \forall n \in \{ 0,...,[M]\},
		\end{equation*}
		and 
		\begin{equation}\label{recurrent-relation-check-q-j}
		\check P_n(V(q(s))) = \left\{  \begin{array}{rcl}
		 && 0  \text{ if } n < 2k, \\[0.3cm]
		 && \displaystyle\sum_{  \substack{ j, l \in \N, j\ge 1 \\
		 2kj +l =n}  } c_{j,l}(p)b^j \check q_l(s)  \\[0.8cm]
	 & & + O(I^{-\gamma -2 }(s)) \text{ if } n \in \{ 2k,...,[M]\}.
	\end{array}
		    \right.  
		\end{equation}	
	\end{lemma}
	\begin{proof} First,  we recall from item \textit{(ii)} of Lemma \ref{lemma-complex-decomposition} that
		\[\begin{array}{lll}
			V(q) 
			= i\left (1 - (p-1)e_b\right ) \check q,
		\end{array}
		\]
		where $\check q$ admits the following composition (see in \eqref{decomposition-q})
		\begin{equation}\label{express-check-q}
			\check q(y,s) = \sum_{n \le [M]}   \check q_n (s)  H_n(y,s) +  \check q_-(y,s).  
		\end{equation}
		Thanks  to Definition  \ref{Definition-complex-decomposition}, it immediately  follows that 
		\[   \hat P_n(V(q))=0, \forall n\in \N.\] %esti-check-P-V-q
	It remains to 	estimate   $\check P_n(V(q))$. 	By \eqref{defi-Q-n} and \eqref{decomp3}, we can express 
	\begin{align*}
			\check P_n(V(q)) &= \|H_n\|^{-2}_{L^2_{\rho_s}}   \int_\R  (1 - (p-1)e_b)\check q(s) H_n \rho(s)  dy   \\
			& = \|H_n\|^{-2}_{L^2_{\rho_s}}   \int_\R  (1 - (p-1)e_b)\left[ \sum_{l=0}^{[M]} \check q_l(s) H_n(y,s)  + \check q_-(y,s) \right] H_n \rho(s)  dy
	\end{align*}
	Let $L\in \N$ fixed  later, we deduce from \eqref{identity-e-b} 
	that
		\[1 - (p-1) e_b =by^{2k}e_b= -\sum_{j=1}^{L}\left (-\frac{b}{p-1} y^{2k}\right )^{j}
		-\left (-\frac{b}{p-1} y^{2k}\right )^{L+1}e_b(y,s).\]
	Therefore, 
	$$ \check P_n(V(q)) =  \|H_n\|^{-2}_{L^2_{\rho_s}} \left(  E_1 + E_2 + E_3   \right),$$
	where

\begin{align*}
E_1 &= \int_{\R} \sum_{ \substack{1\le j \le L\\
0 \le l\le [M]}} \tilde c_j(p) b^j y^{2kj} \check q_l H_l(y,s) H_n(y,s) \rho_s(y) dy  \\
E_2 &= \sum_{j=1}^L \tilde c_j(p)b^j\int_\R  y^{2kj} \check q_-(y,s) H_n(y,s) \rho(y) dy  ,\\
E_3 &= - \left( \frac{b}{p-1} \right)^{L+1} \int_{\R} y^{2k(L+1)} \check q H_n(y,s) \rho_s(y) dy 
\end{align*}
and the constant  $\tilde c_j(p)$ depends only on $b$. 

\medskip 
\noindent 
- \textit{Estimate for $E_3$:}  From \eqref{express-check-q}, and by  the fact that $ (q,b,\theta)(s) \in V_{A, \gamma, b_0,\theta_0}(s)$ for all $s \in [s_0, \bar s]$, we find  
$$  |\check q(y,s)| \le C I^{-\gamma}(s) (1 + |y|^M), \forall y \in \R.$$
Hence, we estimate $E_3$ as follows
\begin{align*}
	|E_3| \le  C I^{-\gamma}(s) \int_{\R}   |y|^{2k(L+1)}(1 + |y|^M ) H_n(y,s) \rho_s(y) dy. 
\end{align*}
By changing variable $z = y I(s)$, we find that
\begin{align*}
	\left| E_3 \right| \le C I^{-\gamma - 2k(L +1) -n }(s) \int_\R |z|^{2k(L+1)}(1 + |z|^{M}) h_n(z) e^{-\frac{|z|^2}{4}} dz \le C I^{-\gamma -2- 2n }(s), 
\end{align*}
	provided that $L \ge L_{9,1}$ and $s_{9,1}(L, b_0)$.   
	
\medskip 
\noindent 
-\textit{ Estimate for $E_2$:} 	
	Using the fact  $ (q,b,\theta)(s) \in V_{A, \gamma, b_0,\theta_0}(s)$  again, we  can estimate as follows
	$$ | \check q(y,s)| \le A  I^{ -\gamma}(s) (I^{-M} + |y|^{M}), $$
	which concludes 
	\begin{align*}
		\left | E_2 \right |\leq C \int_{\R} y^{2k}(I^{-M}(s)+|y|^M)H_n\rho_s dy
	\end{align*}
Using the changing  variable $z=yI(s)$ 
	\[\begin{array}{lll}
		\displaystyle \int_{|y|\le 1} |y|^{2kj}(I^{-M}(s)+|y|^M)|H_n(y,s)|\rho_s dy&=&I^{-M-n-2k}(s)\displaystyle  \int_{\R} |z|^{2k} (1+|z|^M) |h_n(z)| e^{-\frac{|z|^2}{4}} dz .\\
		&\le & C I^{-M-n-2k}(s),
	\end{array}
	\]
	where $h_n$   defined as in \eqref{defi-harmite-H-m}. 
Thus,
$$  |E_2| \le  CI^{-\gamma -2 -2n}(s),$$
provided that $s \ge s_{9,1}(A)$.

\medskip 
\noindent 
-\textit{ Estimate for $E_1$:}  For each $l \in \{ 0,...,[M]\}, j \in \{ 1,...,L\}$, and by the changing  variable $z = y I(s)$,  we  deduce that 
\begin{align*}
\int_{\R} y^{2kj} H_l(y,s) H_n(y,s) \rho_s(y) dy = I^{-2kj - l - n } (s)\int_{\R } z^{2jk} h_l(z) h_n(z) \frac{e^{-\frac{|z|^2}{4}}}{\sqrt{4\pi}}  dz.
\end{align*}
We observe that
\\

  + if $ 2kj + l < n $, then the latter integral vanishes, since the orthogonality of $\{ h_m, m \ge 0\}$,\\
  
  + if $ 2k + l = n $, then  we have 
  $$ \int_{\R} y^{2kj} H_l(y,s) H_n(y,s) \rho_s(y) dy = I^{-2n}(s) d_{j,l}, \text{ for some } d_{j,l} >0, $$ 
	
	+ if $2kj + l > n$, then either  $\int_{\R } z^{2jk} h_l(z) h_n(z) \frac{e^{-\frac{|z|^2}{4}}}{\sqrt{4\pi}}  dz $ vanishes or $2kj + l -  n \ge 2 $. Thus, we get
	$$ \left| \int_{\R} y^{2kj} H_l(y,s) H_n(y,s) \rho_s(y) dy \right| \le C I^{-2n -2}(s).$$
Hence, we obtain 
\begin{align*}
   E_1   = \left\{    \begin{array}{rcl}
   	0 \text{ if }   n < 2k,\\
       \displaystyle\sum_{  \substack{ j, l \in \N, j\ge 1 \\
       		2kj +l =n}  } c_{j,l}(p)b^j \check q_l(s) \|H_n\|^{2}_{L^2_{\rho_s}} + \tilde E_1 \text{ if } n \ge 2k,
   \end{array}             \right. 
\end{align*}
where 
$$ |\tilde E_1| \le C I^{-\gamma -2 -2n }(s).$$

\medskip 
\noindent
Finally, we combine  all estimates for $E_1, E_2$ and $E_2$, then we obtain the  desired estimate for $\check P_n(V(q))$, and we finish the proof of the lemma.
	\end{proof}
	\subsection{Estimates of  the infinite dimensional parts of the terms in equation \eqref{equa-q}}\label{subsection-esti-P--term-equation-s}
	
	Let  $ (q, b, \theta)$ be a solution to the problem \eqref{equa-q} \& \eqref{orthogonal-conditions-H-2k-H-0} on $[s_0, \bar s]$  for some $ \bar s > 0$.  Under the assumption  $(q, \theta, b )(s) \in V_{ A,\gamma, b_0, \theta_0}(s)$ for $s\in [s_0,\bar s]$.  Then, we present the  following results.
	
	\medskip
	\noindent 
	\textbf{ + First term:  $P_-\left(\pa_s q \right)$} where $P_-(\cdot)$ defined as in \eqref{defi-P-}.
	\begin{lemma}\label{lemma-P--partial-s-q}
		For all $s \in [s_0,  \bar s]$, we have 
		\begin{equation}\label{P--partial-s-q}
			P_-\left ( \pa_s q\right )=\displaystyle \pa_s q_- -I^{-2}(s)\left (1-\frac 1k\right )\sum_{n=[M]-1}^{[M]}(n+1)(n+2)\left (\hat q_{n+2} \hat H_n + \check q_{n+2} \check H_{n}\right ),
		\end{equation}
		where   $I(s)$ given as in \eqref{defi-I-s}. Consequently,  
		\begin{align*}
			\mathscr{Q}_{\Re, \delta} \left( P_-(\partial_s q) \right)  & =  \partial_s \hat q_-  -I^{-2}(s)\left (1-\frac 1k\right )\sum_{n=[M]-1}^{[M]}(n+1)(n+2)\hat q_{n+2} H_n(s) ,\\
			\mathscr{Q}_{\Im, \delta} \left( P_-(\partial_s q) \right)  & =  \partial_s \check q_-  -I^{-2}(s)\left (1-\frac 1k\right )\sum_{n=[M]-1}^{[M]}(n+1)(n+2)\check q_{n+2} H_n(s).
		\end{align*}

	\end{lemma}
	
	\begin{proof} %First, \textcolor{blue}{we are based on} \eqref{defi-mathscr-Q-Re-Im}, \eqref{defi-hat-H-and-check-H-n}  and \eqref{defi-hat-q--check-q--}
		 {Using \eqref{defi-mathscr-Q-Re-Im}, \eqref{defi-hat-H-and-check-H-n}  and \eqref{defi-hat-q--check-q--}}, it is sufficient to  prove the  identity \eqref{P--partial-s-q}. By using  \eqref{defi-P-}, we can express
		\begin{align*}
			P_-(\partial_s q) &= \partial_s q - \sum_{ n=0}^{[M]} \left( \hat P_{n}(\pa_s q) \hat  H_n  + \check P_{n}(\pa_s q) \check H_n \right)  \\
			& = \partial_s q_-   + \sum_{n=0}^{[M]} \partial_s \left( \hat q_n \hat H_n  + \check q_n \check H_n \right)   -  \sum_{ n=0}^{[M]} \left( \hat P_{n}(\pa_s q) \hat  H_n  + \check P_{n}(\pa_s q) \check H_n \right).
		\end{align*}
		\iffalse

		Let us start by writing
		\[\begin{array}{lll}
			P_-(\pa_s q) -\pa_s q_- &=&- \left (\pa_s q-P_-(\pa_s q)\right )+\left (\pa_s q-\pa_s q_-\right )  \\
			&=& -\displaystyle \sum_{n=0}^{[M]}\tilde P_{n,M}(\pa_s q) \tilde H_
			n+P_{n,M}(\pa_s q) H_n 
			+\sum_{n=0}^{[M]}\pa_s \left (\tilde q_n \tilde H_
			n+q_n  H_n\right )
		\end{array}
		\] \fi 
		Using the fact that  $\partial_s H_n  \equiv  0$ if $n =0$ or $1$, and  for all $n \ge 2$ 
		\[
		\pa_s H_n(s)=n(n-1)\left (1-\frac 1k\right ) I^{-2}(s)
		H_{2}(y,s).\]
		Hence, we deduce from  Lemma \ref{lemma-pn-pas-q}   that
		\[
		\begin{array}{lll}
			\hat  P_{n}\left(\pa_s q \right) &=&\partial_s \hat  q_n+ (1-\frac 1k)(n+1)(n+2)I^{-2}(s)\hat  q_{n+2},\\[0.2cm]
			\check	P_{n}\left(\pa_s q\right) & = & \pa_s \check q_n+ (1-\frac 1k)(n+1)(n+2)I^{-2}(s) \check q_{n+2},
		\end{array}
		\]
		which implies 
		\[P_-(\pa_s q)= \pa_s q_-- I^{-2}(s)\left (1-\frac 1k\right )\sum_{n=[M]-1}^{[M]}(n+1)(n+2)\left (\hat  q_{n+2}\hat  H_{n}+ \check q_{n+2} \check H_{n}\right ),\]
		Finally, \eqref{P--partial-s-q} follows and we conclude the proof of  the lemma. 
	\end{proof}
	
	\medskip
	\textbf{ + Second term $\mathcal L_{\delta,s} q$.}
	\begin{lemma}
		For all $s\in [s_0,\bar s]$, it holds that
		\begin{equation}\label{P--mathcal-L-delta-s}
			P_-\left (\mathcal L_{\delta,s} q\right )=\mathcal{L}_{\delta,s} q_- -I^{-2}(s)\left(1-\frac 1k\right)\sum_{n=[M]-1}^{[M]}(n+1)(n+2)\left (\hat q_{n+2} \hat H_n+ \check q_{n+2} \check H_n\right ).
		\end{equation}
		Consequently,  
		\begin{align*}
			\mathscr{Q}_{\Re, \delta} \left( P_-(\mathcal L_{\delta,s} q) \right)  & =  \mathcal{L}_s \hat q_-  -I^{-2}(s)\left (1-\frac 1k\right )\sum_{n=[M]-1}^{[M]}(n+1)(n+2)\hat q_{n+2} H_n(s) ,\\
			\mathscr{Q}_{\Im, \delta} \left( P_-(\mathcal L_{\delta,s} q) \right)  & =  \mathcal{L}_{0,s}  \check q_-  -I^{-2}(s)\left (1-\frac 1k\right )\sum_{n=[M]-1}^{[M]}(n+1)(n+2)\check q_{n+2} H_n(s),
		\end{align*}
		where $ \mathcal{L}_{0,s}$ and $ \mathcal{L}_s$ defined as in \eqref{defi-mathcal-L-0} and \eqref{defi-mathcal-L-s}, respectively. 
	\end{lemma}
	
	\begin{proof}
		First, we are based  on \eqref{defi-mathscr-Q-Re-Im}  to  derive 
		$$  \mathscr{Q}_{\Re, \delta} \left(\mathscr{L}_{\delta, s} q_- \right)   = \mathcal{L}_s  \hat q_-  \text{ and } \mathscr{Q}_{\Im, \delta} \left(\mathscr{L}_{\delta, s} q_- \right)  = \mathcal{L}_{0,s} \check q_-.  $$
	It is quite similar to proof of Lemma \ref{lemma-P--partial-s-q}  that we only need  to prove \eqref{P--mathcal-L-delta-s}. Indeed,  by using \eqref{defi-P-}, we can express as follows
		\begin{align*}
			P_-(\mathcal{L}_{\delta,  s})  -  \mathcal{L}_{\delta,s} q_- = \left( \sum_{n=0}^{[M]} \left(  \hat P_{n}(\mathcal{L}_{\delta,s} q)  \hat  H_n - \hat  q_n \mathcal{L}_{\delta, s} \hat H_n\right)  +  \sum_{n=0}^{[M]} \left(   \check P_{n}(\mathcal{L}_{\delta,s} q)  \check H_n - \check P_{n}(\mathcal{L}_{\delta,s} q)  \check H_n  \right)  \right). 
		\end{align*}
		Combining   identities in  \eqref{spectLtilde} and \eqref{spectLtilden=0-1}  with Lemma  \ref{lemma-P-nmathcal-L-delta-s}, we conclude 
		\begin{align*}
			\sum_{n=0}^{[M]} \left(  \hat P_{n}(\mathcal{L}_{\delta,s} q)  \hat  H_n - \hat  q_n \mathcal{L}_{\delta, s} \hat H_n\right)  &= - I^{-2}(s) \left( 1 - \frac{1}{k} \right)  \sum_{n=[M]-1}^{[M]} (n+1)(n+2) \hat  q_{n+2} \hat  H_n  \\
			\sum_{n=0}^{[M]} \left(   \check P_{n}(\mathcal{L}_{\delta,s} q)  \check H_n - \check P_{n}(\mathcal{L}_{\delta,s} q)  \check H_n  \right)    &= - I^{-2}(s) \left( 1 - \frac{1}{k} \right)  \sum_{n=[M]-1}^{[M]} (n+1)(n+2) , \check q_{n+2}  \check H_n.
		\end{align*}
		Finally,  the conclusion of \eqref{P--mathcal-L-delta-s} follow by adding all above related terms, and we finish the proof of the lemma.   
	\end{proof}

	\textbf{ + Third  term: $V(q) =\left((p-1)e_b-1\right)\left((1+i\delta)\Re q -q\right)$.}\\

	%\label{lemma-estimation-P--V).
		\begin{lemma}\label{lemma-estimation-P--V}
			There exists $ s_{10} \ge 1$ such that for all  $s_0 \ge s_{10}$, it holds that  for all $s \in [s_0, \bar s],$  
			\begin{equation}\label{mathscr-Q-Re-delta-P--V-q}
				\mathscr{Q}_{\Re, \delta} \left(  P_-( V (q) ) \right) = 0,
			\end{equation}
			and 	
			\begin{align}
				\left| \mathscr{Q}_{\Im, \delta} \left( P_-( V (q)) \right) - \left(1 - (p-1) e_b \right) \check q_- \right| \le CI^{-\gamma }(s)\left(I^{-M}(s) + |y|^M \right) .\label{mathscr-Q-Im-delta-P--V-q}
			\end{align}
		\end{lemma}
		\begin{proof}
			First, the identity  \eqref{mathscr-Q-Re-delta-P--V-q} follows  by  item (ii)  of  Lemma \ref{lemma-complex-decomposition}. So, it remains  to prove   \eqref{mathscr-Q-Im-delta-P--V-q}. Using item (ii) in that  Lemma again, we have  
			\begin{equation*}
				V(q)  = (1  -  (p-1)e_b)  i  \check{q} = (1 - (p-1)e_b)  i  (\check{q}_+ +  \check{q}_-).    
			\end{equation*}
			Regarding  to \eqref{defi-P-}, we have 
			\begin{align*}
				&P_-(V(q)) -   (1 - (p-1)e_b)  i  \check{q}_- =     (1 - (p-1)e_b)  i  \check{q}_+  + (1 - (p-1)e_b)  i  \check{q}_-  \\
				&   -  \sum_{n \le [M]} Q_n ((1 - (p-1)e_b)  i  \check{q}_+ ) H_n - \sum_{n \le [M]} Q_n ((1 - (p-1)e_b)  i  \check{q}_- ) H_n   \\
				 & -   (1 - (p-1)e_b)  i  \check{q}_-\\
				& =  (1 - (p-1)e_b)  i  \check{q}_+ -  \sum_{n \le [M]} Q_n ((1 - (p-1)e_b)  i  \check{q}_+ ) H_n \\
				&   -   \sum_{n \le [M]}  Q_n ((1 - (p-1)e_b)  i  \check{q}_- ) H_n  =  i \left[ (I) -  (II) \right],
			\end{align*}
			where 
			\begin{align*}
				(I) &=  (1 - (p-1)e_b)    \check{q}_+ -  \sum_{n \le [M]} Q_n ((1 - (p-1)e_b)    \check{q}_+ ) H_n, \\
				(II) &=     \sum_{n \le [M]}  Q_n ((1 - (p-1)e_b)    \check{q}_- ) H_n.
			\end{align*}
			\medskip 
			\noindent 
			\textit{ - Estimate for (I):}   First, we have 
			\begin{align*}
				(1 - (p-1)e_b)	\check q_+ (s) =  \sum_{n \le [M]} \check q_n (1 - (p-1)e_b)H_n (s).
			\end{align*}
			For each $ n  \le [M]$, we  chose  $ L_n  =  \left[   \frac{M - n}{2k}  \right] $, then   we   apply \eqref{identity-e-b}   to derive that 
			\begin{align*}
				(1 - (p-1)e_b)	\check q_+ (s)  = \sum_{n \le [M]} 
				\sum_{j =1}^{L_n}  c_{n,j,b} \check q_n (s) y^{2kj} H_n(s) +   \sum_{n\le [M]}  \tilde c_{n,b} \frac{y^{2k(L_n+1)} }{e_b} \check q_n H_n(s),  
			\end{align*}
			which yields
			\begin{align*}
				(I)  &= \sum_{n\le [M]}  \tilde c_{n,b} \frac{y^{2k(L_n+1)} }{e_b} H_n(s)   - \sum_{n\le [M]} Q_n\left(\sum_{m\le [M]}  \tilde c_{m,b} \frac{y^{2k(L_m+1)} }{e_b} H_m(s) \right)H_n(s)\\
				& =  I_1  - \sum_{n \le [M] } Q_n(I_1) H_n(s), \text{ where } I_1 =   \sum_{n\le [M]}  \tilde c_{n,b} \frac{y^{2k(L_n+1)} }{e_b} H_n(s). 
			\end{align*} 
			We aim to proof that 
			\begin{equation}\label{esti-(I)}
				| (I)  |  \le C (I^{-M}(s) + |y|^M),
			\end{equation}
			it is sufficient to prove 
			\begin{align}
				|I_1|  \le  C   (I^{-M}(s) + |y|^M). \label{estimate-for-I-1-P-minus-V-q}
			\end{align}
			Indeed,  for each $n \le [M]$, we express as follows
			\begin{align*}
				\left|  \frac{y^{2k(L_n+1)} }{e_b} H_n(s) \right| \le  \mathbbm 1_{\{|y| \le 1\}} \left|  \frac{y^{2k(L_n+1)} }{e_b} H_n(s) \right|   + \mathbbm 1_{\{|y| \ge 1\}} \left|  \frac{y^{2k(L_n+1)} }{e_b} H_n(s) \right| 
			\end{align*}
			Since $2k(L_n +1) \ge M-n$, we   estimate
			\begin{align*}
				\mathbbm 1_{\{|y| \le 1\}} \left|  \frac{|y|^{2k(L_n+1)} }{e_b} H_n(s) \right|  
				\le C y^{M-n} (I^{-n}(s) + |y|^n)  \le C  (I^{-M}(s) + |y|^M).
			\end{align*}
			Beside that, it also holds  true  
			\begin{align*}
				\mathbbm 1_{\{|y| \ge 1\}} \left|  \frac{y^{2k(L_n+1)} }{e_b} H_n(s) \right| 
				\le C \frac{|y|^{2k\left(1 +  L_n -\frac{M-n}{2k}\right)}}{e_b}  |y|^{M -n} \left( I^{-n} (s) + |y|^{n}  \right) \le C (I^{-M}(s) + |y|^M),
			\end{align*}
			since $1 +  L_n -\frac{M-n}{2k} \in [0,1]  $ for all $n \le [M]$.  \\
			Thus, taking the sum over  $n$ to the concerning bounds, we get the conclusion of  \eqref{esti-(I)}.\\
			\medskip
			\noindent 
			\textit{ - Estimate for (II):}
			We now notice that 
			\begin{equation}
				\left|	1 - (p-1) e_b(y) \right|  =   \left| \frac{ b y^{2k}}{p-1 + b y^{2k}} \right|   \le Cy^{2k}. 
			\end{equation}
			By  the definition of $Q_n$ in \eqref{defi-Q-n}, we can bound as follows
			\begin{align*}
				\left|  Q_n ((1 - (p-1)e_b)   \check{q}_- )  \right|  \le  CAI^{2n-\gamma}(s)\int_\R y^{2k}|H_n(y,s)|(I^{-M}(s) + |y|^M)\rho_s(y) dy. 
			\end{align*}
			Since 	$|H_n| \le C(I^{-n}(s) + |y|^n)$ and by  changing variable $z =  I(s)  y$,   we estimate the integral  as follows
			\begin{align}
				\left|  Q_n ((1 - (p-1)e_b)  \check{q}_- )  \right|  & \le   CAI^{2n-\gamma-n- M-2k}(s)\int_\R |z|^{2k}|h_n(z)|(1 + |z|^M)e^{-\frac{|z|^2}{4}} dz \nonumber\\
				&\le C A I^{n - M -\gamma - 2k}(s).\label{integral-variable-z}
			\end{align}
			Consequently, 
			$$  \left|(Q_n ((1 - (p-1)e_b)  \check{q}_- )   H_n(y,s) \right| \le C AI^{-\gamma - 2k + n - M }\left(I^{-n} + |y|^n\right) \le C I^{-\gamma - 2k}(s) \left( I^{-M}  + |y|^M \right), $$
			provided that $s \ge s_{10,1}(A)$. Hence, we conclude that
			\begin{equation}\label{esti-sum-Q-n-i-check-q--}
				\left| (II)\right| \le C I^{-\gamma }(s) \left( I^{-M}(s) + |y|^M \right),
			\end{equation}
		provided  that  $s \ge s_{10,1}(A)$.     	Combining \eqref{esti-(I)} with \eqref{esti-sum-Q-n-i-check-q--}, we conclude 
			\begin{equation*}\label{esti-sum-Q-n-H-n-i-check-q}
				\left| \mathscr{Q}_{\Im,\delta} \left(  P_-(V(q)) \right)  - (1 - (p-1)e_b) \check q_- \right|  \le CI^{-\gamma }(s) \left( I^{-M}(s) + |y|^M \right), 
			\end{equation*}
			which  concludes 	\eqref{mathscr-Q-Im-delta-P--V-q}. Finally, the conclusion of the  lemma follows.
		\end{proof}

		\iffalse 
		Let us write
		
		\[V(q)=\left ((p-1)e_b-1\right)\left ((1+i\delta)\Re q -q\right )=-i by^{2k} e_b\left ( \Im q-\delta \Re q\right )\]
		then we can easily see that
		\[\left \|P_-\left (\frac{Vq}{I^{-M}+|y|^M}\right )\right \|_{L^{\infty}}\leq \|V\|_{L^{\infty}}\left \|\frac{q_-}{I^{-M}+|y|^M}\right \|_{L^{\infty}},\]
		where $\|V\|_{L^\infty}\leq 1$
		\fi 
		
		\medskip
		\textbf{+ Fourth term $B(q) =\dsp \frac{b'(s)}{p-1}y^{2k}\left (1+i\delta+(p+i\delta)e_b q\right )$.}
		
		\begin{lemma}\label{lemma-P--B-q}
			For all $s \in [s_0, \bar s]$, it holds that
			$$ \left|P_-(B(q))(s) \right| \le C A|b'(s)| I^{-\gamma}(s) (I^{-M}(s) + |y|^M ). $$
			Consequently,  
			$$	\left|  \mathscr{Q}_{\Re,\delta} \left(P_-(B(q))(s) \right)    \right|  + \left|  \mathscr{Q}_{\Im,\delta} \left(P_-(B(q))(s) \right)    \right| \le  C A|b'(s)| I^{-\gamma}(s) (I^{-M}(s) + |y|^M ).  $$
		\end{lemma}
		\begin{proof}
			From definition  of $P_-$  as in \eqref{defi-P-},  it immediately follows that 
			$$P_-\left((1+i\delta) y^{2k}\right) \equiv 0.$$
			So, we  obtain 
			\[P_-(B(q))=\frac{b'(s)}{p-1}P_-((p+i\delta)y^{2k} e_b q).\]
			According to  Definition \ref{Definition-complex-decomposition},  it is sufficient to check  that 
			\begin{equation}\label{P--p-idelta-y2k-e-b-q}
				\left| P_-((p+i\delta)y^{2k} e_b q) \right| \le C A I^{-\gamma }(s) \left(I^{-M}(s) + |y|^M \right). 
			\end{equation}
			First, let  $\chi \in C_{c}^\infty(\mathbb R)$  satisfying 
			\begin{equation}\label{defi-chi}
				\chi(x) = 1 \text{ for all } |x| \le \frac{1}{2} \text{ and }  \chi(x) = 0 \text{ for all } |x| \ge 1.
			\end{equation}    
			Then, we decompose  $1 =  \chi + (1 - \chi) = \chi + \chi^c$ and   we rely on \eqref{defi-P-} and \eqref{decom-q=q+plus-q-} that 
			\begin{equation*}
				P_-\left((p+i\delta) y^{2k} e_b q \right) = P_-((p+i\delta) y^{2k} \chi e_b q_+) + P_-((p+i\delta) y^{2k} \chi^c e_b q_+) + P_-((p+i\delta) y^{2k} e_b q_-).
			\end{equation*}
			Since $q_- =  (1+ i\delta) \hat q_- + i \check q_-$ and the fact that $(q, b, \theta)(s) \in V_{ A,\gamma, b_0, \theta_0}(s)$, we can bound as follows
			$$  |q_-(y,s)| \le CA I^{-\gamma}(s)(I^{-M}(s) + |y|^M).$$
			Hence, we argue in a similar fashion as in the proof of \eqref{esti-sum-Q-n-i-check-q--} to estimate
			\begin{equation*}
				\left| Q_n((p+i\delta) y^{2k} e_b q_-) H_n \right| \le C I^{-\gamma - 2k }(s) (I^{-M} + |y|^M),
			\end{equation*}
			the, using \eqref{defi-P-}, we conclude that
			\begin{equation}\label{bound-P-p-idelta-y-2k-e-b-q--}
				\left| P_-((p+i\delta) y^{2k} e_b q_-) \right|  \le CA I^{-\gamma}(s) (I^{-M}(s) + |y|^M). 
			\end{equation}
			Additionally, we use  the fact that $(q, b, \theta)(s) \in V_{ A,\gamma, b_0, \theta_0}(s)$  again that
			\begin{equation*}
				\left| \chi^c q_+ \right| \le  C I^{-\gamma}(s) \sum_{n=0}^{[M]} (I^{-n}(s) +|y|^n) \le C I^{-\gamma}(s) (I^{-M}(s) + |y|^M), \text{ since } |y| \ge \frac{1}{2}.  	
			\end{equation*}
			Similarly \eqref{esti-sum-Q-n-i-check-q--}, we have  
			$$ \left|Q_n( y^{2k} e_b \chi^c q_+ ) H_n \right| \le C I^{-\gamma -2k}(s) (I^{-M}(s) + |y|^M).$$
			Accordingly \eqref{defi-P-} and the fact $\| y^{2k} e_b\|_{L^\infty} \le 1$, we conclude that 
			\begin{equation}\label{esti-P--p+idelta-y2kchi-c-q+}
				\left|  P_-( (p+i\delta) y^{2k} \chi^c e_b q_+)  \right| \le CI^{-\gamma}(s) (I^{-M} (s)  + |y|^M).
			\end{equation}
			For  $|y| \le \frac{1}{2}$ and $L \in \N, L \ge 1$ fixed at the  end of the proof. Now,      we  deduce from    \eqref{decomposition-q}     and the identity \eqref{identity-e-b}  that     
			\begin{align*}
				y^{2k} e_b q_+ &= \sum_{\substack{
						j \le L \\
						n \le [M]
				}}    c_{j}(b) y^{2kj} (\hat q_n \widehat{H}_n +  \check q_n \widecheck H_n)  +  \tilde Y\\
				& = \sum_{\substack{
						j \le L \\
						n \le [M]\\
						2kj + n \le [M]
				}}    c_{j}(b) y^{2kj} (\hat q_n \widehat{H}_n +  \check q_n \widecheck H_n) + \sum_{\substack{
						j \le L \\
						n \le [M]\\
						2kj + n \ge [M] +1
				}}    c_{j}(b) y^{2kj} (\hat q_n \widehat{H}_n +  \check q_n \widecheck H_n)   +  \tilde Y\\
				& = Y_1 + Y_2 + \tilde Y,
			\end{align*}
			where $\tilde Y$ satisfies 
			\begin{equation*}
				\left| \tilde Y  \right|     \le C I^{-\gamma}(s) |y|^{(L+1)}, \text{ for all } |y| \le 1.  
			\end{equation*}
			\iffalse 
			Thus,  by the same way in deriving  \eqref{esti-sum-Q-n-i-check-q--},   we  obtain 
			\begin{align*}
				\left| 	P_-\left(  \chi \tilde Y  \right) \right|  &\le C I^{-\gamma }(s) \left( |\chi| y^{(L+1)} + I^{-L-1}(s) \sum_{n \le [M]} ( I^{-n}(s)+ |y|^n)  \right) \\
				& \le C I^{-\gamma} (s) ( I^{-M} + |y|^M),
			\end{align*}
			provided that $ L$ large enough.  Besides that,  by changing  variable $z = y I(s)$, we can prove that
			\begin{equation}\label{integral-H-m}
				\left| 	\int_{|y| \ge \frac{1}{2} }  f H_n \rho_s dy    \right|  \le C(N,n) \|f\|_{L_N^\infty} I^{ -\frac{I(s)}{8}} \text{ with }  s \ge 1 \text{ and  }  n \in  \N,
			\end{equation} 
			where $\|\cdot \|_{L^\infty_N}$ is defined as in \eqref{defi-norm-L-M}. 
			\fi 
			Since $\chi = 1 - \chi^c$ and $ P_- (Y_1)  =0 $, we have then 
			\begin{align*}
				P_- (\chi Y_1) = P_- (Y_1)  -  P_-(\chi^c Y_1) =  - P_-(\chi^c Y_1). 
			\end{align*}
			In the same way for \eqref{esti-P--p+idelta-y2kchi-c-q+}, we obtain
			$$ \left| P_-(\chi^c Y_1 )  \right| \le   C  I^{-\gamma }(s) (I^{-M}(s) + |y|^M), $$
			which yields 
			$$\left|  P_- (\chi Y_1)  \right| \le C I^{-\gamma}(s)(I^{-M}(s) + |y|^M). $$
			Now, by   changing  variable $z = y I(s)$, we can prove that
			\begin{equation}\label{integral-H-m}
				\left| 	\int_{|y| \ge \frac{1}{2} }  f H_n \rho_s dy    \right|  \le C(K,n) \|f\|_{L_N^\infty} e^{ -\frac{I(s)}{8}} \text{ with }  s \ge 1 \text{ and  for some  }  n \in  \N, \text{ and } K  > 0.
			\end{equation} 
			where $\|\cdot \|_{L^\infty_K}$ is similarly defined in \eqref{defi-norm-L-M}.  By applying \eqref{integral-H-m}, we have 
			\begin{align*}
				\left| Q_n(\chi Y_2)   \right|  \le CI^{- \gamma - ([M]+1) -n }(s) , \forall n \le [M],
			\end{align*} 
			since the indices in $Y_2$ always satisfy that $ 2kj + n \le [M] +1$. Hence, we  arrive at
			\begin{align*}
				\left| Q_n(\chi Y_2) H_n  \right|  \le CI^{- \gamma }(s)(I^{-M}(s) + |y|^M), \forall n \le [M],
			\end{align*} 
			In the other hand, we have
			\begin{align*}
				|\chi Y_2| \le C I^{-\gamma}(s)  \left(\sum_{\substack{
						j \le L \\
						n \le [M]\\
						2kj + n \ge [M] +1
				}}    |\chi(y)||y|^{2kj} \left(I^{-n}(s) + |y|^{n}  \right) \right) \le C I^{-\gamma} ( I^{-M}(s) + |y|^M  ).
			\end{align*} 
			So, we have 
			\begin{align*}
				\left|  P_- (\chi Y_2) \right|  \le  CI^{-\gamma}(s) (I^{-M}(s) + |y|^M), 
			\end{align*}
			which concludes 
			\begin{equation}\label{P--p+idelta-chi-e-b-q+}
				\left| P_-((p+i\delta) y^{2k} \chi e_b q_+) \right|  \le C I^{-\gamma} (s) (I^{-M}(s)  + |y|^{M} ). 
			\end{equation}
			Thus, \eqref{P--p-idelta-y2k-e-b-q} follows by \eqref{bound-P-p-idelta-y-2k-e-b-q--},  \eqref{esti-P--p+idelta-y2kchi-c-q+} and  \eqref{P--p+idelta-chi-e-b-q+}. Finally, we  finish the 
			proof of  the lemma.
			\iffalse 
			
			Since our shrinking set $V_{A, \gamma, b_0, \theta_0}$ has the same structure to the one in \cite{DNZArxiv22}. So,  by a  similar proof  as in  \cite[Lemma 5.11]{DNZArxiv22},  we derive 
			\begin{align*}
				\left| P_-((p+i\delta) y^{2k} \chi e_b q) \right|  \le C I^{-\gamma} (s) (I^{-M}(s) + I^{-M}(s) + |y|^{M} )\\
				\left| P_-((p+i\delta) y^{2k} \chi^c e_b q) \right|  \le C I^{-\gamma} (s) (I^{-M}(s) + |y|^{M} ),
			\end{align*}
			which  completely concludes the lemma. 
			\iffalse 
			Using the fact that $q(y,s)\in V_{ A,\gamma, b_0, \theta_0}(s)$, we can derive immediately 
			$|y^{2k} e_b q|\lesssim I^{-\gamma}(s)\left (I^{-M}(s)+|y|^M\right )$, which gives us the same estimation for $I$. We recall that 
			\[P_+((p+i\delta)\chi^c y^{2k} e_b q)=\displaystyle \sum_{j=0}^{[M]}\tilde P_n((p+i\delta)\chi^c y^{2k} e_b q)\tilde H_n+P_n((p+i\delta)\chi^c y^{2k} e_b q) H_n,\]
			then we obtain by Lemma \ref{lemma-P-n-B} 
			\[|II|=\left |P_-((p+i\delta)\chi y^{2k} e_b q)\right |\lessim \lesssim &I^{-\gamma}(s)(I^{-M}(s)+|y|^M). \]
			
			and we conclude that 
			\beqtn
			|P_-(B(q))|\leq C |b'(s)| I^{-\gamma}(s)(I^{-M}(s)+|y|^M).
			\eeqtn
			\fi 
			\fi 
		\end{proof}
		
		\medskip
		\textbf{ + Fifth term $T(q)=-i\theta'(s)(e_b^{-1}+q)=-i\theta'(s)(p-1+by^{2k}+q)$.}
		
		\begin{lemma}\label{P--T-q}
			For all $s \in [s_0, \bar s]$, it holds that
			$$ P_-(T(q)) = - i \theta'(s) q_-(s). $$
			Consequently, 
			\begin{equation*}
				\mathscr{Q}_{\Re, \delta} \left( P_-(T(q)) \right) = - \delta \hat q_- - \check q_-  \text{ and }  \mathscr{Q}_{\Im, \delta} \left( P_-(T(q)) \right) = \delta ( \delta \hat q_- + \check q_-).
			\end{equation*}
		\end{lemma}
		\begin{proof}
			Using the definition of $P_-$ given in \eqref{defi-P-}, we have 
			\begin{equation*}
				P_-(T(q)) = -\theta'(s) P_-(iq) = - i \theta'(s)  q_-= - i \theta'(s) q_-(s).
			\end{equation*}
		Besides that,  we use the definition in \eqref{defi-mathscr-Q-Re-Im}, we find that 
			\begin{align*}
				iq_- = i \left\{  (1 + i \delta) \hat q_- + i \check q_- \right\}  = (1 + i\delta) \left[  - \delta \hat q_- - \check q_-   \right] + i \delta (\delta \hat q_- + \check  q_-)
			\end{align*}
			which yields the complete conclusion of the lemma. 
			\iffalse

			Using the fact that $P_-$ commute with $i$, we obtain
			\[P_-(q)=-i\theta'(s)q_-.\]
			\textbf{Fifth term: $N(q)=(1+i\delta)\left ( |1+e_bq|^{p-1}(1+e_bq)-1-e_b\Re q-\frac{p-1}{2} e_b q-\frac{p-3}{2}e_b \bar q\right )$,}\\
			\fi 
		\end{proof}
		
		\medskip 
		\textbf{ + Sixth term: $N(q)=(1+i\delta)\left ( |1+e_bq|^{p-1}(1+e_bq)-1-2e_b\Re q-\frac{p-1}{2} e_b q-\frac{p-3}{2}e_b \bar q\right )$.}
		We have the following result.
		\begin{lemma}\label{lemma-control-N-q-outer-region}
			There exists $s_{12}(A) \ge 1$ such that  for all $s_0  \ge s_{12}$, and  for all $s \in [s_0, \bar s]$
			it holds that
			\begin{equation*}
				\left| P_-\left(  N(q)(s) \right) \right|  \le C A^{\max(p,2)} I^{- \min(p,2)\gamma}(s) (I^{-M}(s) + |y|^M).
			\end{equation*}
		\end{lemma}
		\begin{proof}
			Let $\chi$ defined as in \eqref{defi-chi}, and we  decompose $N= N(q)$ as follows
			$$ N  = \chi  N + (1 - \chi) N = \chi N + \chi^c N. $$
			It  suffices to verify the following:
			\begin{align}
				\left| P_-(\chi^c N)(s)   \right| &\le C(A^{\max(p,2)}) I^{-\min(p,2)\gamma}(s)\left( I^{-M}(s) + |y|^M \right), \label{esti-P--chi-c-N}  \\
				\left| P_-(\chi N)(s)  \right|    &\le C A^2I^{-2\gamma}(s)\left( I^{-M}(s) + |y|^M \right), \label{esti-P--chi-N}
			\end{align}	
			provided that $s \ge s_0 $ with $s_0 \ge s_{12}(A, M)$.
			\\
			\textit{-  For \eqref{esti-P--chi-c-N}:}    First, let us prove that 
			\begin{equation}\label{chi-c-N-q-le-estimate}
				|\chi^c {N}(q)| \le C A^{\max(2,p)} I^{-\min(2,p)\gamma}(s) (I^{-M}(s) + |y|^M),  p >1.  
			\end{equation}
			The proof of  \eqref{chi-c-N-q-le-estimate} is    divided  into two cases  where  $ p \ge 2$
			and $ p  \in (1,2)$. 
			
			\medskip 
			\noindent
			\textit{- Case 1:  $p \ge 2$.} 		 By a simple expansion, we   estimate
			\begin{align*}
				|\chi^c {N}(q)  | \le C \chi^c (  |e_bq |^2  + |e_b q|^p ). 
			\end{align*}
			Since ${\rm supp}(\chi^c) \subset \{ |y| \ge \frac{1}{2}\}$,   the estimate in 
			\eqref{bound-q-ygep1-2} implies
			\begin{align*}
				| \chi^c e_b q|  \le C A I^{-\gamma}(s) |y |^{M-2k} . 
			\end{align*}
			Notice that $M = \frac{2kp}{p-1}$, then we get 	
			\begin{align}
					\left| \chi^c  e_b q \right|^p  \le C \left( AI^{-\gamma}(s)  |y|^{M-2k} \right)^p    
				 \le CA^{p} I^{-p\gamma}(s) |y|^{\frac{2kp}{p-1}}  \le CA^p I^{-p\gamma}(s) (I^{-M}(s) + |y|^M).\label{esti-chi-c-N-q}
			\end{align}
			Similarly, 
			\begin{align*}
				&	| \chi^c  e_b q |^2  \le   CA^2 I^{-2\gamma}(s) |y|^{2(M - 2k)} \chi^c \le CA^2 I^{-2\gamma}(s) |y|^{\frac{4k }{p-1} } \chi^c\\
				& \le  CA^{2} I^{-2\gamma}(s) |y|^{\frac{2k}{p-1} p}  \chi^c  \le  CA^{2} I^{-2\gamma}(s)( I^{-M}(s) + |y|^M)   \text{ since }  p \ge 2. 
			\end{align*}
			
			Hence,  \eqref{chi-c-N-q-le-estimate}  holds true for all $p \ge 2$.

			\medskip 
			\noindent
			\textit{ -   Case 2 i.e. $p \in (1,2)$:}  By an improvement on the  formula of $N(q)$, we find that
			$$ |\chi^c {N}(q) | \le C \chi^c  |e_b q|^p.$$
			By the same way of \eqref{esti-chi-c-N-q}, we deduce that
			$$   |\chi^c \mathcal{N}(q) |  \le C A^p  I^{-p\gamma} (s) (I^{-M}(s) + |y|^M) . $$
			So, \eqref{chi-c-N-q-le-estimate} also holds true for the case  $ p \in (1,2)$.
			
			\medskip 
			Now, we  use \eqref{chi-c-N-q-le-estimate}  to establish   for all  $n \le [M]$
			\begin{align*}
				&	|Q_n(\chi^c N(q)(s))| \le CA^{\max(p,2)} I^{-\min(p,2) \gamma}(s)\int_{|y| \ge \frac{1}{2}} (I^{-M}(s) +|y|^M)|H_n(y,s)| \rho_s(y) dy\\
				& \le CA^{\max(p,2)} I^{-\min(p,2) \gamma}(s) e^{-\frac{I(s)}{16}},
			\end{align*}
			which yields
			\begin{align*}
				&\sum_{n \le [M]}  \left|Q_n(\chi^c N(s) ) \right| |H_n(y,s)| \le C 	A^{\max(p,2)} I^{-\min(p,2) \gamma}(s) e^{-\frac{I(s)}{16}} \sum_{n \le [M]} (1 + |y|^n) \\
				&\le C I^{-\min(p,2) \gamma}(s) (I^{-M}(s) + |y|^M),
			\end{align*}
			provided that $s \ge s_{12,1}(A, M)$. \\
			Consequently, 
			\begin{align*}
				&|P_-(\chi^c N(s))| \le |\chi^c N(s)| + \sum_{n \le [M]}  \left|Q_n(\chi^c N(s) ) \right| |H_n(y,s)| \\
				&\le CA^{\max(2,p)}I^{-\min(p,2)}(s) ( I^{-M}(s) + |y|^M),
			\end{align*}
			which concludes \eqref{esti-P--chi-c-N}.
			
			\medskip
			\noindent
			\textit{ - For  \eqref{esti-P--chi-N}.}  Since  ${\rm supp}(\chi) \subset \{ |y| \le 1\}$, so it  suffices to consider $|y| \le 1$  and we have 
			$$ |e_b(y) q(y,s)| \le CAI^{-\gamma}(s)  \quad \forall |y| \le 1. $$
			Therefore, by Lemma \ref{lemma-estimation-N}, we have   the following for some  $K \in \N, K \ge 1$
			\begin{equation*}
				\chi N =  \chi \left( N_{K,1} + N_{K,2} + \tilde N_K  \right) , 
			\end{equation*}
			where 
			\begin{align*}
				N_{K,1}  &=   \sum_{ \substack{0 \le j,\ell \le K\\
						2\le j + \ell \le K}	} a_{K,j,\ell} (e_b)^{j +\ell} q_+^j \bar q_+^\ell, \quad a_{K,j,\ell} \in \R,\\
				N_{K,2} & = \sum_{j=2}^{2K} \sum_{\substack{0 \le \ell_1+ \ell_2 \le j-1\\
						\ell_1 \ge 0, \ell_2 \ge 0}} \left( \sum_{\substack{\ell_3 + \ell_4 = j - (\ell_1 + \ell_2 ) \\
						\ell_3 \ge 0, \ell_4 \ge 0}} d_{2,K,j, \ell_1 \ell_2,\ell_3,\ell_4} (e_b)^{j} q_+^{\ell_1} \bar q_+^{\ell_2} q_-^{\ell_3} \bar q_-^{\ell_4} \right),
			\end{align*}
			where $ d_{2,K,j, \ell_1 \ell_2,\ell_3,\ell_4}  \in  \R  $ and $\tilde N_K$ satisfies
			\begin{equation*}
				|\chi \tilde N_K| \le C |\chi e_b q|^{K+1} \le C A^{K+1}I^{-(K +1) \gamma}(s).
			\end{equation*}
			By an analogue to  \eqref{esti-P--chi-c-N}, it leads to 
			\begin{align*}
				|P_-(\chi \tilde N)| \le C I^{-2\gamma}(s)(I^{-M}(s) + |y|^M),
			\end{align*}%proof-N-q-
			provided that $K \ge K_{12,2}( M)$ and $s_{12,2}(K,A)$.  From \eqref{decomp2}, we  have the following decomposition 
			\begin{align*}
				N_{K,1}	&=  \sum_{\substack{ 0 \le  |\textbf{n}| \le K \\
						0 \le |\textbf{m}| \le K\\
						2 \le |\textbf{n}| + |\textbf{m}| \le K  \\
						0 \le \ell \le K}}  c_{\textbf{n}, \textbf{m}, \ell, K} b^\ell y^{2k\ell} \Pi_{j=1}^{[M]} Q_j^{n_j} \bar Q_j^{m_j} H^{n_j + m_j}_j  + N_{K,1,2}\\
				& := N_{K,1,1} + N_{K,1,2}, \text{ respectively},
			\end{align*}
			where we denote $ \textbf{n} = (n_1,...n_{[M]}) $ and $\textbf{m} = (m_1,...,m_{[M]}) $, $|n|=\sum n_i$ and $|m|=\sum m_i$ and  $N_{K,1,2}$ satisfies
			$$ |\chi N_{K,1,2}| \le C A^2I^{-2\gamma}(s) |y|^{2k(K+1)}, \text{ povided that }   s \ge s_{12,3}(A).  $$
			By the same way  to  \eqref{integral-variable-z},  we get the following bound
			$$  |Q_n(\chi N_{K,1,2})| \le CA^2I^{-2\gamma + n -M - 2k(K+1)}(s).$$
			By repeating a similar process as for \eqref{esti-P--chi-c-N},  we obtain
			\begin{equation*}
				\left| 	P_-\left(\chi N_{K,1,2} \right) \right|  \le CA^2I^{-2\gamma}(s) (I^{-M}(s) + |y|^M),
			\end{equation*}
			provided that $K \ge K_{12,3}(M)$ and $s \ge s_{12,4}(A)$.\\
			For $N_{K, 1,1}$, we  decompose as follows
			\begin{align*}
				N_{K, 1,1} &= 	\sum_{\substack{ 0 \le  |\textbf{n}| \le K \\
						0 \le |\textbf{m}| \le K\\
						2 \le |\textbf{n}| + |\textbf{m}| \le K  \\
						0 \le \ell \le K\\
						\sum_{j=1}^{[M]} j(n_j+m_j) + 2k\ell\le [M]	}}  c_{\textbf{n}, \textbf{m}, \ell, K} b^\ell y^{2k\ell} \Pi_{j=1}^{[M]} Q_j^{n_j} \bar Q_j^{m_j} H^{n_j + m_j}_j \\
				& + \sum_{\substack{ 0 \le  |\textbf{n}| \le K \\
						0 \le |\textbf{m}| \le K\\
						2 \le |\textbf{n}| + |\textbf{m}| \le K  \\
						0 \le \ell \le K\\
						\sum_{j=1}^{[M]} j(n_j+m_j) + 2k\ell \ge [M] + 1}}  c_{\textbf{n}, \textbf{m}, \ell, K} b^\ell y^{2k\ell} \Pi_{j=1}^{[M]} Q_j^{n_j} \bar Q_j^{m_j} H^{n_j + m_j}_j\\
				& = N_{K,1,1,1} + N_{K,1,1,2}, \text{ respectively}.
			\end{align*} 
			Since $N_{K,1,1,1}$ is a polynomial in $y$ of degree less or equal to $[M]$, we find that
			$$  P_-(\chi N_{K,1,1,1})  = - P((1- \chi) N_{K,1,1,1}) =  - P_- (\chi^cN_{K,1,1,1} ).$$ 
			In a similar way to \eqref{esti-P--chi-c-N}, we obtain
			\begin{align*}
				\left|P_-(\chi^cN_{K,1,1,1}) \right| \le \left| \chi^c N_{K,1,1,1}\right| + CA^2 I^{-2\gamma}(s) e^{-\frac{I(s)}{16}}\sum_{n \le [M]} (1 + |y|^n) \le CA^2 I^{-2\gamma}(s) (I^{-M} +|y|^M),
			\end{align*}
			provide  that $ s \ge s_{12,5}(A, M)$. \\
			Estimate for $N_{K,1,1,2}$, we firstly have the fact that 
			\begin{align*}
				|\chi N_{K,1,1,2}| &\le  C A^2I^{-2\gamma}(s) \sum_{\substack{ 0 \le  |\textbf{n}| \le K \\
						0 \le |\textbf{m}| \le K\\
						2 \le |\textbf{n}| + |\textbf{m}| \le K  \\
						0 \le \ell \le K\\
						\sum_{j=1}^{[M]} j(n_j+m_j) + 2k\ell \ge [M] + 1}} |\chi(y)||y|^{2k\ell} \left( I^{-\sum_{j=0}^{[M]} j (n_j +m_j)}(s) +  |y|^{\sum_{j=0}^{[M]} j (n_j +m_j)} \right) \\
				& \le  CA^2 I^{-2\gamma}(s) (I^{-M}(s) + |y|^M), \text{ since }  \sum_{j=0}^{[M]} j (n_j +m_j) + 2k\ell \ge [M] +1 \text{ and } |y| \le 1.
			\end{align*}
			Additionally, we  have
			\begin{align*}
				&\left|Q_n(\chi N_{K,1,1,2} ) \right| \\
				&\le  CA^2I^{-2\gamma}(s) \sum_{\substack{ 0 \le  |\textbf{n}| \le K \\
						0 \le |\textbf{m}| \le K\\
						2 \le |\textbf{n}| + |\textbf{m}| \le K  \\
						0 \le \ell \le K\\
						\sum_{j=1}^{[M]} j(n_j+m_j) + 2k\ell \ge [M] + 1}}\int_{\R} |y|^{2k\ell} \left( I^{-\sum_{j=0}^{[M]} j (n_j +m_j)}(s) +  |y|^{\sum_{j=0}^{[M]} j (n_j +m_j)} \right) |H_n(y,s)| \rho_s(y) dy \\
				& \le C A^2I^{-2\gamma + n - [M] -1}(s),
			\end{align*}
			where the last estimate   is obtained by the same technique as in  \eqref{integral-variable-z}  and the fact that    $\sum_{j=0}^{[M]} j (n_j +m_j) + 2k\ell \ge [M] +1$.  \\
			Consequently, 
			\begin{align*}
				&	|P_-(\chi N_{K,1,1,2})| \le |\chi N_{K,1,1,2}| + \sum_{n \le [M]}  |Q_{n}(\chi N_{K,1,1,2})| |H_n| \\
				&\le CA^2  I^{-2\gamma}(s)(I^{-M}(s) +|y|^M) +  CA^2I^{-2\gamma}(s)\sum_{n \le [M]} I^{n -[M] -1}(s) (I^{-n}(s) + |y|^n) \\
				& \le CA^2I^{-2\gamma}(s) \left(I^{-M}(s) + |y|^M \right). 
			\end{align*}
			Combining the established estimates, we conclude that
			\begin{align*}
				& \left|P_-(\chi N) \right| \le  \left|P_-(\chi N_{K,1}) +  \right| + \left|P_-(\chi N_{K,2}) \right| +   \left|  P_-(\chi\tilde N_K) \right|\\
				& \le  \left|P_-(\chi N_{K,1,1,1}) +  \right| + \left|P_-(\chi N_{K,1,1,2}) +  \right|  + \left|P_-(\chi N_{K,1,2}) +  \right| + \left|P_-(\chi N_{K,2}) \right| +   \left|  P_-(\chi\tilde N_K) \right|\\
				& \le CA^2 I^{-2\gamma}(s) ( I^{-M}(s)   +   |y|^M),
			\end{align*}
			which concludes \eqref{esti-P--chi-N}. Finally, the conclusion of the lemma follows. 
		\end{proof}

		%	\medskip 
		%	\textbf{ + Seven term $D_s(q)=-\displaystyle\frac{p+i\delta}{p-1}4kby^{2k-1} I^{-2}(s) e_b\nabla q$.}\\
		%	\begin{lemma}
			%		Let $b_0 > 0$, then  there exists $\delta_{12}(b_0)$ such that for all $ \delta \in (0,\delta_{12})$ there  exists $s_{12}(\delta, b_0)$ such that  for all $s_0 s_{12}$, the following statement holds: Assume that $(q, b, \theta)(s) \in V_{ A,\gamma, b_0, \theta_0}(s)$ for all $s  \in [s_0, \%bar s]$ then we have   for all $s_0, \tau < s \le \bar s$
			%		\begin{align*}
				%			\left|  \mathcal{K}_{0, s,\tau} (P_-(\widecheck{D(q)})(\tau)) \right|_s &\le C e^{-\frac{p}{p-1}(s -\tau)} \left( 1+ \frac{1}{\sqrt{s-\tau}} \right) I^{-1-\gamma}(\tau),\\ 
				%			\left| \mathcal{K}_{s,\tau} (P_-(\widehat{D(q)})(\tau))   \right|_s &\le C^{-\frac{s -\tau}{p-1}} \left( 1 + \frac{1}{\sqrt{s -\tau}} \right)  I^{-1 -\gamma}(\tau),
				%		\end{align*}
			%		where $\widecheck{D(q)} = \mathscr{Q}_{\Im, \delta}(D_\tau (q))$ and  $\widehat{D_\tau(q)} = \mathscr{Q}_{\Re, \delta}(D_\tau (q))$ with  $\mathscr{Q}_{\Im, \delta}$ and  $\mathscr{Q}_{\Re, \delta}$ defined   in \eqref{defi-mathscr-Q-Re-Im}. 
			%	\end{lemma}
		
		%	\begin{proof}
			%		The result of the lemma  follows  direct estimates on the kernels of the semigroup by the integration by parts,  and then the estimates in Lemma \eqref{lemma-estimation-K-q-}
			%	\end{proof}
		
		\medskip 
		\textbf{ + Seven term $D_s(q)=-\displaystyle\frac{p+i\delta}{p-1}4kby^{2k-1} I^{-2}(s) e_b\nabla q$.}\\
		\begin{lemma}
			For all $s  \in [s_0, \bar s]$, and $s,  \tau \in   [s_0, \bar s], s > \tau$, it holds that
			\begin{align*}
				\left|  \mathcal{K}_{0, s,\tau} (P_-(\widecheck{D(q)})(\tau)) \right|_s &\le C A e^{-\frac{p}{p-1}(s -\tau)} \left( 1+ \frac{1}{\sqrt{s-\tau}} \right)  I^{-1-\gamma}(\tau),\\ 
				\left| \mathcal{K}_{s,\tau} (P_-(\widehat{D(q)})(\tau))   \right|_s &\le C A e^{-\frac{s -\tau}{p-1}} \left( 1 + \frac{1}{\sqrt{s -\tau}} \right)  I^{-1 -\gamma}(\tau),
			\end{align*}
			where $\widecheck{D(q)} = \mathscr{Q}_{\Im, \delta}(D_\tau (q))$ and  $\widehat{D_\tau(q)} = \mathscr{Q}_{\Re, \delta}(D_\tau (q))$ with  $\mathscr{Q}_{\Im, \delta}$ and $\mathscr{Q}_{\Re, \delta}$ defined   in \eqref{defi-mathscr-Q-Re-Im}. 
		\end{lemma}
		
		\begin{proof}
			\iffalse	The result of the lemma  follows  direct estimates on the kernels of the semigroup by the integration by parts,  and then the estimates in Lemma \ref{lemma-estimation-K-q-}.  Additionally, \fi First, we observe that the proof of the two estimates are the same. So, it is  sufficient to give the proof to the first one. According to   the definition of  $\mathscr{Q}_{\Im, \delta}$,     we  can write 
			\begin{align*}
				\mathcal{K}_{0,s,\tau}( \widecheck{D_\tau (q)} ) = 
				\mathscr{Q}_{\Im, \delta} \left(  \mathcal{K}_{0, s, \tau} \mathcal{D}_\tau (q) \right). 
			\end{align*}
			{Now, let us estimate $\mathcal{K}_{0, s, \tau} \mathcal{D}_\tau (q)$ that we follow the computation from the proof of 
			 \cite[Lemma 5.13]{DNZCPAA24}. For simplicity, we  only  give the main estimates, and the detail can be  found in that reference. 
			 Let $\textbf{d} = \mathcal{K}_{0, s, \tau} \mathcal{D}_\tau (q)$, and by using  the kernel \eqref{Kernel-Formula}, we express as follows
			 \begin{align*}
			 	\textbf{d} & = -\frac{ (p+i\delta) 4 k b}{p-1}I^{-2}(\tau)\int_\R  \mathcal{K}_{0, s,\tau}(y,z)   e_{b(\tau)}(z) z^{2k-1} \nabla_z q(\tau) dz.\\
			 	& =\dsp 4(p+i\delta)kb(p-1)^{-1}  I({\tau})^{-2} \int 
			 	\mathcal{K}_{0, s,\tau }(y,z)\partial _z\left ( e_b(z) z^{2k-1}\right )q(z,\tau)dz\\
			 	&+ 4(p+i\delta)kb(p-1)^{-1} I({\tau})^{-2} \int 
			 	\partial _z(\mathcal{K}_{0, s,\tau }(y,z)) e_b(z) z^{2k-1}q(z,\tau) dz\\
			 	&= \textbf{d}_1 + \textbf{d}_2.
			 \end{align*}
	We obtain from  the bounds \eqref{esti-rough-q-+},  \eqref{esti-rought-pointwise-q--}, and \eqref{esti-semi-K-0s-sigma-q-sigma} that	 
			 $$ \left| \textbf{d}_1 \right|_s  \le C A e^{-\frac{p}{p-1} (s-\tau) } I^{-2-\gamma}(\tau).$$
	Besides that, we use \eqref{Kernel-Formula} to estimate $\partial _z(\mathcal{K}_{0, s,\tau }(y,z)) $	 and then by \eqref{esti-semi-K-0s-sigma-q-sigma}
			  we  obtain
			 $$ \left|\textbf{d}_2 \right|_s \le  \frac{CA}{\sqrt{s -\tau}} e^{-\frac{p}{p-1}(s -\tau)}I^{-1 -\gamma}(\tau).$$
Therefore, 			
			$$ |\mathcal{K}_{0, s, \tau} (\mathcal{D}_\tau(q))|_s \le CA e^{-\frac{p}{p-1}(s-\tau)}I^{-1-\gamma}(\tau) \left(1 + \frac{1}{\sqrt{s  -\tau}} \right),  $$
		which implies
			$$ \left| \mathscr{Q}_{\Im, \delta} (\mathcal{D}_\tau (q)) \right|_s \le CA e^{-\frac{p}{p-1}(s -\tau)}I^{-1-\gamma}(\tau) \left(1 + \frac{1}{\sqrt{s  -\tau}} \right).$$	
			Thus, the  first estimate of the lemma follows, and  we conclude the proof of  the lemma.}
		\end{proof}
		
		\medskip 
		\textbf{+ Eighth term  $R_s(q)= y^{2k-2}I^{-2}(s)e_b\left ( \alpha_1+\alpha_2y^{2k}e_b+ \left (\alpha_3+\alpha_4 y^{2k}e_b\right )q\right )$.}
		We have the following result.
		\begin{lemma}\label{lemma-P-minus-R-s}
			For all  $s \in [s_0,\bar s]  $, it hold that we then have 
			\begin{equation*}
				\left| P_-\left( \mathcal{R}(q) \right)   \right| \le C I^{-2\gamma}\left( I^{-M}(s) + |y|^M \right).
			\end{equation*}
		\end{lemma}
		\begin{proof}
			The result is  quite the same Lemma as \ref{lemma-P--B-q}. We  kindly refer the reader to check the detail. 
		\end{proof}

		\subsection{Conclusion of the proof of Proposition \ref{proposition-ode}.  }\label{subsection-conclusion-propo-priori-estimate}
		We consider $(q,b, \theta)(s) \in V_{ A,\gamma, b_0, \theta_0}(s),  \forall s \in [s_0, \bar s]$. In addition, we let  $\gamma \le \gamma_{3,1}$ and $ s_0 \ge s_{3,1}$ such that Lemmas \ref{lemma-pn-pas-q}-\ref{lemma-estimation-Pn-V} are valid.\\
		\medskip
		\textit{- Proof of (i) of Proposition \ref{proposition-ode}:}
		First we prove the smallness of the modulation parameter $\theta$ given by (i) of Proposition \ref{proposition-ode}. We project equation \eqref{equa-q} on $\check{H}_0=iH_0$, using the fact that $\check q_0=0$ and Lemma \ref{lemma-estimation-Pn-T} we get
		\[\pa_s \check q_0=0= -\theta'(s)((p-1)+(1+\delta^2)\hat q_0)+ \check P_0(T(q))+\check P_0(N(q))+ \check P_0(D_s(\nabla q))+\check P_0(R_s(q)+V(q)),
		\]
		then, we obtain by estimations given in Lemmas \ref{lemma-P-n-B}, \ref{lemma-estimation-Pn-N}, \ref{lemma-estimation-Pn-Ds}, \ref{lemma-estimation-Pn-Rs} and \ref{lemma-estimation-Pn-V}
		\[|\theta'(s)|\leq CI^{-2\gamma}(s), \forall s \in [s_0,\bar s].\]
		Moreover,  we have 
		$$  \int_{s_0}^s  |\theta'(\tau)| d\tau \le    \tilde CI^{-2\gamma} (s_0),   $$
	which ensures %proof-4.8
	$$     \left| \theta(s)  - \theta_0 \right|  \le \frac{1}{8}, \forall s \in [s_0, \bar s],     $$	
		provided that $s \ge s_0 \ge s_{3,2}$.\\
		\textit{- Proof of (ii) of Proposition \ref{proposition-ode}:}
		We  project equation \eqref{equa-q} on $\hat H_{2k}$ in combining with the vanishing condition $\hat q_{2k} \equiv 0$ for all $s \in [s_0, \bar s]$,	and the results in Lemmas \ref{lemma-pn-pas-q}-\ref{lemma-estimation-Pn-V}, we derive
		\beqtn\label{inequality-b}
		|b'(s)|\leq CI^{-2\gamma}(s).
		\eeqtn
		Besides that, we have $b(s_0) = b_0$, then we derive 
		$$ \left| b(s)  -  b_0 \right|  \le  \int_{s_0}^s |b'(\tau)| d\tau \le C \int_{s_0}^s I^{-2\gamma}(\tau) d\tau,$$	
		which implies 
		$$    \frac{3b_0}{4} \le b(s) \le \frac{5}{4} b_0, \forall s \in [s_0, \bar s],  $$
		provided that $s_0 \ge s_{3,3}(\gamma, b_0)$ large enough. Thus, we get the conclusion of item (ii).\\
		
		\textit{- Proof of (iii) of Proposition \ref{proposition-ode}:}
		
		By Lemmas \ref{lemma-pn-pas-q}-\ref{lemma-estimation-Pn-V}, (i) and item (ii) of Proposition \ref{proposition-ode}, we obtain  for all $ n \in \{0,...,[M]$\}, and for all $s \in [s_0, \bar s]$%proof-4.8.1
		\begin{align*}
		\displaystyle &\left |\pa_s \hat q_n(s)  -\left( 1-\frac{n}{2k} \right)\hat q_n(s)\right |\leq CI^{-2\gamma}(s),\\[0.3cm]
		\displaystyle &\left |\pa_s \check  q_n(s) +\frac{n}{2k}\check q_n(s)\right |\leq CI^{-2\gamma}(s), \text{ if } n \le 2k,\\[0.3cm]
		\displaystyle &\left |\pa_s \check  q_n(s) +  \frac{n}{2k} \check q_n(s)  - \left(  \displaystyle\sum_{  \substack{ j, l \in \N, j\ge 1 \\
				2kj +l =n}  } c_{j,l}(p)b^j(s) \check q_l(s)  \right)   \right|\leq CI^{-2\gamma}(s), \text{ if } n \ge 2k,
		\end{align*}
		which   concludes item (iii) of Proposition \ref{proposition-ode}.\\

		\textit{- Proof of  (iv) of Proposition \ref{proposition-ode}} First, we rely on  equation \eqref{equa-q} and  the decomposition in \eqref{decomp3} to obtain the following system
		\begin{equation*}\label{sys-hat q-chech-q} 
			\left\{  \begin{array}{rcl}
				\partial_s \hat q &=& \mathcal{L}_s \hat q + \mathscr{R}_1,\\
				\partial_s \check{q} &=& \mathcal{L}_{0,s} \check{q} +  \mathscr{V} \check{q} + \mathscr{R}_2 ,
			\end{array}   \right. 
		\end{equation*}
		where $\mathcal{L}_{0,s}$ and $\mathcal{L}_s$ respectively  defined as in  \eqref{defi-mathcal-L-0} and 
		\eqref{defi-mathcal-L-s},  and   
		\begin{align*}
			\mathscr{V} & = 1 - (p-1) e_b,\\
			\mathscr{R}_1  &  =  \mathscr{Q}_{\Re, \delta} \left( b'(s) B(q) + i \theta'(s) T(q) + N(q) + \mathcal{D}_s(\nabla q) + \mathcal{R}_s(q) \right),\\
			\mathscr{R}_2  &  =  \mathscr{Q}_{\Im, \delta} \left( b'(s) B(q) + i \theta'(s) T(q) + N(q) + \mathcal{D}_s(\nabla q) + \mathcal{R}_s(q) \right).
		\end{align*}
		Applying the infinite projection $P_-$ defined as in \eqref{defi-P-}, we get
		\begin{equation*}
			\left\{ \begin{array}{rcl}
				\partial_s \hat{q}_-(s) &= &  \mathcal{L}_s \hat q_-   +   P_- \left(\mathscr{R}_1 \right), \\[0.2cm]
				\partial_s \check q_-(s)  &=&  \mathcal{L}_{0,s} \check q_- + \mathscr{V} \check q_- \\[0.2cm]
				& &  + \left( P_-(\mathscr{V}\check q) - \mathscr{V} \check q_- \right) + P_-(\mathscr{R}_2).
			\end{array} 
			\right.
		\end{equation*}
		In particular, we can write the above system in integral form as follows
		\begin{equation*}
			\left\{ 	\begin{array}{rcl}
				\hat q_- (s)  &  = & \mathcal{K}_{s,\sigma } \hat q_-(\sigma )  + \int_\sigma^s \mathcal{K}_{s,\tau} \left(  P_- \left(\mathscr{R}_1 \right)(\tau)\right)  d\tau, \\[0.2cm]
				\check q_-(s)  & = &   \mathcal{K}_{0, s, \sigma} \check q_-(\sigma)  + \int_{\sigma}^s \mathcal{K}_{0,s, \tau}  \left( \mathscr{V} \check q_- (\tau) \right)d\tau \\[0.2cm]
				& & +  \int_{\sigma}^s \mathcal{K}_{0, s, \tau} \left( P_-(\mathscr{V}\check q) - \mathscr{V} \check q_- +  P_- \left(\mathscr{R}_2 \right)(\tau)\right)      d\tau.
			\end{array}\right.
		\end{equation*}
		Now, we claim the following
		\begin{cl} Let $\bar p = \min (p,2)$, then it  holds that
			\begin{equation}\label{proof-esti-semi-group-mathscrP-1-2}
				\begin{array}{rcl}
					\left| \mathcal{K}_{ s, \tau} \left(  P_- \left(\mathscr{R}_1 \right)(\tau) \right) \right|_s \le & & C e^{-\frac{s - \tau}{p-1}} I^{-\frac{\bar p +1}{2} \gamma}(\tau), \\[0.2cm]
					\left| \mathcal{K}_{0, s, \tau} \left(P_-(\mathscr{V}\check q) - \mathscr{V} \check q_- +  P_- \left(\mathscr{R}_2 \right)(\tau)\right) \right|_s \le & & C e^{-\frac{p}{p-1}(s - \tau)} I^{- \gamma}(\tau),
				\end{array}
			\end{equation}
			provided that $ s_0 \ge s_{13}(A)$. 		
		\end{cl}
		\begin{proof}
			As the estimates involving to $\mathscr{R}_1$ and $\mathscr{R}_2$ are the same, we will just give the proof of the estimate involving to $\mathscr{R}_2$. Indeed,  we use \eqref{mathscr-Q-Im-delta-P--V-q} to obtain
			\begin{equation*}
				\left| \left( P_-(\mathscr{V}\check q) - \mathscr{V} \check q_- \right)(\tau) \right|_\tau \le C I^{-\gamma }(\tau).  
			\end{equation*} 
			Additionally,  the infinite projection $P_-$ commutes with $\mathscr{Q}_{\Re, \delta}$ and $\mathscr{Q}_{\Im, \delta}$.  Hence, we apply \textit{a priori estimates}  established in Lemmas \ref{lemma-P--partial-s-q} - \ref{lemma-P-minus-R-s}   to obtain
			\begin{align*}
				\left|P_- \left( \mathscr{R}_2 \right)(\tau) \right|_\tau \le CI^{- \frac{\min(p,2) +1}{2} \gamma }(\tau), 
			\end{align*} 
			provided that $\tau \ge \sigma \ge s_0 \ge s_{13}(A)$. Thus, we combine with 
			the semigroup estimates in Lemma \ref{lemma-estimation-K-q-} to derive
			\begin{equation*}
				\left| \mathcal{K}_{0,s, \tau}  \left(  P_- \left(\mathscr{R}_2 \right)(\tau)   \right) \right|_\tau \le C e^{-\frac{p}{p-1}(s -\tau)}  I^{-\frac{ \min(p,2) +1}{2}\gamma } (\tau)
			\end{equation*}
			provided that $\gamma \in \left(0,\frac{1}{2} \right)$ and we obtain \eqref{proof-esti-semi-group-mathscrP-1-2}.
		\end{proof}
		
		Now, let us give the proof  of item (iv) of Proposition \ref{proposition-ode}. Taking $|\cdot|_s$-norm  defined in \eqref{defi-norm-q---s} and using Lemma \ref{lemma-estimation-K-q-},  we obtain 
		\begin{align*}
			|\hat q_-(s)|_s  & \le    e^{-\frac{s - \sigma}{p-1}} |\hat q_-(\sigma)|_\sigma  + \int_{\sigma}^s \left| \mathcal{K}_{s, \tau}(\mathscr{P}_1(\tau)) \right|_s d\tau\\
			& \le    e^{-\frac{s - \sigma}{p-1}} |\hat q_-(\sigma)|_\sigma + \int_\sigma^\tau e^{-\frac{s -\tau}{p-1}} I^{-\frac{\bar p +1}{2}\gamma}(\tau) d\tau,\\
			|\check q_-(s)|_s  & \le   e^{-\frac{p}{p-1}(s - \sigma)} |\check q_-(\sigma)|_\sigma + \int_{\sigma}^s \|\mathcal{V}\|_{L^\infty(\R)} |\check q(\tau)|_\tau d\tau   + \int_\sigma^s \left| \mathcal{K}_{s, \tau}(\mathscr{P}_2(\tau)) \right|_s d\tau\\
			& \le  e^{-\frac{p}{p-1}(s - \sigma)} |\check q_-(\sigma)|_\sigma + \int_\sigma^s |\check q_-(\tau)|_\tau d\tau + \int_\sigma^s e^{-\frac{p}{p-1}(s -\tau)} I^{-\gamma}(\tau) d\tau,
		\end{align*}
		since $\|\mathscr{V}\|_{L^\infty}  \le  1$.  By using Gr\"onwall's lemma, we get
		$$ 	|\hat q_-(s)|_s   \le  e^{-\frac{s - \sigma}{p-1}} |\hat q_-(\sigma)|_\sigma +   C \left(  I^{-\frac{\bar p +1}{2}}(s) + e^{-\frac{s-\sigma}{p-1}} I^{-\frac{\bar p +1}{2}}(\sigma) \right), $$
		and 
		$$ |\check q_-(s)|_s   \le  e^{-\frac{s - \sigma}{p-1}} |\check q_-(\sigma)|_\sigma +  C \left(  I^{-\gamma}(s) + e^{-\frac{s-\sigma}{p-1}} I^{-\gamma}(\sigma) \right). $$
		which   concludes  the proof  of  item (iv) and also finish the proof of Proposition \ref{proposition-ode}.
		
		\iffalse
		We notice the fact that the infinite projection $P_-(\cdot)$ actually  commutes with $\mathscr{Q}_{\Re,\delta}$ and $ \mathscr{Q}_{\Im, \delta}$ defined as in \eqref{defi-mathscr-Q-Re-Im}.  Therefore, we apply  Lemmas \ref{lemma-P--partial-s-q} - \ref{lemma-P-minus-R-s} to conclude 
		\fi 	
		
		\section{Spectral gap estimates on semigroups}	
		\label{section-estimation-on-the-semigroup}
		In this section, we  provide spectral gap estimates for  semigroups $\mathcal{L}_{0,s}$ and $\mathcal{L}_s$. More  precisely, the results read.
		\begin{lemma}\label{lemma-estimation-K-q-}  Let us consider $\mathcal{L}_{0,s}$ and $\mathcal{L}_s$ defined as in  \eqref{defi-mathcal-L-0} and \eqref{defi-mathcal-L-s}, and their semigroup be $\mathcal{K}_{0,\tau, \sigma}$ and  $\mathcal{K}_{ \tau, \sigma}$,  respectively. {Then, for each $q_- \in L^\infty_{M}(\R)$ which is orthogonal to $\{ H_n, \text{ for all } 0 \le n \le [M] \} $, we have }
			\begin{align}
				\left| \mathcal{K}_{0, \tau,\sigma} q_- \right|_\tau  & \le   C e^{-\frac{p}{p-1}(\tau-\sigma)} |q_-|_{\sigma},  \label{esti-semi-K-0s-sigma-q-sigma}\\
				\left |\mathcal{K}_{\tau,\sigma} q_-\right|_{\tau} &\leq Ce^{-\frac{1}{p-1}(\tau-\sigma)}\left | q_-\right |_{\sigma},  \tau \ge \sigma, \label{esti-semi-K-s-sigma-q-sigma} 
			\end{align}
			where $|\cdot|_{\sigma} $ is  defined as in  \eqref{defi-norm-q---s}.
		\end{lemma}
		\begin{proof}
			The technique of the proof is based on \cite{BKLcpam94}.  First, we derive from \eqref{Kernel-Formula}   that
			$$ \mathcal{K}_{\tau, \sigma} = e^{\tau - \sigma} \mathcal{K}_{0, \tau, \sigma}.$$
			Then, \eqref{esti-semi-K-s-sigma-q-sigma} is a direct consequence of 	\eqref{esti-semi-K-0s-sigma-q-sigma}. Indeed,  let us assume that \eqref{esti-semi-K-0s-sigma-q-sigma} holds.  Hence, it follows that
			\begin{align*}
				\left| \mathcal{K}_{\tau, \sigma} q_- \right|_\tau  = e^{\tau - \sigma} \left| \mathcal{K}_{0,\tau,\sigma} q_- \right|_\tau \le C e^{\tau -\sigma} e^{ - \frac{p}{p-1}(\tau - \sigma )} |q_-|_\sigma, 
			\end{align*}
			which yields  \eqref{esti-semi-K-s-sigma-q-sigma}. Now, it suffices to prove \eqref{esti-semi-K-0s-sigma-q-sigma}.  Let us define
			\begin{equation}\label{defi-theta-z}
				\Theta(z) = q_- \left(z I^{-1}(\sigma) \right)  \text{ and }  \tilde \Theta(z) =\mathcal{K}_{0, \tau, \sigma } q_-(I^{-1}(\tau) z). 
			\end{equation}
			Using \eqref{Kernel-Formula} again, we obtain 
			\begin{align*}
				\tilde \Theta(z)  & = \mathcal{K}_{0,\tau, \sigma }(q_-)(I^{-1}(\tau) z) = \int_{\R} \mathcal{F}\left( e^{\frac{\tau -\sigma}{2k}} I^{-1}(\tau) z -  y'  \right) q_-(y')dy'\\
				& = \frac{1}{\sqrt{4\pi (1 - e^{-(\tau-\sigma)})}} \int_\R \exp \left( - \frac{(z e^{-\frac{\tau -\sigma}{2} } -z')^2}{4(1 -e^{-(\tau -\sigma)})} \right) q_-(I(\sigma)^{-1}z') dz'  \\
				&= \int_{\R} e^{(\tau-\sigma)\mathcal L}(z,z') \Theta(z') dz',
			\end{align*}
			where $e^{(\tau-\sigma)\mathcal L}(z,z')$ defined by 
			\begin{align*}
				e^{(\tau-\sigma)\mathcal L}(z,z') = \frac{1}{\sqrt{4\pi (1 - e^{-(\tau-\sigma)})}} \exp \left( - \frac{(z e^{-\frac{\tau -\sigma}{2} } -z')^2}{4(1 -e^{-(\tau -\sigma)})} \right).
			\end{align*}
			\\
			\textit{- The case $\tau -\sigma \le 1$:} From \eqref{defi-norm-q---s}, we have 
			$$ |\Theta(z)|  \le  I^{-M}(\sigma) (1 + |z|^M) |q_-|_{\sigma}.   $$
			Now, we  apply  the classical estimate  in \cite[Lemma 4, page 555]{BKnon94}   that we obtain
			\begin{equation*}
				\left| \tilde \Theta(z) \right|  \le C I^{-M}(\sigma) (1 + |z|^M) \left|q_- \right|_\sigma. 
			\end{equation*}  
			Let $z = I (\tau) y$, we obtain 
			\begin{align*}
				\left| \mathcal{K}_{0,\tau, \sigma} q_-(y) \right| \le C I^{-M}(\sigma) I^{M}(\tau) (I^{-M}(\tau) + |y|^M) |q_-|_\sigma, 
			\end{align*}
			which implies 
			\begin{equation*}
				\left| \mathcal{K}_{0,\tau,\sigma}q_- \right|_\tau \le C I^{-M}(\sigma) I^{M}(\tau) |q_-|_\sigma \le C e^{-\frac{[M]+1}{2}(\tau -\sigma)} |q_-|_\sigma, \text{ since } \tau -\sigma  \le 1.
			\end{equation*}
			\\
			\textit{- The case $\tau -\sigma \ge 1$:} We use the following decomposition 
			\begin{equation}\label{decom-tilde-theta-int-N-f}
				\tilde \Theta (z) = \int_{\R} \mathscr{N}(z,z') f(z') dz',
			\end{equation}
			where 
			\begin{align*}
				\mathscr{N}(z,z') = \frac{e^{\frac{(z')^2}{4}}}{\sqrt{4\pi(1 -e^{-(\tau -\sigma)})}}\exp \left( - \frac{(z e^{-\frac{\tau -\sigma}{2} } -z')^2}{4(1 -e^{-(\tau -\sigma)})} \right)  \text{ and } f(z') = e^{-\frac{(z')^2}{4}} \Theta(z',\sigma).
			\end{align*}
			Since 
			\begin{equation*}
				\frac{(ze^{-\frac{\rho}{2}} -z')^2}{1 - e^{-\rho}} - (z')^2 = - z^2  + \frac{(z -z'e^{-\frac{\rho}{2}})^2}{1 - e^{-\rho}},
			\end{equation*}
			we can write 
			\begin{align*}
				\mathscr{N}(z,z') = \frac{e^{\frac{z^2}{4}}}{\sqrt{4\pi(1 -e^{-(\tau -\sigma)})}}\exp \left( - \frac{(z  - z'e^{-\frac{\tau -\sigma}{2} })^2}{4(1 -e^{-(\tau -\sigma)})} \right).
			\end{align*}
			So, we  have
			\begin{equation}\label{estimate-mathscr-N-z-zprime}
				\left|  \partial_{z'}^n \mathscr{N}(z,z')  \right|  \le C e^{-\frac{n(\tau -\sigma)}{2}}(|z| + |z'|)^n e^{\frac{(z')^2}{4}} e^{(\sigma -\tau)\mathcal{L}}(z,z'), \text{ for all } n \ge 0, n \in \N.
			\end{equation}
			Next, let us  define 
			$$	f^{(-m -1)}(z) = \int_{-\infty}^z f^{(-m)} (z') dz'.$$
			From \eqref{defi-theta-z}, we have 
			$$ \int_\R (z')^mf(z') dz' =0 \text{  for all } m \in \{0,1,...[M]\}  \text{ and } {|f(z)| \le I^{-M}(\sigma)(1+ |z|^M) e^{-\frac{(z)^2}{4}}|q_-|_\sigma. }$$
			It is similar to  \cite[Lemma 6, page 557]{BKnon94}, we can estimate 
			\begin{equation}\label{esti-f--m-}
				\left|  f^{(-m)}(z) \right| \le C {e^{-\frac{(z)^2}{4}} I^{-M}(\sigma) (1 +|z|)^{M-m}, }\text{ for all } m \le [M] +1.
			\end{equation}
		{Using integration by part in \eqref{decom-tilde-theta-int-N-f} and estimates \eqref{estimate-mathscr-N-z-zprime} and \eqref{esti-f--m-},  we obtain
			\begin{align*}
				|\tilde \Theta(z)|  &= \left|  \int_\R  \partial_{z'}^{[M]+1} \mathscr{N}(z,z')f^{(-[M]+1)} (z') dz' \right| \\
				& \le  e^{-\frac{[M]+1}{2}(\tau -\sigma)}  \left|q_- \right|_{\sigma}    \int_{\R}   (|z| + |z'|)^{[M]+1} (1 + |z'|)^{M - [M] -1} e^{(\tau -\sigma)\mathcal{L}}(z,z')dz'       \\	
				&\le C e^{-\frac{[M]+1}{2}(\tau -\sigma)} I^{-M}(\sigma)(1 + |z|^M) |q_-|_\sigma \left( 1 + \int_{\R}  \left(\frac{|z|}{1 +|z'|} \right)^{[M]+1 - M}  e^{(\tau-\sigma)\mathcal{L}}(z,z') dz' \right)  .
			\end{align*}
			Now, we aim to prove that
		\begin{equation}\label{proof-int-z-over-z-prime-semigroup}
		\mathcal{I}=\int_{\R} \left(\frac{|z|}{1 +|z'|} \right)^{[M]+1 - M}  e^{(\tau-\sigma)\mathcal{L}}(z,z') dz' \le C e^{\frac{\tau - \sigma}{2}([M]+1 -M)}.
		\end{equation}
	Indeed, we firstly express
	\begin{align*}
	\mathcal{I} & = e^{\frac{\tau- \sigma}{2}([M]+1-M)} \int_{\R}  \left(\frac{|z|}{1 +|z'|} \right)^{[M]+1 - M}  e^{(\tau-\sigma)\mathcal{L}}(z,z') dz'\\
	&\le   e^{\frac{\tau- \sigma}{2}([M]+1-M)} \left\{ \int_{z', |e^{-\frac{\tau-\sigma}{2}} z - z'| \le \frac{1}{8}e^{-\frac{\tau-\sigma}{2}} |z|}\left( \frac{|e^{-\frac{\tau-\sigma}{2}}z|}{1 +|z'|} \right)^{[M]+1 - M}  e^{(\tau-\sigma)\mathcal{L}}(z,z') dz'    \right.\\
	&   +     e^{\frac{\tau- \sigma}{2}([M]+1-M)} \left. \int_{z', |e^{-\frac{\tau-\sigma}{2}} z - z'| \ge \frac{1}{8}e^{-\frac{\tau-\sigma}{2}} |z|}\left( \frac{|e^{-\frac{\tau-\sigma}{2}}z|}{1 +|z'|} \right)^{[M]+1 - M}  e^{(\tau-\sigma)\mathcal{L}}(z,z') dz'    \right\}\\
	& \le C e^{\frac{\tau- \sigma}{2}([M]+1-M)}  \left(   \mathcal{I}_1 + \mathcal{I}_2\right), \text{ respectively}.
	\end{align*}  }		

\medskip 
{+ For $\mathcal{I}_1$: 	Since $|e^{-\frac{\tau-\sigma}{2}} z - z'| \le \frac{1}{8}e^{-\frac{\tau-\sigma}{2}} |z|$, we get 
$$        |z'| \ge   e^{-\frac{\tau-\sigma}{2}} |z|  -  |e^{-\frac{\tau-\sigma}{2}} z - z'|  \ge \frac{7}{8} e^{-\frac{\tau-\sigma}{2}} |z|,$$
which implies
\begin{align*}
	\mathcal{I}_1 \le \tilde C \int_{\R} e^{(\tau -\sigma)\mathcal{L}} (z,z') dz' \le C. 
\end{align*}}

\medskip 
{+ For $\mathcal{I}_2$: We observe that for $z'$ satisfying   $ |e^{-\frac{\tau-\sigma}{2}} z - z'| \ge  \frac{1}{8}e^{-\frac{\tau-\sigma}{2}} |z|,$  and $\tau- \sigma \ge 1 $,  so get 
$$   e^{-\frac{\tau-\sigma}{2}([M] +1 - M)} \sqrt{e^{(\tau-\sigma)\mathcal L}(z,z')}  \le C. $$		
Hence, we find that
\begin{align*}
	\mathcal{I}_2 \le \tilde C \int_\R  \sqrt(e^{(\tau -\sigma)\mathcal{L}(z,z')}) dz' \le C.  
\end{align*}
	Thus, we conclude \eqref{proof-int-z-over-z-prime-semigroup}.  Now, we rely on  \eqref{defi-M} and  the fact $M \le [M] +1$, we have then
			\begin{equation*}
				|\tilde \Theta(z)| \le C e^{\frac{M}{2}(\tau -\sigma)} I^{-M}(\sigma) (1 + |z|^M)  \left| q_-\right|_\sigma = C e^{-\frac{p}{p-1}(\tau -\sigma)}  I^{-M}(\tau) (1 + |z|^M) |q_-|_\sigma, 
			\end{equation*}
			which yields that
			$$ |\mathcal{K}_{0,\tau, \sigma}(q_-)(y)|_\tau  \le   C e^{-\frac{p}{p-1}(\tau -\sigma)} \left|q_- \right|_\sigma.  $$
			Thus, we get the conclusion of  \eqref{esti-semi-K-0s-sigma-q-sigma} for the case $\tau -\sigma   \ge 1$ and finish the proof of the lemma.}
		\end{proof}

		\newpage

		\bibliographystyle{alpha}
		\bibliography{mybib}

	\end{document}